\documentclass[12pt,leqno]{article}

\usepackage{amsthm,showidx}

\usepackage{amsmath}
\usepackage{amscd}
\usepackage{amsmath}
\usepackage{amssymb}
\usepackage{latexsym}
\makeatletter

\@addtoreset{equation}{section}
\makeatother

\newtheorem{thm}{Theorem}[section]
\newtheorem{pr}[thm]{Proposition}
\newtheorem{df}[thm]{Definition}
\newtheorem{lm}[thm]{Lemma}
\newtheorem{cor}[thm]{Corollary}
\newtheorem{cn}[thm]{Conjecture}

\newtheorem{ex}[thm]{Example}
\newcommand{\sm}{\raisebox{2.33pt}{~\rule{6.4pt}{1.3pt}~}}

\input xy \xyoption{all} \CompileMatrices

\begin{document}

\title{Wild ramification and
the cotangent bundle\footnote{
MSC-classification:
14F20, 14G17, 11S15}}
\author{{\sc Takeshi Saito}}
\maketitle

\begin{abstract}
We define the characteristic cycle
of a locally constant \'etale sheaf
on a smooth variety in positive characteristic 
ramified along boundary
as a cycle in the cotangent
bundle of the variety,
at least on a neighborhood of
the generic point of the divisor
on the boundary.
The crucial ingredient in the definition
is an additive structure on the boundary
induced by the groupoid structure
of multiple self products.

We prove a compatibility with pull-back
and local acyclicity 
in non-characteristic situations.
We also give a relation 
with the characteristic cohomology class
under a certain condition
and a concrete example
where the intersection with the $0$-section
computes the Euler-Poincar\'e characteristic.
\end{abstract}

Let $k$ be a perfect field of 
characteristic $p>0$,
$X$ be a smooth scheme
of dimension $d$ over $k$
and $D$ be a divisor of $X$
with simple normal crossings.
Let $\Lambda$ be a 
ring finite over ${\mathbf Z}_\ell[\zeta_p]$
or a finite extension of
${\mathbf Q}_\ell(\zeta_p)$
for a prime $\ell\neq p$
and ${\cal F}$
be a locally constant constructible sheaf of
free $\Lambda$-modules
on the complement $U=X\sm D$.

We formulate a condition
that the ramification of ${\cal F}$
along $D$ is {\em non-degenerate} along $D$
in Definition \ref{dfic}.
The condition roughly says that
the ramification along $D$ is uniformly
controlled by the ramification
at the generic points of irreducible components
and is satisfied
after shrinking $X$ to
a neighborhood of each generic point.
It is a non-logarithmic variant
of that studied
in \cite{mu} and \cite{Tohoku},
which is a generalization of the cleanness
originally introduced by Kato in \cite{Kato}.

Assuming this condition,
we define the characteristic cycle
${\rm Char}({\cal F})$
in Definition \ref{dfccy}
as a $d$-dimensional cycle
on the cotangent bundle $T^*X$.
We show that the characteristic cycle has 
coefficients in ${\mathbf Z}[\frac1p]$ in
Proposition \ref{prHA}.
The author does not know
how to define the characteristic cycle
without the non-degenerate assumption. 

For a morphism $f\colon X'\to X$ of
smooth schemes over $k$,
we define the condition that
$f$ is non-characteristic with
respect to ${\cal F}$ on $X$
in Definition \ref{dfncF}
in terms of ${\rm Char}({\cal F})$.
We prove that
the construction of the characteristic cycles
commutes with the pull-back by
non-characteristic morphisms
in Proposition \ref{prncF}.
This is a special case of
an expected compatibility
of the construction of the characteristic cycle
with pull-back, cf.\ \cite[Proposition 9.4.3]{KS}.
We deduce a characterization of
the support of the characteristic cycle
in terms of the restrictions to
curves transversally meeting
the boundary
in Proposition \ref{prcv}.

The results in this article
are non-logarithmic variants
of the logarithmic version studied 
in \cite{mu} and \cite{Tohoku}
obtained by a different method of localization.
The relation between the two methods
is discussed in Section \ref{slog}.
A significant advantage of the non-logarithmic version
is that it behaves better with
the restriction to curves.
In particular,
Proposition \ref{prcv}
enables us to deduce the local acyclicity,
Proposition \ref{pracy}.

We also define 
the condition for
a smooth morphism $f\colon X\to Y$
of smooth schemes to be non-characteristic with
respect to ${\cal F}$ on $X$
in Definition \ref{dfncf}
in terms of 
${\rm Char}({\cal F})$.
This property is an analog of
that in a transcendental
context \cite[Proposition 8.6.4]{KS}.
We deduce the local acyclicity
for a non-characteristic morphism in Proposition \ref{pracy}
from a result of Deligne-Laumon \cite{DL}.

We give a relation with the characteristic class
defined in \cite{cc} in Corollary \ref{corcc}.
As an application, we give a concrete example
of a computation of the Euler-Poincar\'e 
characteristic in Example \ref{excc}.
This is an extreme case of
an expected compatibility
of the construction of the characteristic cycle
with proper push-forward, cf.\ \cite[Proposition 9.4.2]{KS},
for a morphism to a point.

The crucial ingredient in the definition 
of the characteristic cycle
is an additive structure on
the boundary established in
Corollary \ref{cormain}.
Using a classification of
a vector bundle in characteristic $p>0$
by a finite \'etale group scheme
of ${\mathbf F}_p$-vector spaces
recalled in Section \ref{ssExt},
we relate the additive structure
to the cotangent bundle in Definition \ref{dfcf}.
As an application of this relation,
we study the graded quotients
of the filtration by non-logarithmic
ramification groups of
a local field of equal characteristic
with imperfect residue field
in Section \ref{sslf}.

The additive structure
on the boundary 
is defined as the restriction
of an extension constructed
in Theorem \ref{thmmain}
of the groupoid structure of multiple self products.
Functoriality of multiple self products
is interpreted as a groupoid structure
using an equivalence 
of categories Proposition
\ref{prgpd} stated in terms of an extra structure
on simplicial objects introduced in Section \ref{sssimp}.

More precisely speaking,
the additive structure
is defined on the boundary 
of the largest \'etale
scheme inside the normalizations
of some partial blow-ups,
called dilatations,
of the multiple self products
in ramified coverings.
Some functorial properties of the \'etale part of
normalizations are established in Section
\ref{ssnm}.
We study properties of dilatations
abstractly in Section \ref{ssdl} and more concretely
in Section \ref{ssPn}.

An essential part of this work
was done during the
author's stay at IHES
in September and October
2012.
He thanks the 
hospitality of Ahmed Abbes.
The author also thanks
Pierre Deligne for 
sending him unpublished notes
\cite{bp}
on ramification of sheaves,
Vladimir Drinfeld
for an inspiring comment
on the groupoid structure.
He would like to thank
Luc Illusie
for discussion on the cohomology of 
classifying spaces and local acyclicity
and Alexander Beilinson
for discussion on the
definition of the characteristic cycles.
He thanks to an anonymous referee
for careful reading and numerous helpful comments.
The research was partially supported
by JSPS Grants-in-Aid 
for Scientific Research
(A) 22244001.

\tableofcontents

\section{Preliminaries}

\subsection{Oversimplicial objects
and groupoids}\label{sssimp}

To describe groupoid objects
in a category, we equip
an extra structure on simplicial object
and call it an oversimplicial object.
Although one could restrict to $n\leqq 3$ in the following
description for this purpose,
we do not make this restriction
for esthetic reason.

Recall that a category 
is said to be {\em finite complete}
if finite inverse limits
are representable.

\begin{df}\label{dfss}
Let ${\mathcal C}$
be  a category and let
$\widetilde \Delta$
denote the full subcategory
of the category of sets
consisting of the objects
$[0,n]=\{0,1,\ldots,n\}$ for 
integers $n\geqq  0$.

{\rm 1.}
We call a contravariant functor
$P\colon 
\widetilde \Delta\to {\mathcal C}$
an {\rm oversimplicial object}
of ${\mathcal C}$.
For an oversimplicial object $P$
and an integer $n\geqq  0$,
we write $P_n$ for $P([0,n])$.
For oversimplicial objects
$P$ and $Q$ of ${\mathcal C}$,
we call a morphism
$f\colon P\to Q$ of functors
a morphism of
oversimplicial objects.

{\rm 2.}
We call a cocartesian diagram
\begin{equation}
\begin{CD}
[0,m]@>>> [0,n]\\
@AAA@AAA\\
[0]@>>>[0,l]
\end{CD}
\label{eqadd}
\end{equation}
in $\widetilde \Delta$
an {\rm additive} cocartesian diagram.
We say that an
additive cocartesian diagram
{\rm (\ref{eqadd})} is a {\em standard} one
if the upper horizontal arrow
$[0,m]\to [0,n]$ is the inclusion
and 
if the right vertical arrow
$[0,l]\to [0,n]$ is the addition by $m$.

Assume that ${\cal C}$ is finite complete.
We say that an oversimplicial object
$P$ of ${\cal C}$
is {\rm multiplicative}, if
every additive cocartesian diagram
{\rm (\ref{eqadd})}
defines a cartesian diagram
\begin{equation}
\begin{CD}
P_m@<<< P_n\\
@VV{\qquad\Box }V@VVV\\
P_0@<<<P_l.
\end{CD}
\label{eqmn}
\end{equation}
We say a multiplicative object of
$P$ is {\rm strictly} multiplicative
if the two morphisms
$P_1\to P_0$ are equal.

We say that a 
morphism $Q\to P$ of oversimplicial objects
of ${\cal C}$ is {\rm multiplicative}, if,
for every additive cocartesian diagram
{\rm (\ref{eqadd})},
the diagram
\begin{equation}
\begin{CD}
Q_m\times_{P_m}P_n@<<< Q_n\\
@VVV @VVV\\
P_n@<<<Q_l\times_{P_l}P_n
\end{CD}
\label{eqmnPQ}
\end{equation}
defined by the commutative diagrams
\begin{equation}
\begin{CD}
Q_m@<<< Q_n@>>>Q_l\\
@VVV @VVV@VVV\\
P_m@<<<P_n@>>>P_l
\end{CD}
\label{eqmnPQb}
\end{equation}
induced by {\rm (\ref{eqadd})} is cartesian.
\end{df}

The category $\widetilde \Delta$
contains the subcategory
$\Delta$ with the same underlying set
and increasing morphisms.
Consequently, an oversimplicial object
defines a simplicial object
by restriction.

In the rest of this subsection,
we assume that ${\cal C}$ denotes 
a category that is finite complete.
We will write an oversimplicial
object $P$ as $P_\bullet$
and a morphism of oversimplicial
objects $f$ as $f_\bullet$
in the following.
For morphisms $Q_\bullet\to P_\bullet$ and
$R_\bullet\to P_\bullet$ of 
oversimplicial objects,
the fibered product
$Q_\bullet\times_{P_\bullet}R_\bullet$
defined component wise 
is an oversimplicial objects.

\begin{lm}\label{lmgrpd}
Let $P_\bullet$ be a multiplicative
oversimplicial object.
Then, for
a morphism $f_\bullet\colon 
Q_\bullet\to P_\bullet$
of oversimplicial objects,
the following conditions
{\rm (1)} and {\rm (2)}
are equivalent
if $f_0\colon Q_0\to P_0$
is a monomorphism.

{\rm (1)}
$Q_\bullet$ is multiplicative.

{\rm (2)}
$f_\bullet$ is multiplicative.
\end{lm}

\begin{proof}
For an additive
cocartesian diagram (\ref{eqadd}),
if $P_n\to P_m\times_{P_0}P_l$
is an isomorphism, then the map
$$(Q_m\times_{P_m}P_n)
\times_{P_n}
(Q_l\times_{P_l}P_n)
=
Q_m\times_{P_m}P_n
\times_{P_l}Q_l
\to
Q_m\times_{P_m}
P_m\times_{P_0}P_l
\times_{P_l}Q_l
=
Q_m\times_{P_0}Q_l$$
is also an isomorphism.
If $Q_0\to P_0$
is a monomorphism.,
the last term is
$Q_m\times_{Q_0}Q_l.$
Hence the diagram
(\ref{eqmn}) with $P$ replaced by $Q$
is cartesian if and only if
the diagram
(\ref{eqmnPQ})
is cartesian.
\end{proof}

For a pair of morphisms 
$s,t\colon P\to S$,
we fix notation
for fibered multi-products.
For integers $n\geqq  0$,
we will define 
\begin{equation}
s_n,t_n\colon P_S^{\times n}
=P\times_SP\times_S
\cdots \times_SP\to S
\label{eqPSxn}
\end{equation}
inductively as follows.
For $n=0$, we set
$P_S^{\times 0}=S$ and $s_0=t_0={\rm id}_S$.
For $n=1$, we set
$P_S^{\times 1}=P$ and $s_1=s,t_1=t$.
For $n\geqq  1$,
we define $P_S^{\times n+1}=
P_S^{\times n}\times_SP$
by the cartesian diagram
\begin{equation}
\begin{CD}
P_S^{\times n}@<{{\rm pr}_1}<< P_S^{\times n+1}\\
@V{t_n}VV @VV{{\rm pr}_2}V\\
S@<{s}<<P
\end{CD}
\label{eqPn}
\end{equation}
and set
$s_{n+1}=
s_n\circ {\rm pr}_1,
t_{n+1}=
t\circ {\rm pr}_2.$
For integers $n=l+m$,
a canonical isomorphism
$(l,m)\colon P^{\times n}_S
\to
P^{\times l}_S
\times_S
P^{\times m}_S$
is defined
by induction on $m$.
Here and in the following,
for pairs of two morphisms
$X\to S$ and $Y\to S$
denoted by $s$ and $t$
or their variants,
the fiber product
$X\times_SY$ is taken
with respect to $t\colon X\to S$
and $s\colon Y\to S$
if otherwise is explicitly stated.

\begin{ex}\label{eggpd}
{\rm
Let $X\to S$ be a morphism in ${\mathcal C}$.
Then by setting $P_n=X_S^{\times (n+1)}
=(X\times_SX)_X^{\times n}$,
we obtain a multiplicative oversimplicial object.

Let $G$ be a group.
We set $P_n=G^{\times (n+1)}/\Delta G$
to be the quotient by the diagonal action
for integer $n\geqq  0$.
Then, they define a multiplicative
oversimplicial object $P_\bullet$.

Let $S$ be a scheme
and $E$ be a vector bundle over $S$.
Then by setting $P_n=
{\rm Coker}({\rm diag}\colon E$
$\to E_S^{\times (n+1)})$,
we obtain a strictly multiplicative oversimplicial object.}
\end{ex}

\begin{df}\label{dfgpd}
Let ${\cal C}$ be
a category that is finite complete.
We say
that a 7-ple $(P,S,s,t,e,$ $\mu,\iota)$
is a {\rm groupoid} in ${\cal C}$
if $P,S$ are objects of ${\cal C}$,
$s,t\colon P\to S,
e\colon S\to P,
\mu\colon P^{\times2}_S=P
\times_SP\to P,
\iota \colon P\to P$
are morphisms of ${\cal C}$
such that the following diagrams
are commutative:
$$
\xymatrix{
P\times_SP\times_SP
\ar[r]^{\ \ \ \mu\times{\rm id}}
\ar[d]_{{\rm id}\times \mu}&
P\times_SP
\ar[d]^{\mu}\\
P\times_SP
\ar[r]^{\mu}&P,
}\qquad
\xymatrix{
P\ar[r]^{\!\!\!(e\circ s)\times{\rm id}}
\ar[d]_{{\rm id}\times (e\circ t)}
\ar[dr]^{\rm id}
&P\times_SP
\ar[d]^{\mu}\\
P\times_SP
\ar[r]^{\mu}&P,
}$$
$$\begin{CD}
P@>{s}>>S\\
@V{{\rm id}\times \iota}VV
@VV{e}V\\
P\times_SP
@>{\mu}>>P,
\end{CD}\qquad
\begin{CD}
P@>{\iota\times{\rm id}}>>
P\times_SP\\
@VtVV
@VV{\mu}V\\
S@>{e}>>P.
\end{CD}$$
We say a groupoid
is a {\rm group} if
$s=t$.
\end{df}

We show that a multiplicative oversimplicial object defines a groupoid.
Let $P_\bullet$ be a multiplicative oversimplicial object.
We define 
$s,t\colon P_1\to P_0,
e\colon P_0\to P_1,
\iota\colon P_1\to P_1$
to be the morphisms
defined by the maps
$[0]\to [0,1]$
sending $0$ to $0$ and
to $1$ respectively,
by the unique map
$[0,1]\to [0]$
and by the map
$[0,1]\to [0,1]$
switching $0$ and $1$.
We define
$\mu\colon P_1
\times_{P_0} P_1\to P_1$
to be the composition 
of the inverse of
the isomorphism
$P_2\to P_1\times_{P_0}P_1$
defined by the cartesian
diagram (\ref{eqmn})
corresponding to the 
standard additive cocartesian diagram
for $l=m=1$
with the map
$P_2\to P_1$
defined by the map
$[0,1]=\{0,1\}\to 
[0,2]=\{0,1,2\}$
sending $0$ to $0$
and $1$ to $2$.

\begin{pr}\label{prgpd}
Let ${\cal C}$ be
a category that is finite complete.

{\rm 1.}{\rm (\cite[Proposition 2.2.3]{cxct})}
Let $P_\bullet$ be a multiplicative oversimplicial object.
Then $(P_1,P_0,$ $s,t,e,\mu,\iota)$ defined above
is a groupoid in ${\mathcal C}$.

{\rm 2.}
Let ${\cal G}$
be the category of groupoids
in ${\cal C}$
and ${\cal M}$
be the full subcategory
of the category of
oversimplicial 
objects in ${\cal C}$
consisting of multiplicative objects.
Then, the functor
\begin{equation}
{\cal M}\to {\cal G}
\label{eqGE}
\end{equation}
defined by the construction 
in {\rm 1.} is an equivalence of 
categories.
The functor {\rm (\ref{eqGE})}
induces an equivalence
of categories on
the full subcategories
consisting of
strictly multiplicative
objects of ${\cal C}$
and of groups in ${\cal C}$.
\end{pr}

We say that a multiplicative
oversimplicial object $P_\bullet$
is associated to
the groupoid $P_1$ over $P_0$.
If $P_\bullet$ is associated to 
a groupoid $P$ over $S$,
the groupoid $P$ over $S$
is determined by $P_0,P_1,P_2$
and the morphisms between them.

\begin{proof}
1.
Define an isomorphism 
$P_n\to P_{1P_0}^{\times n}$
inductively to be the composition 
$P_n\to P_{n-1}\times_{P_0}P_1
\to P_{1P_0}^{\times n-1}\times_{P_0}P_1
=P_{1P_0}^{\times n}$
where the first isomorphism
is defined by the standard 
additive cocartesian diagram
for $m=n-1$ and $l=1$.
We identify $P_n$ with $P_{1P_0}^{\times n}$
by this isomorphism.
Then, the commutative diagrams
in Definition \ref{dfgpd} follow
from the commutative diagrams
$$
\xymatrix{
[0,3]&&[0,2]\ar[ll]_{0\mapsto 0,1\mapsto 2,
2\mapsto 3}\\
[0,2]\ar[u]^{0\mapsto 0,1\mapsto 1,
2\mapsto 3}
&&
[0,1]\ar[u]_{0\mapsto 0,1\mapsto 2}
\ar[ll]_{0\mapsto 0,1\mapsto 2},}
\quad
\xymatrix{
[0,1]&&[0,2]\ar[ll]_{0\mapsto 0,1\mapsto 0,
2\mapsto 1}\\
[0,2]\ar[u]^{0\mapsto 0,1\mapsto 1,
2\mapsto 1}
&&
[0,1]\ar[u]_{0\mapsto 0,1\mapsto 2}
\ar[llu]_{\rm id}
\ar[ll]_{0\mapsto 0,1\mapsto 2},}
$$
$$
\xymatrix{
[0,1]&&[0]\ar[ll]_{0\mapsto 0}\\
[0,2]\ar[u]^{0\mapsto 0,1\mapsto 1,
2\mapsto 0}
&&
[0,1]\ar[u]
\ar[ll]_{0\mapsto 0,1\mapsto 2},}
\quad\quad\quad\quad
\xymatrix{
[0,1]&&[0,2]\ar[ll]_{0\mapsto 1,1\mapsto 0,
2\mapsto 1}\\
[0]\ar[u]^{0\mapsto 1}
&&
[0,1]\ar[ll]
\ar[u]_{0\mapsto 0,1\mapsto 2},}
$$
in $\widetilde \Delta$.

2.
We construct a quasi-inverse functor.
Let $(P,S,s,t,e,\mu,\iota)$
be a groupoid in ${\cal C}$.
We put $P_n=P^{\times n}_S$
for integer $n\geqq  0$.
For a map $f\colon [0,m]\to [0,n]$
in $\widetilde \Delta$,
we define a morphism
$f^*\colon P_n\to P_m$ inductively on $m\geqq 0 $.

Assume $m=0$ and
$i\colon [0]\to [0,n]$
be the map sending $0$ to $i\in [0,n]$.
Then, we define 
$i^*\colon P_n\to P_0$
to be the composition
\begin{equation}
\begin{CD}
P_n
@>{\rm can}>>
P_i\times_SP_{n-i}
@>{t_i\times s_{n-i}}>>
S\times_SS
=P_0.
\end{CD}
\label{eqi*}
\end{equation}

Next, we consider the case $m=1$.
We define $\mu_n\colon P_n
\to P_1$ for $n\geqq  0$
inductively as follows.
We set $\mu_0=e,
\mu_1={\rm id}$
and $\mu_2=\mu$.
For $n\geqq  1$,
we define $\mu_{n+1}$
to be the composition
\begin{equation}
\begin{CD}
P_{n+1}= P_n\times_SP_1
@>{\mu_n\times {\rm id}}>> P_1\times_SP_1
@>{\mu}>>P_1.
\end{CD}
\label{eqmun}
\end{equation}
If $0\leqq i\leqq j\leqq n$,
for the map
$(ij)\colon [0,1]\to [0,n]$
sending $0\mapsto i$ and $1\mapsto j$,
we define $(ij)^*$
to be the composition
\begin{equation}
\begin{CD}
P_n@>{\rm can}>>
P_i\times_SP_{j-i}\times_SP_{n-j}
@>{t_i\times\mu_{j-i}\times s_{n-j}}>> 
S\times_SP_1\times_SS
=P_1
\end{CD}
\label{eqij}
\end{equation}
and define $(ji)^*$
to be the composition
$\iota\circ (ij)^*$.

Let $f\colon [0,m+1]\to [0,n]$ be a map.
Let $g\colon [0,m]\to [0,n]$
be the restriction of $f$
and define $h\colon [0,1]\to [0,n]$
by $h(0)=f(m)$ and $h(1)=f(m+1)$.
Then, we define
$f^*\colon P_n\to P_{m+1}$ to be
\begin{equation}
\begin{CD}
P_n@>{g^*\times h^*}>>
P_m\times_{P_0}P_1=P_{m+1}
\end{CD}
\label{eqf*}
\end{equation}

For an integer $i\in [0,m]$
and a morphism $f\colon 
[0,m]\to [0,n]$,
if $f_i\colon [0,i]\to [0,n]$
denotes the restriction of $f$
and $f'_i\colon [0,m-i]\to [0,n]$
be the composition of
$+i\colon [0,m-i]\to [0,m]$
with $f$.
Then,
by induction on $m-i$,
one shows that
$f^*\colon P_n\to P_m$ is the composition 
\begin{equation}
\begin{CD}
P_n@>{f_i^*\times f_i^{\prime*}}>>
P_{i}\times_{P_0}P_{m-i}
@>{{\rm can}^{-1}}>>
P_m.
\end{CD}
\label{eqfgh}
\end{equation}

We show that $P_\bullet$
defined above is a functor.
It suffices to show
that, for morphisms $f\colon[0,l]\to [0,m]$
and $g\colon[0,m]\to [0,n]$,
we have $(g\circ f)^*=f^*\circ g^*$.
By the inductive definition (\ref{eqf*}),
it suffices to consider
the cases $l=0$ and $l=1$.
We show the case $l=0$. 
Let $i\colon [0]\to[0,m]$
be the map sending $0\mapsto i$.
Then, by the decomposition
(\ref{eqfgh})
and by the definition
(\ref{eqi*}),
the composition $i^*\circ g^*
\colon P_n\to P_0$ is the composition
\begin{equation}
\begin{CD}
P_n
@>{\rm can}>>
P_{f(i)}\times_{P_0}P_{n-f(i)}
@>{f_i^*\times f_i^{\prime*}}>>
P_{i}\times_{P_0}P_{m-i}
@>{t_i\times s_{m-i}}>> 
S\times_SS
=P_0.
\end{CD}
\label{eqfi}
\end{equation}
Since 
$t_i\circ f_i^*=t_{f(i)},
s_{m-i}\circ f_i^{\prime*}=s_{n-f(i)}$,
it is equal to
$f(i)^*\colon P_n\to P_0$
as required.

The case $l=1$
is an immediate consequence
of the following elementary fact on
groupoid.
\end{proof}

\begin{lm}\label{lmass}
Let $(P,S,s,t,e,\mu,i)$
be a groupoid in the category of sets
and $x_1,\ldots,x_n$ be
elements of $P$
satisfying $s(x_i)=t(x_{i-1})$
for $i=2,\ldots,n$.
For $0\leqq i\leqq j\leqq n$,
let $x_{ij}$ denote the product
$x_{i+1}\cdots x_j$
and $x_{ji}$ the inverse $x_{ij}^{-1}$.

Then, for a morphism
$f\colon [0,m]\to [0,n]$
and $0\leqq i\leqq j\leqq m$,
we have
$$x_{f(i)f(j)}
=x_{f(i)f(i+1)}\cdots
x_{f(j-1)f(j)}.$$
\end{lm}

\begin{proof}
It follows by induction on
$j-i$.
\end{proof}
\subsection{Normalization and \'etaleness}\label{ssnm}

\begin{lm}\label{lmnsf}
We consider a 
commutative diagram
\begin{equation}
\begin{CD}
V@<<< V'\\
@VVV @VVV\\X@<f<< X'
\end{CD}
\label{eqnsf}
\end{equation}
of morphisms
of normal schemes
where the vertical arrows
are \'etale, separated and of finite type.
Let $Y$ and $Y'$ be
the normalizations of
$X$ and $X'$ in
$V$ and $V'$
and let
$W\subset Y$
and
$W'\subset Y'$
be the largest open subschemes
\'etale over $X$ and over $X'$
respectively.

{\rm 1.}
The map $V'\to V$
is extended uniquely to a map
$g\colon Y'\to Y$.

{\rm 2.}
Assume that the diagram
{\rm (\ref{eqnsf})} is cartesian.
Assume either that 
each component of $X'$ dominates
a component of $X$
or that $V\to X$ is
the composition $V\to U\to X$ of
a finite \'etale map $V\to U$
and an open immersion $U\to X$.
Then, 
we have $g^{-1}(W)=W\times_XX'\subset W'$.

{\rm 3.}
Assume that $f\colon X'\to X$
is smooth.
Then, we have $g(W')
\subset W$.
\end{lm}

\begin{proof}
1. Since the assertion
is local on $X$ and on $X'$,
we may assume
$X={\rm Spec}\ A,
Y={\rm Spec}\ B$
and 
$X'={\rm Spec}\ A'$
are affine and
$A'$ is an integral domain.
Let $K'$ be the fraction field
of $A'$. Then,
$Y'={\rm Spec}\ B'$
for the integral closure
$B'$ of $A'$ in
the \'etale $K'$-algebra
$L'=\Gamma(V'\times_{X'}K',{\cal O})$.
The image of
any element in $B$ in $L'$ is
integral over $A'$
and hence contained in $B'$.

2.
If the diagram
{\rm (\ref{eqnsf})} is cartesian,
$W\times_XX'$ is a normal scheme
containing 
$V'=V\times_XX'$ as an open
subscheme.
Further if each component of $X'$ dominates
a component of $X$
if or
$V\to X$ is 
the composition of
a finite \'etale map
and an open immersion, then
$V'=V\times_XX'$ is dense
in $W\times_XX'$.
Hence $Y'$ contains
$g^{-1}(W)=W\times_XX'$.

3.
Since $f\colon X'\to X$ is assumed smooth,
the base change $Y\times_XX'$ is
normal.
Hence, by replacing
$X$ and $V$ by 
$X'$ and $V\times_XX'$,
it is suffices to show the assertion 
in the case where $X=X'$.
Further we may assume
$X, Y$ and $Y'$ are strictly local
and then it is clear.
\end{proof}

\begin{cor}\label{cornsf}
Let \begin{equation}
\begin{CD}
V_1@<{g_1}<<V@>{g_2}>>V_2\\
@VVV @VVV @VVV\\
X_1@<{f_1}<<X@>{f_2}>>X_2
\end{CD}
\label{eqpr}
\end{equation}
be a commutative diagram
of normal schemes
where the vertical arrows are \'etale, separated and of finite type.
Let $Y,Y_1,Y_2$
be the normalizations of $X,X_1,X_2$
in $V,V_1,V_2$
and 
let $W,W_1,W_2$ be the largest open subschemes
of $Y,Y_1,Y_2$
\'etale over $X,X_1,X_2$ 
respectively.

Assume that
$f_1\colon X\to X_1$
and
$f_2\colon X\to X_2$
are smooth and that
the diagram
\begin{equation}
\begin{CD}
V_1\times_{X_1}X@<<<
V\\
@VV{\qquad\quad\Box }V @VVV\\
X@<<<V_2\times_{X_2}X
\end{CD}
\label{eqUV}
\end{equation}
induced by {\rm (\ref{eqpr})} is cartesian.
Then
the diagram {\rm (\ref{eqUV})}
induces a cartesian diagram
\begin{equation}
\begin{CD}
W_1\times_{X_1}X
@<<<W
\\
@VV{\qquad\quad\Box }V @VVV\\
X@<<<
W_2\times_{X_2}X.
\end{CD}
\label{eqUW}
\end{equation}
\end{cor}

\begin{proof}
By Lemma \ref{lmnsf}.1,
the morphisms
$g_1\colon V\to V_1$ and
$g_2\colon V\to V_2$ induce 
$Y\to Y_1$ and $Y\to Y_2$.
By the assumption that
$f_1$ and $f_2$ are smooth and
by Lemma \ref{lmnsf}.3,
they induce
$W\to W_1$ and $W\to W_2$.
Hence, we obtain a commutative diagram
(\ref{eqUW}).

We define an \'etale scheme $W'$ over $X$
by the cartesian diagram
\begin{equation}
\begin{CD}
W_1\times_{X_1}X
@<<<W'
\\
@VV{\qquad\quad\Box }V @VVV\\
X@<<<
W_2\times_{X_2}X.
\end{CD}
\label{eqW'}
\end{equation}
The cartesian diagram
(\ref{eqUV}) defines
an open immersion
$V\to W'$.
It induces an open immersion
$W'\to W$.
Since the commutative diagram (\ref{eqUW})
and the cartesian diagram (\ref{eqW'})
define the inverse
$W\to W'$,
the assertion follows.
\end{proof}

We prove a key lemma
for the proof of Theorem \ref{thmmain}.

\begin{lm}\label{lmkey}
Let $V_\bullet \to X_\bullet$ be an \'etale, 
separated, of finite type
and multiplicative morphism
of oversimplicial normal schemes.

{\rm 1.}
For integers $n\geqq 0 $,
let $Y_n$ be the normalization of
$X_n$.
Then, the schemes $Y_n$ form 
an oversimplicial normal scheme
$Y_\bullet$
containing
$V_\bullet$
as an oversimplicial subscheme.

{\rm 2.}
For integers $n\geqq 0 $, let
$W_n\subset Y_n$
be the largest open subschemes
\'etale over $X_n$.
Assume that the following conditions
are satisfied:
\begin{itemize}
\item[{\rm (1)}]
For each injection
$[0,m]\to [0,n]$,
the morphism
$X_n\to X_m$ is smooth.
\item[{\rm (2)}]
The morphism
$V_0\to V_1$
induces $W_0\to W_1$.
\end{itemize}
Then, the schemes $W_n$ form an
oversimplicial open subscheme $W_\bullet$
of $Y_\bullet$
and the restriction 
$W_\bullet \to X_\bullet$
of the morphism
$Y_\bullet \to X_\bullet$
is multiplicative.
\end{lm}

\begin{proof}
1. It follows from
Lemma \ref{lmnsf}.1.

2.
We show that
$W_n$ form an
oversimplicial subscheme $W_\bullet$
of $Y_\bullet$.
It suffices to show that
for each map $[0,n]\to [0,m]$,
the morphism
$Y_m\to Y_n$ maps
$W_m$ to $W_n$.
We show this by induction on $n\geqq 0 $.
If $[0,n]\to [0,m]$ is an injection,
it follows from Lemma \ref{lmnsf}.3
by the assumption (1).
Thus the case $n=0$ is proved.
We show the case $n=1$.
It suffices to consider the case
where $[0,1]\to [0,m]$ is not an injection.
In this case, it is uniquely
decomposed as
$[0,1]\to [0]\to [0,m]$.
Hence, it follows from the
case where $n=0$
and the assumption (2).

We consider an additive cocartesian diagram (\ref{eqadd}).
Let
\begin{equation}
\begin{CD}
V_m@<<< V_n@>>>V_l\\
@VVV @VVV @VVV\\
X_m@<<< X_n@>>>X_l
\end{CD}
\label{eqVVVXXX}
\end{equation}
be the induced commutative diagram.
Since $V_\bullet\to X_\bullet$
is assumed multiplicative,
the diagram (\ref{eqVVVXXX})
defines a cartesian diagram
\begin{equation}
\begin{CD}
V_m\times_{X_m}X_n
@<<< V_n\\
@VV{\qquad\quad\Box }V @VVV\\
X_n@<<< V_l\times_{X_l}X_n
\end{CD}
\label{eqVmult}
\end{equation}
By Corollary \ref{cornsf},
it is extended to a cartesian diagram
\begin{equation}
\begin{CD}
W_m\times_{X_m}X_n
@<<< W_n\\
@VV{\qquad\quad\Box }V @VVV\\
X_n@<<< W_l\times_{X_l}X_n
\end{CD}
\label{eqWmult}
\end{equation}
since the lower horizontal arrows
in (\ref{eqVVVXXX}) are smooth
by the assumption (1).

Apply the
above consideration to
the additive cocartesian diagram
\begin{equation}
\begin{CD}
[0,n]@>{\rm inclusion}>> [0,n+1]\\
@A{0\mapsto n}AA@AA{+n}A\\
[0]@>{\rm inclusion}>>[0,1].
\end{CD}
\label{eqn+1}
\end{equation}
Then, to extend
$V_m\to V_{n+1}$ to
$W_m\to W_{n+1}$,
it suffices to extend the commutative diagram
\begin{equation}
\begin{CD}
V_n@<<< V_m@>>>V_1\\
@VVV @VVV@VVV\\
X_n@<<<X_{n+1}@>>> X_1
\end{CD}
\label{eqVkn}
\end{equation}
corresponding to  $V_m\to V_{n+1}$ 
by the cartesian diagram  (\ref{eqVmult}) to a commutative diagram
\begin{equation}
\begin{CD}
W_n@<<< W_m@>>>W_1\\
@VVV @VVV@VVV\\
X_n@<<<X_{n+1}@>>> X_1
\end{CD}
\label{eqWkn}
\end{equation}
Since $V_m\to V_n$
is extended to $W_m\to W_n$
by the induction hypothesis,
the morphism $V_m\to V_{n+1}$
is extended to $W_m\to W_{n+1}$
as required.

By the cartesian diagram (\ref{eqWmult}),
the morphism $W_\bullet\to X_\bullet$
is multiplicative.
\end{proof}

\begin{lm}[{\cite[Lemma 2.7]{Tohoku}}]\label{lmet}
Let $$\begin{CD}
X'@<{\supset}<<U'@<<<V'\\
@VV{\qquad\Box }V @VV{\qquad\Box }V @VVV\\
X@<{\supset}<<U@<<<V
\end{CD}$$
be a cartesian diagram
of normal schemes.
Assume that the vertical arrows
are quasi-finite,
that each component of $X'$ dominates
a component of $X$,
that the left horizontal arrows
are open immersions
and that
the right horizontal arrows
are finite and \'etale.
Let $X\gets Y$ and $X'\gets Y'$ be the normalizations
in $V$ and $V'$
respectively.

Let
\begin{equation}
\begin{CD}
Z'@>{\subset}>> X'\\
@VVV @VVV\\
Z@>{\subset}>> X
\end{CD}
\label{eqXZ}
\end{equation}
be a commutative diagram
where the horizontal arrows
are closed immersions.
Assume that it is lifted to a commutative diagram
$$\begin{CD}
Z'@>{s'}>> Y'\\
@VVV @VVV\\
Z@>{s}>> Y.
\end{CD}$$
Then, $Y\to X$ is \'etale
on a neighborhood of $s(Z)$,
if $Y'\to X'$ is \'etale
on a neighborhood of $s'(Z')$
and if the following conditions
are satisfied:

{\rm (1)}
$Z'\to Z$ is surjective.

{\rm (2)}
The base change of
the diagram {\rm (\ref{eqXZ})}
by $U\to X$
is cartesian.

{\rm (3)}
$Z\cap U$ is dense in $Z$.
\end{lm}

For the reader's convenience,
we include the proof
in \cite[Lemma 2.7]{Tohoku}.

\begin{proof}
Let $x$ be a point of $Z$.
We show that
$Y\to X$ is \'etale at $s(x)$.
Let $x'$ be a point of $Z'$ above $x$
and take a geometric point $\bar x$
of $Z$ above $x$
and a geometric point $\bar x'$
of $Z'$ above $x'$ and $\bar x$.
By replacing $X,Y,X'$
by their strict localizations
at $\bar x, s(\bar x)$ and at $\bar x'$,
we may assume $X,Y,X'$
are strictly local and $X'\to X$ is finite
and surjective.
It suffices to prove that
$Y\to X$ is an isomorphism.

Since $Y'$ is the normalization of
$X'\times_XY$,
it suffices to show that
$Y'\to X'$ is an isomorphism.
Since $Y'$ is the disjoint union
of finitely many connected components,
it suffices to show the following.

(a) There exists a unique component
meeting $s'(Z')$.

(b) A component meeting $s'(Z')$
is isomorphic to $X'$.

(c) Every component meets $s'(Z')$.

Since $Z'$ is a closed subscheme of
$X'$, its image $s'(Z')$ is also a local scheme and
(a) follows.
Since the component of $Y'$ 
containing $s'(x')$ is
the strict localization 
at $s'(\bar x')$,
it is isomorphic to $X'$.
Hence (b) is proved.
We show (c).
By the assumption (2), 
$s(Z\cap U)\times_YY'$
is isomorphic to
$(Z\cap U)\times_XX'
=
Z'\cap U'$
and hence is equal to
$s'(Z'\cap U')$.
By the assumption (3),
$s(Z\cap U)\subset Y$ is not empty.
Let $Y'_1$ be a component of $Y'$. 
Since $Y'_1\to Y$ is finite and dominant,
it is surjective
and hence
$s(Z\cap U)\times_YY'_1
=s'(Z'\cap U')\cap Y'_1\subset
s'(Z')\cap Y'_1$
is not empty.
Thus (c) is proved.
\end{proof}

The following example shows
that we can't drop the
assumption that $X'\to X$ is
quasi-finite even if $U'=U$.

\begin{ex}
{\rm
Let $k$ be an algebraically closed
field of characteristic $p>0$,
$E$ be an elliptic curve over $k$
and ${\cal L}$ be
a very ample invertible ${\cal O}_E$-module e.g.\ 
${\cal O}(3\cdot[0])$.
Let $X={\rm Spec}\ 
\bigoplus \Gamma(E,{\cal L}^{\otimes n})$
be the affine cone.
The blow-up $X'$ of $X$ at
the origin is the line bundle $L$
over $E$ defined
by the symmetric algebra
$\bigoplus_{n\geqq 0 } {\cal L}^{\otimes n}$.
The complement
$U\subset X$
of the origin is the complement
$L\sm E$
of the $0$-section.
Define a finite \'etale morphism $V\to U$ to be
the pull-back of 
an \'etale isogeny $E'\to E$ of degree $>1$.

The normalization $Y'$ of $X'$ in $V$
is the pull-back $E'\times_EX'$
and is \'etale over $X'$.
The canonical map
$Y'\to Y$ 
to the normalization $Y$ of $X$ in $V$
contracts the closed fiber $E'$
to a point and
the morphism $Y\to X$ is not \'etale.

The inverse image $Z$ of
the origin $0\in E$
is a fiber of the line bundle
$X'\to E$ and is
isomorphic to ${\mathbf A}^1_k$.
The composition
$Z\to X'\to X$
is a closed immersion
since ${\cal L}$ is base point free.
The inverse image of
$0\in E'$
defines closed immersions
$Z\to Y'$ and $Z\to Y$
lifting
$Z\to X'$ and $Z\to X$.
They satisfy the conditions
{\rm (1)--(3)} in Lemma {\rm\ref{lmet}}.}
\end{ex}

\subsection{Dilatations}\label{ssdl}

\begin{df}\label{dfdl}
Let 
\begin{equation}
\begin{CD}
D@>>>P@<<<X
\end{CD}
\label{eqPDX}
\end{equation}
be closed immersions
of schemes.
Assume that
$D$ is a Cartier divisor of $P$ and
let $P'\to P$ denote the blow-up at the intersection
$D\cap X=D\times_P X$.
We call the complement
$$\widetilde P=P^{(D\cdot X)}\subset P'$$
of the proper transform of $D$
the {\rm dilatation} of
$P$ with respect to
$D$ and $X$.
\end{df}

If $P={\rm Spec}\ A$,
if the Cartier divisor $D$
is defined by a non-zero divisor
$t\in A$
and if the closed subscheme
$X\subset P$
is defined by an ideal $I\subset A$,
we have
$P^{(D\cdot X)}={\rm Spec}\ \widetilde A$
for the subring
$$\widetilde A=A\left[\dfrac It\right]
\subset A\left[\dfrac 1t\right].$$

The closed subscheme 
$\widetilde D$ of $\widetilde P$
defined by the cartesian diagram
$$\begin{CD}
\widetilde D@>>>
\widetilde P&\ =P^{(D\cdot X)}\\
@VV{\quad\ \Box }V@VVV\\
D@>>> P&\longleftarrow X.
\end{CD}$$ 
is a Cartier divisor
and is a scheme over $D\cap X
=D\times_PX$.
The canonical map
$\widetilde P\to P$
induces an isomorphism
$\widetilde P\sm\widetilde D
\to P\sm D$.
If $X=P$,
the canonical map
$\widetilde P\to P$
is an isomorphism.

\begin{ex}\label{exdnc}
{\rm 1.
Let $Y$ be a closed subscheme
of a scheme $X$.
Set $P={\mathbf A}^1_X$
and regard the 0-section
$X\to P$ 
as a Cartier divisor $D$ of $P$.
Then, the dilatation
$P^{(D\cdot {\mathbf A}^1_Y)}$
is the deformation to the normal cone of
$Y$ in $X$.

2.
Let $D$ be a Cartier divisor of $X$
and $Z$ be the 0-section of
$P={\mathbf A}^1_X$.
Then, the dilatation 
$P^{({\mathbf A}^1_D\cdot Z)}$
is the line bundle $L(D)$
over $X$ defined by
the symmetric ${\cal O}_X$-algebra
$\bigoplus_{n\geqq  0}
{\cal I}_D^{n}$.}
\end{ex}

The dilatation has the following universality,
as an immediate consequence of
that of blow-up.
Let (\ref{eqPDX}) and $\widetilde P=
P^{(D\cdot X)}$ be as in Definition \ref{dfdl}.
Let ${\cal I}_D\subset {\cal O}_P$
and ${\cal I}_X\subset {\cal O}_P$
be the ideals defining 
$D$ and $X$ respectively.
Then, for a scheme $T$ over $P$,
if the pull-back
${\cal I}_D{\cal O}_T$
is an invertible ideal of ${\cal O}_T$ 
and contains
${\cal I}_X{\cal O}_T$,
then there exists a unique morphism
$T\to \widetilde P$ of
schemes over $P$.

For example, if $D\cap X=D_X$
is a divisor of $X$,
then the immersion $X\to P$
is uniquely lifted to $X\to \widetilde P=
P^{(D\cdot X)}$
by the universality
of dilatation.
If $D=D_1+D_2$ is the sum
of Cartier divisors,
then we have a canonical morphism
$\widetilde P=
P^{(D\cdot X)}
\to 
\widetilde P_1=
P^{(D_1\cdot X)}$
by the universality.
Further assume that
$D_1\cap X$
is a divisor of $X$,
regard $X$ as a closed
subscheme of $\widetilde P_1$
and let $\widetilde D_2$
denote the pull-back of $D_2$ to
$\widetilde P_1$. Then, we have
a canonical isomorphism
$\widetilde P
\to 
\widetilde P_1^{(\widetilde D_2\cdot X)}$.

We study the structure of
dilatations.
For a regular immersion $X\to P$ of
schemes (\cite[D\'efinition 1.4]{Be}
for the definition),
the normal bundle $T_XP$
is the vector bundle over $X$ defined
by the symmetric algebra
$S^\bullet {\cal N}_{X/P}$
where the conormal sheaf
${\cal N}_{X/P}={\cal I}_X/{\cal I}_X^2$
is defined by
the ideal sheaf ${\cal I}_X\subset
{\cal O}_P$.
If $X$ and $P$ are smooth over a scheme
$S$ and if 
the immersion $X\to P$
is a morphism over $S$,
it fits in an exact sequence
$$0\to TX\to TP\times_PX\to T_XP\to0$$
of vector bundles on $X$
where $TX$ and $TP$ denote
the tangent bundles defined
by the symmetric algebras
of the locally free modules
$\Omega^1_{X/S}$ and
$\Omega^1_{P/S}$ respectively.

For vector bundles $E$ and $E'$
on a scheme $S$,
let $E\otimes E'$ denote
the tensor product.
For a line bundle $L$
on a scheme $S$
and an integer $n$,
let $L^{\otimes n}$
denote the $n$-th power.
For a vector bundle $E$ on 
a scheme $S$ and a Cartier divisor $D$ of $S$,
let $E(D)=E\otimes L(D)$ 
denote the tensor product
with the line bundle $L(D)$ on $S$.
We have $T_DS=L(D)\times_SD$.

\begin{lm}\label{lmdl}
We consider a cartesian diagram
$$\begin{CD}
\widetilde D@>>>
\widetilde P&\ =P^{(D\cdot X)}\\
@VV{\quad\ \Box }V@VVV\\
D@>>> P&\longleftarrow X\end{CD}$$ 
where the arrows in the
lower line are closed immersions
as in {\rm (\ref{eqPDX})}.
Assume that
$X\to P$ is a regular immersion
and that
$D\cap X=D_X$ is a Cartier divisor of $X$.

{\rm 1.}
We have a 
canonical isomorphism 
\begin{equation}
\widetilde D
\to T_XP(-D_X)\times_XD_X
\label{eqPD}
\end{equation}
on $D_X$.

{\rm 2.}
The immersion $X\to \widetilde P$
is also a regular immersion and
there is a canonical isomorphism
\begin{equation}
T_X\widetilde P
\to T_XP(-D_X)
\label{eqnb}
\end{equation}
for the normal bundle.
\end{lm}

\begin{proof}
1.
First, we prove the assertion locally on $P$.
We assume that
$P={\rm Spec}\ A$,
that
$t_0,t_1,\ldots,t_m$
is a regular sequence of
$A$ and that
$D$ and $X$ 
are defined by the ideals
$(t_0)$ and 
$(t_1,\ldots,t_m)$
respectively.
Then, we have
$P^{(D\cdot X)}
={\rm Spec}\ \widetilde A$ for 
$\widetilde A=A[T_1,\ldots,T_m]/
(t_0T_1-t_1,
\cdots,t_0T_m-t_m)$.
Hence, we have
an isomorphism
\begin{equation}
A/(t_0,t_1,\ldots,t_m)
[T_1,\ldots,T_m]
\to
\widetilde A/(t_0).
\label{eqA'}
\end{equation}
Since $D_X=
{\rm Spec}\ A/(t_0,t_1,\ldots,t_m)$
and $t_0$ and $t_1,\ldots,t_m$
define linear coordinates of 
$T_DP$ and of $T_XP$
respectively,
the isomorphism (\ref{eqA'})
defines an isomorphism
(\ref{eqPD}).
It is easily checked that
the isomorphism
(\ref{eqPD}) thus defined
is independent of the choices
and defined globally.

2.
In the description in the proof of 1,
$X\subset \widetilde P$
is defined by the ideal of $\widetilde A$
generated by the
regular sequence $T_1,\ldots,T_m$.
Hence, 
we obtain an isomorphism
(\ref{eqnb}).
It is easily checked that
the isomorphism
(\ref{eqnb}) thus defined
is independent of the choices
and defined globally.
\end{proof}

\begin{cor}\label{cordl}
Assume that $P$ and $X$ are regular and
that $D\subset P$
is a divisor with normal crossings
meeting $X$ transversally.
Then $\widetilde P=P^{(D\cdot X)}$ is also regular and the pull-back
$\widetilde D=D\times_P\widetilde P$
is a divisor with normal crossings.
\end{cor}

\begin{proof}
Since the assertion is \'etale local on $P$,
we may assume that
$D$ has simple normal crossings.
Let $D_1,\ldots,D_h$ be
the irreducible components of $D$.
Then, 
by the isomorphism
(\ref{eqPD}),
the inverse images 
$\widetilde D_i=\widetilde D
\times_DD_i$ are
regular divisors of $\widetilde P$
meeting transversally. 
Since the complement
$\widetilde P\sm 
\widetilde D$ is 
isomorphic to $P\sm D$
and is regular, the scheme
$\widetilde P$ is regular.
\end{proof}

We consider the functoriality
of dilatations.
Let
\begin{equation}
\begin{CD}
E@>>> Q@<<< Y\\
@VV{\quad\ \Box }V @VfVV @VVV\\
D@>>> P@<<< X
\end{CD}
\label{eqdlf}
\end{equation}
be a commutative diagram 
of schemes satisfying 
the following properties:
The horizontal arrows
are closed immersions and
 $D$ and $E$ are Cartier divisors
of $P$ and $Q$ respectively.
Further, the left square is cartesian.
Then, the map $Q\to P$
is uniquely lifted
to a morphism
$\tilde f\colon
\widetilde Q=Q^{(E\cdot Y)}\to 
\widetilde P=P^{(D\cdot X)}$
by the universality of dilatations.

\begin{lm}\label{lmdlf}
We consider the commutative diagram
{\rm (\ref{eqdlf})}
of schemes satisfying the conditions
following it. Assume that $f\colon Q\to P$ is flat.

{\rm 1.}
If the diagram
{\rm (\ref{eqdlf})}
is cartesian,
the diagram
\begin{equation}
\begin{CD}
Q@<<<\widetilde Q&=Q^{(E\cdot Y)}\\
@VV{\quad\ \Box }V @VVV&\\
P@<<<\widetilde P&=P^{(D\cdot X)}
\end{CD}
\label{eqPPQ}
\end{equation}
is cartesian and the vertical arrows are flat.

{\rm 2.}
Assume that $D\cap X=D_X$ 
and $E\cap Y=E_Y$ are Cartier divisors
of $X$
and of $Y$ respectively and 
that $E_Y\to D_X$ is flat.
If $X\to P$ and $Y\to Q\times_PX$ 
are regular immersions, the morphism 
$\tilde f\colon
\widetilde Q=Q^{(E\cdot Y)}\to
\widetilde P=P^{(D\cdot X)}$ is flat.
\end{lm}

\begin{proof}
1.
Since the question is local,
we may assume $P={\rm Spec}\ A$ and
$Q={\rm Spec}\ B$ are affine,
the divisor $D$ 
is defined by a non-zero divisor $t\in A$
and 
$X$ is defined by an ideal
$I\subset A$.
Then, $\widetilde P={\rm Spec}\ A[\frac It]$
and 
$\widetilde Q={\rm Spec}\ B[\frac {IB}t]$.
Since $A\to B$ is flat,
the injection
$A[\frac It]
\to A[\frac 1t]$ induces an injection
$B\otimes_AA[\frac It]
\to B\otimes_AA[\frac 1t]
=B[\frac 1t]$
and the assertion follows.

2.
The restriction
$\widetilde Q
\sm \widetilde E
\to
\widetilde P\sm
\widetilde D$
to the complements
is flat by the assumption.
Since the pull-backs $\widetilde D=
\widetilde P\times_PD
\subset \widetilde P$
and $\widetilde E=
\widetilde Q\times_QE
\subset \widetilde Q$
are Cartier divisors
and
$\tilde f^*\widetilde D
=\widetilde E$,
it suffices to show that
$\widetilde E
\to
\widetilde D$ is flat by
\cite[Proposition (15.1.21)]{EGA4-1}.

Since $Q\to P$ is flat,
the composition $Y\to 
Q\times_PX\to Q$
is a regular immersion.
The isomorphisms (\ref{eqPD})
for 
$\widetilde D$ and
$\widetilde E$ are functorial
and make a commutative diagram
\begin{equation}
\begin{CD}
\widetilde E@>>>
T_YQ(-E_Y)
\times_YE_Y
\\
@VVV@VVV\\
\widetilde D@>>>
T_XP(-D_X)
\times_XD_X
\end{CD}
\label{eqXYDE}
\end{equation}
Since $Q\to P$ is flat and
the immersions
$Y\to Q\times_PX$ 
and $X\to P$ are regular
immersions,
the linear map
$T_YQ
\to T_XP\times_XY$
of vector bundles
is a surjection.
By $f^*D=E$,
we have a canonical isomorphism
$L(-E_Y)
\to L(-D_X)\times_XY$.
Since $E_Y\to D_X$
is flat,
the right vertical arrow is flat as required.
\end{proof}

\begin{cor}\label{cordlf}
{\rm 1.}
We consider the commutative diagram
{\rm (\ref{eqdlf})}
of schemes satisfying the conditions loc.\ cit. 
Assume that $D\cap X=D_X$ 
and $E\cap Y=E_Y$ are Cartier divisors
of $X$
and of $Y$ respectively
and that the vertical arrows
and $E_Y\to D_X$
are smooth.
If $X\to P$ is a regular immersion,
the morphism 
$Q^{(E\cdot Y)} \to
P^{(D\cdot X)}$
is smooth.

{\rm 2.}
We consider a cartesian diagram
\begin{equation}
\begin{CD}
P_1@<<< P_3\\
@VV{\qquad \Box }V @VVV\\
S@<<<P_2
\end{CD}
\label{eqdlfp2}
\end{equation}
of flat separated morphisms of schemes.
Let $X_1\subset  P_1$ and 
$X_2\subset  P_2$ be closed subschemes
and define a closed subscheme $X_3=
X_1\times_SX_2\subset P_3$.
Let $D$ be a Cartier divisor of $S$
and $D_1,D_2,D_3$
be the pull-back to
$P_1,P_2,P_3$.
Assume $X_1$ is flat over $S$
and the immersion
$X_1\to P_1$ is a regular immersion.
Then, the diagram {\rm (\ref{eqdlfp2})}
induces a cartesian diagram
\begin{equation}
\begin{CD}
P_1^{(D_1\cdot X_1)}@<<< 
P_3^{(D_3\cdot X_3)}\\
@VV{\qquad\ \Box }V @VVV\\
S@<<<P_2^{(D_2\cdot X_2)}
\end{CD}
\label{eqdlfp}
\end{equation}
\end{cor}

\begin{proof}
1. Since $Y\to X$
and $Q\to P$ are smooth,
the immersion
$Y\to Q\times_PX$
is a regular immersion.
Hence, by Lemma \ref{lmdlf}.2,
the morphism $\widetilde Q\to \widetilde P$
is flat.
Further since 
$\tilde f^*\widetilde D
=\widetilde E$,
it suffices to show
that $\widetilde E\to \widetilde D$
is smooth.
Since
$E_Y\to D_X$
is smooth,
the right vertical arrow in
(\ref{eqXYDE}) is smooth.

2.
Since the assertion is local,
we may assume that
$S={\rm Spec}\ A$ and
$P_i={\rm Spec}\ A_i$
are affine,
$D$ is defined by a non-zero divisor
$t\in A$
and the closed subschemes
$X_i\subset P_i$
are defined by ideals $I_i\subset A_i$.
By Lemma \ref{lmdlf}.2,
the dilatation
$\widetilde P_1
=P_1^{(D_1\cdot X_1)}$ is flat over
$S=S^{(D\cdot S)}$.
Therefore, the injection
$A_2[\frac {I_2}t]
\to
A_2[\frac 1t]$
induces
an injection
$A_1[\frac {I_1}t]
\otimes_AA_2[\frac {I_2}t]
\to
A_1[\frac {I_1}t]
\otimes_AA_2[\frac 1t]
=
A_1
\otimes_AA_2[\frac 1t]$.
Hence  the isomorphism
$A_1\otimes_AA_2\to A_3$
induces an isomorphism
$A_1[\frac {I_1}t]
\otimes_AA_2[\frac {I_2}t]
\to
A_3[\frac {I_1A_3+I_2A_3}t]
=A_3[\frac {I_3}t]
\subset
A_3[\frac 1t]$.
\end{proof}

We give a construction similar to the 
deformation to normal cone in Example \ref{exdnc}.1.
Let $X$ be a scheme,
$D$ be a Cartier divisor and $m\geqq  1$ be
an integer.
We consider the $0$-section
$X\to {\mathbf A}^1_X$
as a Cartier divisor.
Let it be denoted by $Z$
and consider the closed immersions
$$\begin{CD}
{\mathbf A}^1_D
@>>>
{\mathbf A}^1_X
@<<<
mZ.\end{CD}$$
Then, 
we define an open subscheme
$\widetilde X^{(mD)}
\subset {{\mathbf A}^1_X}^{
({\mathbf A}^1_D\cdot mZ)}$
of the dilatation by further removing
the proper transform of
the zero-section $Z$.
If $m=1$,
we recover the construction
in Example \ref{exdnc}.2.

Let $D_1,\ldots,D_h$ be
Cartier divisors of $X$
and $M=m_1D_1+\cdots+m_hD_h$
be a formal linear combination with
integral coefficients
$m_1\geqq  1,\ldots,m_h\geqq  1$.
Then, we define
\begin{equation}
f\colon \widetilde X^{(M)}\to X
\label{eqdfXM}
\end{equation}
to be the fibered product
of $\widetilde X^{(m_iD_i)}$
over $X$ for $i=1,\ldots,h$.

In the terminology of log product
\cite[Proposition 4.2.1]{KSI},
the construction of $\widetilde X^{(M)}$
is described as follows.
We regard $X$ as a log scheme
by the log structure
defined by the Cartier divisors
$D_1,\ldots,D_h$.
Let $P$ and $Q$ denote the
monoids ${\mathbf N}^h$
and let $M\colon Q\to P$
be the multiplication by $m_i$
on the $i$-th component.
The frame
$X\to [Q]$ defined by the Cartier divisors
$D_1,\ldots,D_h$
and the canonical frame
${\mathbf A}^h_{\mathbf Z}=
{\mathbf S}[P]=
{\rm Spec}\ {\mathbf Z}[P]\to [P]$
induce a frame
${\mathbf A}^h_X\to [P+Q]$.
Then $\widetilde X^{(M)}$
is the log product
${\mathbf A}^h_X\times_{[P+Q]}[P]$
with respect to the surjection
${\rm id}_P+M\colon P+Q\to P$
constructed in 
\cite[Proposition 4.2.1]{KSI}.

Let $U=X\sm D$ be the complement
of the union $D=D_1\cup\cdots\cup D_h$.
Then, the inverse image 
$f^{-1}(U)$ is a trivial ${\mathbf G}_m^h$-torsor over $U$.
The action of ${\mathbf G}_m^h$ on
$f^{-1}(U)$ is uniquely extended to
an action of ${\mathbf G}_m^h$ on
$\widetilde X^{(M)}$ over $X$
by the universality of dilatation.

\begin{lm}\label{lmXM}
Let $X$ be a smooth scheme
over a perfect field $k$,
$D$ be a divisor of $X$
with simple normal crossings
and $M=m_1D_1+\cdots+m_hD_h$
be a linear combination with
integral coefficients
$m_1\geqq  1,\ldots,m_h\geqq  1$
of the irreducible components
$D_1,\ldots,D_h$.

Then, $\widetilde X^{(M)}$
is smooth over $k$
and the canonical morphism
$f\colon \widetilde X^{(M)}\to X$
is flat.
Let $\widetilde D_i^{(M)}$
be the inverse image of $Z$
by the composition 
$\widetilde X^{(M)}
\to \widetilde X^{(m_iD_i)}
\to X\times_k{\mathbf A}^1_k$.
Then, the sum
$\widetilde D^{(M)}=
\widetilde D_1^{(M)}+\cdots+
\widetilde D_h^{(M)}$ is 
a divisor of 
$\widetilde X^{(M)}$ with simple
normal crossings.
For $i=1,\ldots,h$,
we have $f^*D_i=m_i
\widetilde D_i^{(M)}$.
If $M=D$,
the canonical morphism
$\widetilde X^{(M)}\to X$ is smooth.
\end{lm}

\begin{proof}
Since the question is local on $X$,
we may assume that
there is a smooth morphism
$X\to {\mathbf A}_k^h$
such that $D_i$ is the inverse image
of the $i$-th coordinate hyperplane for
$i=1,\ldots,h$.
Further we may assume that
$X= {\mathbf A}_k^h$
and $h=1$ as
$\widetilde X^{(M)}$ is then defined
as the product.
In this case, 
$\widetilde X^{(M)}
= \widetilde {\mathbf A}_k^{1\ (mZ)}\to 
X= {\mathbf A}_k^1=
{\rm Spec}\ k[T]$
is given by
${\rm Spec}\ k[T,S,U^{\pm 1}]/(T-US^m)$
and the assertion follows.

If $M=D$,
the assertion
follows from Example \ref{exdnc}.2.
\end{proof}

We study the functoriality
of the construction of $\widetilde X^{(M)}$.
Let $f\colon Y\to X$ be a morphism
of schemes,
$E_1,\ldots,E_k$ be
Cartier divisors of $Y$
and $N=n_1E_1+\cdots+n_kE_k$
be a linear combination with
integral coefficients
$n_1\geqq  1,\ldots,n_k\geqq  1$.
Assume for each $i=1,\ldots,h$
that the pull-back
$f^*D_i=\sum_{j=1}^ke_{ij}E_j$
for integers $e_{ij}$ 
and that
$l_{ij}=e_{ij}n_j/m_i$ 
is an integer for every $j=1,\ldots,k$.
Let 
$(T_i)$ and $(S_j)$ 
be the coordinates
of
${\mathbf A}^h_X$ and
${\mathbf A}^k_Y$ respectively
and we define a morphism
$\tilde f\colon {\mathbf A}^k_Y\to 
{\mathbf A}^h_X$
lifting $f\colon Y\to X$
by sending $T_i$
to $S_1^{l_{i1}}\cdots S_k^{l_{ik}}$.
Then, 
by the universality of dilatations,
the morphism
$\tilde f\colon {\mathbf A}^k_Y\to 
{\mathbf A}^h_X$
is uniquely lifted to
\begin{equation}
\tilde f\colon  
\widetilde Y^{(N)}
\to \widetilde X^{(M)}.
\label{eqXMYN}
\end{equation}
If $h=k$, if each $E_i$
is the pull-back of $D_i$
and if $n_i=m_i$ for each $i=1,\ldots,h$,
then the diagram
\begin{equation}
\begin{CD}
\widetilde Y^{(N)}
@>>> \widetilde X^{(M)}\\
@VV{\qquad\ \Box }V@VVV\\
Y@>>>X
\end{CD}
\label{eqXMYNc}
\end{equation}
is cartesian.

We consider the case where
$f\colon Y\to X$
is the identity of $X$.
Let $M'=m'_1D_1+\cdots+m'_hD_h$
be another linear combination with
integral coefficients
$m'_1\geqq  1,\ldots,m'_h\geqq  1$
divisible by $M$ in the sense that
$l_i=m'_i/m_i$ is an integer for each $i=1,\ldots,h$.
Then, we have a canonical morphism
\begin{equation}
\widetilde X^{(M')}\to
\widetilde X^{(M)}
\label{eqXMM'}
\end{equation}
 over $X$.
It is compatible with 
the actions of 
${\mathbf G}_{m,X}^h$
with respect to the morphism
${\mathbf G}_{m,X}^h\to
{\mathbf G}_{m,X}^h$
defined by $l_i$-th power
on $i$-th component.

\subsection{Extensions of a vector bundle
in characteristic $p>0$}\label{ssExt}

We study extensions
of a vector bundle
by a finite \'etale group scheme
on a scheme of characteristic $p>0$.

\begin{lm}\label{lmExt}
Let $E$ be a vector bundle
over a scheme $S$
of characteristic $p>0$ and
let 
$1\to G\to \widetilde E\to E\to 1$
be an extension of $E$ by
an \'etale group scheme $G$ over $S$.
Assume that for every point $s$ of $S$,
the fiber $\widetilde E_s$ is connected.
Then, $\widetilde E$ and consequently
$G$ are commutative and killed by $p$.
\end{lm}

\begin{proof}
Since $E$ is commutative,
the morphism
$\widetilde E\times_S \widetilde E\to
\widetilde E$ defined by
sending $(x,y)$ to the commutator
$[x,y]$ induces
a morphism of schemes
to the kernel 
$G$ of $\widetilde E\to E$.
Since the fiber of
$\widetilde E \times_S \widetilde E$ is
connected for every point of $S$,
the image is in the identity section $S\subset G$.
It defines a morphism
$\widetilde E \times_S \widetilde E\to S$
of schemes
since the identity section is
an open subscheme of $G$ \'etale over $S$.
Hence $\widetilde E$
is commutative.

Similarly,
the morphism
$p\colon \widetilde E\to
\widetilde E$
defined by sending
$x$ to $px$ induces
a morphism of schemes
to the kernel $G$
of $\widetilde E\to E$
and to a morphism of schemes
to $S$.
Hence $\widetilde E$
is killed by $p$.
\end{proof}

Let $E$ be a vector bundle
over a scheme $S$
of characteristic $p>0$ and
let $G$ be a finite \'etale commutative
group scheme over $S$,
killed by $p$.
Let $E^\vee={\rm Hom}_S(E,{\mathbf A}^1)$
be the dual vector bundle of $E$
and $G^\vee={\rm Hom}_S(G,{\mathbf F}_p)$
be the dual finite \'etale scheme of $G$.
For commutative
group schemes $A$ and $B$ over $S$,
let ${\rm Mor}_S(A,B)$ and
${\rm Ext}_S(A,B)$
denote the abelian group of 
morphisms of group schemes
and that of extensions of group schemes
respectively.

Let $F\colon {\mathbf A}^1
\to {\mathbf A}^1$ denote
the Frobenius morphism
defined by sending the coordinate $t$ to $t^p$.
For
$E={\mathbf A}^1$
and $G={\mathbf F}_p$,
the Artin-Schreier sequence
\begin{equation}
\begin{CD}
0@>>> {\mathbf F}_p
@>>> {\mathbf A}^1
@>{F-1}>>
{\mathbf A}^1
@>>>0
\end{CD}
\label{eqas}
\end{equation}
define an element
$[AS]\in {\rm Ext}_S({\mathbf A}^1,
{\mathbf F}_p)$.
The pull-back and the push-forward
of $[AS]$ define a canonical morphism
\begin{equation}
{\rm Mor}_S(G^\vee,E^\vee)
\to {\rm Ext}_S(E,G)
\label{eqAS}
\end{equation}
of abelian groups by \'etale descent
since a vector bundle $E$
is locally isomorphic to
a direct sum of ${\mathbf A}^1$
and a finite \'etale group
scheme $G$ of
${\mathbf F}_p$-vector spaces
is \'etale locally isomorphic to
a direct sum of ${\mathbf F}_p$.

For an integer $n\geqq  1$,
let $S_n$ denote the scheme $S$
regarded as a scheme over $S$
by the $n$-times iteration of 
the absolute Frobenius $S\to S$.
Since the \'etale site remains the same
by a radicial surjective morphism,
the pull-back map
${\rm Ext}_S(E,G)
\to {\rm Ext}_{S_n}(E\times_SS_n,
G\times_SS_n)$ is an isomorphism.
Hence,
(\ref{eqAS}) induces a morphism
\begin{equation}
\varinjlim_n
{\rm Mor}_{S_n}(G^\vee,E^\vee)
\to {\rm Ext}_S(E,G).
\label{eqASn}
\end{equation}

\begin{pr}\label{prExt}
Let $E$ be a vector bundle
over a scheme $S$
of characteristic $p>0$ and
let $G$ be a finite \'etale commutative
group scheme over $S$,
killed by $p$.
The morphism
{\rm (\ref{eqASn})}
of abelian groups
is an isomorphism
if $S$ is 
quasi-compact.
\end{pr}

The following proof 
is a modification of
that in the case $S={\rm Spec}\ k$
for a perfect field $k$
in
\cite[Lemme 3]{GC}.

\begin{proof}
Since a morphism
$E\to G$ 
of group schemes
from a vector bundle
to a finite \'etale group scheme
is trivial,
the presheaf
$U\mapsto {\rm Ext}_U(E,G)$
is an \'etale sheaf on $S$.
Since $S$ is assumed quasi-compact,
we may assume $S
={\rm Spec}\ A$ is affine,
$E={\mathbf A}^n$ and $G$ is constant.
Further, we may assume 
$E={\mathbf A}^1$ and 
$G={\mathbf F}_p$.

We identify the ring
${\rm End}_A({\mathbf A}^1)$
with the non-commutative ring
$A[F]=\bigoplus_nAF^n$ defined by the relations
$F\cdot a=a^p F$
for $a\in A$.
Then, the boundary map
for the Artin-Schreier sequence (\ref{eqas})
induces an injection
\begin{equation}
A[F]/(F-1)A[F]
\to 
{\rm Ext}_A({\mathbf A}^1,
{\mathbf F}_p).
\label{eqasin}
\end{equation}
By the isomorphism
defined by the inductive
system 
$$\begin{CD}
A@>{a\mapsto aF^{n+m}}>>
A[F]/(F-1)A[F]\\
@A{F^m}AA@|\\
A@>{a\mapsto aF^{n}}>>
A[F]/(F-1)A[F]
\end{CD}$$
of morphisms of abelian groups,
we identify 
$A[F]/(F-1)A[F]$
with the additive group
of the perfection
$A^{p^{-\infty}}
=\varinjlim_{F^n}A$.

For
$S={\rm Spec}\ A$,
$E={\mathbf A}^1$ and 
$G={\mathbf F}_p$,
the abelian group
${\rm Mor}_S(G^\vee,E^\vee)=
{\rm Mor}_A({\mathbf F}_p,
{\mathbf A}^1)$
is identified with
the additive group $A$
and the transition map
$A={\rm Mor}_{S_n}(G^\vee,E^\vee)
\to
A={\rm Mor}_{S_{n+1}}(G^\vee,E^\vee)$
is the absolute Frobenius.
Hence the direct limit
$\varinjlim_n
{\rm Mor}_{S_n}(G^\vee,E^\vee)$
is identified with
the additive group of
the perfection
$A^{p^{-\infty}}$
and to
$A[F]/(F-1)$.
Thus, the morphism
(\ref{eqASn}) is an injection
by (\ref{eqasin}).

We show the surjectivity.
By replacing $A$ by
the perfection
$A^{p^{-\infty}}$,
we may assume that
the absolute Frobenius
$A\to A$ is a bijection.
We consider a commutative diagram
\begin{equation}
\begin{CD}
{\rm Ext}_A({\mathbf A}^1,
{\mathbf F}_p)
@>>>
H^1({\mathbf A}^1_A,
{\mathbf F}_p)
@>{+^*-{\rm pr}_1^*-
{\rm pr}_2^*}>>
H^1({\mathbf A}^2_A,
{\mathbf F}_p)\\
@AAA @AAA @AAA\\
A[F]@>>>
A[T]@>{+^*-{\rm pr}_1^*-
{\rm pr}_2^*}>>
A[T_1,T_2]
\end{CD}
\label{eqAS2}
\end{equation}
where
the left vertical arrow
is induced by {\rm (\ref{eqasin})}
and $+,{\rm pr}_1,
{\rm pr}_2\colon
{\mathbf A}^2_A\to
{\mathbf A}^1_A$
denote the addition
and the projections.
The composition of the upper line is $0$
since ${\rm Ext}$ is an additive functor.
The middle and the right
vertical arrows
are the surjections
defined by the pull-back
of the Artin-Schreier covering
(\ref{eqas}).
The lower left horizontal
arrow is the left $A$-linear
map sending $F^n$ to $T^{p^n}$
for $n\geqq  0$.

Since the constant \'etale covering
${\mathbf A}^1_A\times
{\mathbf F}_p$ of
${\mathbf A}^1_A$
has a unique structure
of extension of
${\mathbf A}^1_A$ by
${\mathbf F}_p$,
the upper left horizontal arrow
${\rm Ext}_A({\mathbf A}^1,
{\mathbf F}_p)
\to
H^1({\mathbf A}^1_A,
{\mathbf F}_p)$
is an injection.
We regard
$A[T]$ and $A[T_1,T_2]$
as left $A[F]$-modules
by the left multiplication of $A$
and the action of $F$
defined as the absolute Frobenius.
Then, the lower horizontal arrows
are $A[F]$-linear and
the middle and the right vertical arrows
induce isomorphisms
from the quotients
\begin{equation}
\begin{CD}
A[F]/(F-1)A[F]
\to
A[T]/(F-1)A[T]@>{+^*-{\rm pr}_1^*-
{\rm pr}_2^*}>>
A[T_1,T_2]/(F-1)A[T_1,T_2]
\end{CD}
\label{eqAS3}
\end{equation}
Thus, by an elementary diagram chasing,
the exactness of the sequence (\ref{eqAS3})
implies the surjectivity of (\ref{eqasin}).

We show that the sequence (\ref{eqAS3}) is exact.
As $A[F]$-modules,
we have direct sum decompositions
$A[T]=A\oplus \bigoplus_{p\nmid n}
A[F]\cdot T^n$ and
$A[T_1,T_2]=A\oplus \bigoplus_{p\nmid (n_1,n_2)}
A[F]\cdot T_1^{n_1}T_2^{n_2}$
where the second factors are free 
left $A[F]$-modules.
Since the absolute Frobenius on $A$
is assumed to be a bijection,
the inclusion
induces an isomorphism $A\to A[F]/(F-1)A[F]$.
Thus, the sequence (\ref{eqAS3}) is
isomorphic to
\begin{equation}
\begin{CD}
A
@>{1\mapsto T}>>
A/(F-1)\oplus
\bigoplus_{p\nmid n}
A\cdot T^n
@>{+^*-{\rm pr}_1^*-
{\rm pr}_2^*}>>
A/(F-1)\oplus
\bigoplus_{p\nmid (n_1,n_2)}
A\cdot T_1^{n_1}T_2^{n_2}.
\end{CD}
\label{eqAS4}
\end{equation}
The second map
$+^*-{\rm pr}_1^*
-{\rm pr}_2^*$
sends
$T^n$ to
$$(T_1+T_2)^n-T_1^n-T_2^n
=
\sum_{i=1}^{n-1}
\binom ni T_1^{n-i}T_2^i
$$
for an integer $p\nmid n$.
Since this is $0$ if and only if $n=1$,
the assertion follows.
\end{proof}

The author thanks an anonymous referee
for suggesting an alternative proof of the
surjectivity of (\ref{eqasin})
using the injectivity of the upper left
horizontal arrow in (\ref{eqAS2}),
the long exact sequences
deduced from (\ref{eqas}) 
and the fact that ${\rm Ext}({\mathbf A}^1,{\mathbf A}^1)$ 
is isomorphic to the group of symmetric 2-cocycles.

\begin{lm}\label{lmconn}
Let $S$ be a scheme over
${\mathbf F}_p$,
$E$ be a vector bundle
over $S$ and
$G$ be a finite \'etale group scheme
of ${\mathbf F}_p$-vector
spaces over $S$.
For a morphism $G^\vee \to E^\vee$
of the duals
and for the corresponding extension
$0\to G\to \widetilde E\overset \pi\to E\to 0$,
the following conditions are
equivalent:

{\rm (1)}
For every point $s$ of $S$,
the fiber
$\widetilde E_s$ is connected.

{\rm (2)}
For every geometric point $\bar s$ of $S$,
the geometric fiber
$G^\vee_{\bar s}\to E^\vee_{\bar s}$
is an injection.

{\rm (3)}
$G^\vee\to E^\vee$
is a closed immersion.
\end{lm}

\begin{proof}
(1)$\Leftrightarrow$(2)
We may assume $S={\rm Spec}\ k$
for a field $k$. Further, we may
assume $k$ is algebraically closed
since a group scheme over a field
is connected if and only if
its geometric fiber is connected.

Since the composition
$G\to 
\widetilde E
\to
\pi_0(\widetilde E)$
is a surjection of ${\mathbf F}_p$-vector spaces,
it has a section.
Hence, the restriction of
$G^\vee \to E^\vee$
to 
$\pi_0(\widetilde E)^\vee$
regarded as a subgroup
is trivial.
Hence, if
$G^\vee\to E^\vee$ is injective,
we have
$\pi_0(\widetilde E)=0$
and $\widetilde E$ is connected.

Let $\chi\colon G\to {\mathbf F}_p$
be a non-trivial character
and $f\colon E\to {\mathbf A}^1$
be the image of $\chi$
by $G^\vee\to E^\vee$.
Then, the quotient
$\widetilde E/{\rm Ker}\ \chi$
is the extension defined by
the commutative diagram
\begin{equation}
\begin{CD}
0@>>>{\mathbf F}_p
@>>>
\widetilde E/{\rm Ker}\ \chi
@>>> E@>>>0\\
@.@|@VVV @VVfV@.\\
0@>>>{\mathbf F}_p
@>>>
{\mathbf A}^1
@>{F-1}>> {\mathbf A}^1@>>>0.
\end{CD}
\label{eqchi}
\end{equation}
Hence $f\neq 0$ if and only if $\widetilde E/{\rm Ker}\ \chi$
is connected.
If $\widetilde E$ is connected,
$\widetilde E/{\rm Ker}\ \chi$
is connected for every $\chi\in G^\vee$
and $G^\vee\to E^\vee$ is injective.

(2)$\Leftrightarrow$(3)
Since 
$G^\vee$
is finite \'etale,
the map 
$G^\vee\to E^\vee$
is proper and unramified.
The condition (2)
is equivalent to
that 
$G^\vee\to E^\vee$
is radicial by \cite[Proposition (3.7.5)]{EGA1}.
Hence it follows from 
\cite[Corollaire (18.12.6) c)$\Rightarrow$ a)]{EGA4-4}.
\end{proof}

We study \'etale sheaves on $E$
such that the restrictions
on the geometric fibers of $\widetilde E$
are constant.
First, we consider a more general setting.
Let $G$ be a finite \'etale commutative
group scheme
over a scheme $S$
and $n\geqq  1$ be an integer
annihilating $G$.
Set $\Lambda={\mathbf Z}[\frac 1n,\zeta_n]$.
The dual $G^\vee=
{\cal H}om(G,{\mathbf Z}/
n{\mathbf Z})$ of $G$
is defined as a finite \'etale commutative
group scheme
over a scheme $S$.

Let $X$ be a scheme over $S$
and $\pi\colon E\to X$ be a $G$-torsor over $X$.
We will define a locally constant sheaf
${\cal L}_E$ of free $\Lambda$-modules
of rank 1 on the scheme 
$X\times_SG^\vee$ as follows.
The push-forward
$\pi_*\Lambda$
is a locally constant sheaf on $X$ of
free $\Lambda[G]$-modules
of rank 1.
On $G^\vee$,
the tautological character
$G\to {\mathbf Z}/
n{\mathbf Z}$
induces a character
$G\to \Lambda^\times$
by $1\mapsto \zeta_n$
and hence defines
a morphism
$\Lambda[G]\to \Lambda$
of $\Lambda$-algebras.
We define ${\cal L}_E$
as ${\rm pr}_1^*\pi_*\Lambda
\otimes_
{\Lambda[G]}\Lambda$.

\begin{lm}\label{lmGG}
Let $S$ be a scheme and
$n\geqq  1$ be an integer.
Let $G$ be a finite \'etale group
scheme of ${\mathbf Z}/n{\mathbf Z}$-modules
over $S$
and $G^\vee
={\cal H}om(G,
{\mathbf Z}/n{\mathbf Z})$ be the dual
finite \'etale group scheme.
Let $\Lambda$ be
the ring ${\mathbf Z}[\frac 1n,
\zeta_n]$ and
${\cal A}$ be
an \'etale sheaf of $\Lambda$-algebras
on $S$.
Let $X$ be a scheme over $S$
and $\pi\colon E\to X$ 
be a $G$-torsor over $X$.
Let ${\cal L}_E$ denote
the locally constant sheaf
of $\Lambda$-modules
of rank $1$ on $X\times_SG^\vee$
defined above.

Let ${\cal M}$ be an \'etale sheaf of
${\cal A}$-modules on $X$
and $\iota\in \Gamma(E, {\cal M})$
be a section defining an isomorphism of
${\cal A}$-modules
${\cal A}\to {\cal M}$ on $E$.
Let $t\colon S\to E$
be a section and
assume that the restriction map
$t^*\colon \Gamma(E,{\cal A})
\to 
\Gamma(S,{\cal A})$ is an isomorphism.

Then, there exists an idempotent
$e_{\cal M}\in \Gamma(X\times_S
G^\vee,{\cal A})$
and an isomorphism
\begin{equation}
{\rm pr}_{1*}
((e_{\cal M}\cdot {\cal A})\otimes_{\Lambda} 
{\cal L}_E)
\to {\cal M}
\label{eqAM}
\end{equation}
of ${\cal A}$-modules
on $X$ where ${\rm pr}_1
\colon X\times_SG^\vee\to X$
denotes the projection.
\end{lm}

\begin{proof}
First, we assume $G$ and hence
$G^\vee$ are constant.
Let ${\cal A}$ also denote its pull-backs
abusively.
Since the restriction map
$\Gamma(E,{\cal A})
\to 
\Gamma(S,{\cal A})$ is assumed
an isomorphism,
the action of $G$ on
$\Gamma(E,{\cal M})$
induced by that on $E$
defines a character
$\alpha\colon
G\to \Gamma(S,{\cal A}^\times)$
satisfying
$g(\iota)=\alpha(g)\cdot \iota$.
The idempotents
$e_\chi=\frac1{|G|}\sum_g\chi^{-1}(g)\alpha(g)$
for characters $\chi\colon G\to
\Lambda^\times$
satisfy
$g(e_\chi\cdot \iota)
=\chi(g)e_\chi\cdot \iota$
and $\sum_\chi e_\chi=1$.
The isomorphism
$\iota\colon
\pi^*{\cal A}\to\pi^*{\cal M}$
induces an isomorphism
$e_\chi{\cal A}\otimes
{\cal L}_\chi\to
e_\chi{\cal M}$ for each $\chi$.
They define an idempotent
$e_{\cal M}=(e_\chi)\in \Gamma(X\times_SG^\vee,{\cal A})
=\bigoplus_\chi \Gamma(X,{\cal A})$
and 
an isomorphism
${\rm pr}_{1*}
((e_{\cal M}\cdot {\cal A})\otimes
{\cal L}_E)
=\bigoplus_\chi
e_\chi{\cal A}\otimes
{\cal L}_\chi\to
\bigoplus_\chi
e_\chi
{\cal M}={\cal M}$.

In general, we obtain an idempotent $e_{\cal M}$
and an isomorphism (\ref{eqAM}) by patching.
\end{proof}

\begin{cor}\label{corGG}
Let $S$ be a scheme over ${\mathbf F}_p$
and
let $0\to G\to \widetilde E\to E\to 0$
be the extension of
a vector bundle $E$ over $S$
by a finite \'etale group
scheme $G$
of ${\mathbf F}_p$-vector spaces
over $S$
corresponding to a morphism $G^\vee
\to E^\vee$ on the duals.
Let $\Lambda$ be
the ring ${\mathbf Z}[\frac 1p,
\zeta_p]$
 and
${\cal A}$ be
an \'etale sheaf of $\Lambda$-algebras
on $S$.

{\rm 1.}
The locally constant sheaf
${\cal L}_{\widetilde E}$ of 
free $\Lambda$-modules
of rank $1$
on $E\times_SG^\vee$
defined above
is the pull-back of 
the locally constant sheaf of 
free $\Lambda$-modules
of rank $1$
on $E\times_SE^\vee$
defined by the Artin-Schreier
equation $t^p-t=\langle f,x\rangle$
by the map $E\times_SG^\vee\to
E\times_SE^\vee$.

{\rm 2.}
Let ${\cal M}$ be an \'etale sheaf of
${\cal A}$-modules on $E$
and $\iota\in \Gamma(\widetilde E, {\cal M})$
be a section defining an isomorphism of
${\cal A}$-modules
${\cal A}\to {\cal M}$ on $\widetilde E$.
Assume that the geometric fibers of 
$\widetilde E \to S$
are connected.

Then, there exists an idempotent
$e_{\cal M}\in \Gamma(E\times_SG^\vee,{\cal A})$
and $\iota$ induces an isomorphism
$${\rm pr}_{1*}
((e_{\cal M}\cdot {\cal A})\otimes_{\Lambda}
{\cal L}_{\widetilde E})
\to {\cal M}$$
of ${\cal A}$-modules
on $E$ where ${\rm pr}_1
\colon E\times_SG^\vee\to E$
denotes the projection.
\end{cor}

\begin{proof}
The assertion 1 is clear from
the definition of
the extension $0\to G\to \widetilde E
\to E\to 0$.
The assumption that the geometric fibers of 
$\widetilde E \to S$
are connected implies 
that the restriction
$\Gamma(\widetilde E,{\cal A})
\to \Gamma(S,{\cal A})$ is an isomorphism.
Hence, the assertion 2 is a special case
of Lemma \ref{lmGG}.
\end{proof}

\section{Ramification}

In this section,
$k$ denotes a perfect field
of characteristic $p>0$
and 
$X$ denotes a smooth separated scheme over $k$.
Let $D$ be a divisor of $X$ 
with simple normal crossings 
and $U=X\sm D$ be the complement.
Let $D_1,\ldots,D_h$ be
the irreducible components of $D$
and $R=r_1D_1+\cdots+r_hD_h$
be a linear combination
with rational coefficients
$r_i\geqq  1$ for every $i=1,\ldots, h$.
Let $M=m_1D_1+\cdots+m_hD_h$
be a linear combination
with integral coefficients
$m_i\geqq  1$ such that
$m_ir_i$ is an integer for every $i=1,\ldots, h$.
Let $Z\subset D$ be the
union of the irreducible components $D_i$
such that $r_i>1$.

\subsection{Construction of dilatations}\label{ssPn}

We define oversimplicial schemes 
\begin{equation}
\begin{CD}
T_\bullet^{(R,M)}
@>{\subset}>>
P_\bullet^{(R,M)}
@<{\supset}<<
U^{\bullet+1}\times_k{\mathbf G}_m^h
=
(U\times_kU)_U^{\times\bullet}
\times_k{\mathbf G}_m^h
\end{CD}
\label{eqTPU}
\end{equation}
in Lemma \ref{lmPRM} below
and study their structures.
The fibered multi-products
$(U\times_kU)_U^{\times\bullet}$
(\ref{eqPSxn})
is taken with respect to the
first and the second projections
${\rm pr}_1,
{\rm pr}_2\colon 
U\times_kU\to U$.
The superscript $h$ denotes
the number of irreducible components of $D$.
If $R$ has integral coefficients,
we also define
\begin{equation}
\begin{CD}
T_\bullet^{(R)}
@>{\subset}>>
P_\bullet^{(R)}
@<{\supset}<<
U^{\bullet+1}=
(U\times_kU)_U^{\times\bullet}
\end{CD}
\label{eqtPR0}
\end{equation}
directly
without introducing
an auxiliary divisor $M$
in Lemma \ref{lmPR}.

To define (\ref{eqTPU}),
we construct a commutative
diagram of oversimplicial schemes 
\begin{equation}
\begin{CD}
T_\bullet^{(R,M)}
@>{\subset}>>
D_\bullet^{(R,M)}
@>{\subset}>>
P_\bullet^{(R,M)}
@<{\supset}<<
U^{\bullet+1}\times_k{\mathbf G}_m^h\\
@.@VVV @VVV @|\\
@.D_\bullet^{(D,M)}
@>{\subset}>>
P_\bullet^{(D,M)}
@<{\supset}<<
U^{\bullet+1}\times_k{\mathbf G}_m^h\\
@.@VVV @VVV @VVV\\
@.D_\bullet^{(D)}
@>{\subset}>>
P_\bullet^{(D)}
@<{\supset}<<
U^{\bullet+1}\\
@.@. @VVV @|\\
@.&X^{\bullet+1}=&
P_\bullet
@<{\supset}<<
U^{\bullet+1}\\
\end{CD}
\label{eqtPR}
\end{equation}
in Lemmas \ref{lmPnD}, \ref{lmPDM}, \ref{lmPRM} below.
The horizontal arrows in the right column
are
open immersions of the
complements of the images of
the closed immersions on the left.
The second line from below is
a preliminary construction to make the pull-back of
$D$ by projections independent of the projection.
The third line from below is
also a preliminary one to avoid introducing
singularity even in the case where 
the coefficients $r_i$ of $R$ may not be integers.

Let $n\geqq  0$ be an integer.
We consider $X$ as a closed subscheme
of $X^{n+1}=P_n$ by the diagonal immersion.
For an integer $0\leqq j\leqq n$,
let ${\rm pr}_j\colon
X^{n+1}\to X$ denote the projection
and 
we consider closed subschemes
\begin{equation}
\begin{CD}
{\rm pr}_j^*D@>>> X^{n+1}@<<< X.
\end{CD}
\label{eqPnD}
\end{equation}
The dilatation $P_{n,j}^{(D)}
=X^{n+1\ ({\rm pr}_j^*D\cdot X)}$
is an open subscheme of 
the blow-up of $X^{n+1}$
at $D\subset X$.
We define
$P_n^{(D)}$ to be the intersection of
$P_{n,j}^{(D)}$ for $j=0,\ldots,n$.

\begin{lm}\label{lmPnD}
{\rm 1.}
The inverse image 
$D_n^{(D)}$ of $D\subset X$
by the canonical map
$P_n^{(D)}\to X^{n+1}$
is the same as that by the
composition
$P_n^{(D)}\to X^{n+1}\to X$
with $n+1$ projections.
The complement
$P_n^{(D)}\sm D_n^{(D)}$
is $U^{n+1}$.

{\rm 2.}
The schemes $P_n^{(D)}$ 
form an oversimplicial scheme $P_\bullet^{(D)}$.
The morphisms 
$P_n^{(D)}\to X^{n+1}$
define a morphism 
$P_\bullet^{(D)}\to P_\bullet=X^{\bullet+1}$
of oversimplicial schemes.

{\rm 3.}
For an injection $[0,m]\to [0,n]$,
the induced morphism
$P_n^{(D)}\to P_m^{(D)}$
is smooth.

{\rm 4.}
The scheme $P_n^{(D)}$ is smooth over $k$.
The inverse image $D_n^{(D)}$
is a divisor of $P_n^{(D)}$ with simple normal crossings.
\end{lm}

\begin{proof}
1.
Since $P_{n,j}^{(D)}$
is the complement in the blow-up of $X^{n+1}$
at $D\subset X$
of the proper transform of
${\rm pr}_j^*D\subset X^{n+1}$,
it follows from the definition of 
$P_n^{(D)}$ as their intersection.

2.
It follows from the assertion 1
and the universality of dilatations.

3.
Let $f\colon [0,m]\to [0,n]$ be an injection.
By Corollary \ref{cordlf}.1 applied to
the commutative diagram
$$\begin{CD}
{\rm pr}_{f(0)}^*D
@>>>
X^{n+1}@<<<X\\
@VVV@VVV @|\\
D\times_k X^m@>>>
X^{m+1}@<<<X,
\end{CD}$$
the induced map
$P_n^{(D)}\to P_m^{(D)}$ is smooth.

4.
Taking $m=0$ in 3.,
we see that
the scheme $P_n^{(D)}$ is smooth over $k$
since $P_0^{(D)}=X$ is smooth over $k$.
The assertion on $D_n^{(D)}$ follows from this
and the assertion 1.
\end{proof}

Next, we define the schemes in the second
line in (\ref{eqtPR}).
Let $n\geqq  0$ be an integer.
For $i=1,\ldots,h$,
let $D_{n,i}^{(D)}\subset P_n^{(D)}$
be the pull-back of $D_i$.
By Lemma \ref{lmPnD},
the sum  $D_n^{(D)}=
D_{n,1}^{(D)}+\cdots+
D_{n,h}^{(D)}$ is a divisor with
simple normal crossings.
We define 
$P_n^{(D,M)}$
by applying the construction (\ref{eqdfXM})
to the linear combination
$m_1D_{n,1}^{(D)}+\cdots+
m_hD_{n,h}^{(D)}$.
For $i=1,\ldots, h$,
we define a divisor
$D_{n,i}^{(D,M)}
\subset P_n^{(D,M)}$ 
as the pull-back of the zero-section
by the $i$-th projection as in loc.\ cit.

\begin{lm}\label{lmPDM}
{\rm 1.}
The scheme
$P_n^{(D,M)}$ is smooth 
over $k$ and flat over $P_n^{(D)}$.
The union 
$D_n^{(D,M)}=
D_{n,1}^{(D,M)}\cup\cdots\cup D_{n,h}^{(D,M)}$ is
a divisor with simple normal crossings.
The pull-back of $D_{n,i}^{(D)}$ by
$P_n^{(D,M)}\to P_n^{(D)}$
is $m_iD_{n,i}^{(D,M)}$
for $i=1,\ldots, h$.
The complement
$P_n^{(D,M)}\sm
D_n^{(D,M)}$
is canonically isomorphic to
$U^{n+1}\times_k{\mathbf G}_m^h$.

{\rm 2.}
The schemes $P_n^{(D,M)}$ 
form an oversimplicial scheme $P_\bullet^{(D,M)}$.
The morphisms 
$P_n^{(D,M)}\to P_n^{(D)}$
define a morphism 
$P_\bullet^{(D,M)}\to P_\bullet^{(D)}$
of oversimplicial schemes.
If $M=D$, the morphism
$P_\bullet^{(D,M)}\to P_\bullet^{(D)}$
is smooth.

{\rm 3.}
The diagram
\begin{equation}
\begin{CD}
P_n^{(D,M)}@>>>P_n^{(D)}\\
@VV{\qquad\ \Box }V @VVV\\
P_m^{(D,M)}@>>>P_m^{(D)}
\end{CD}
\label{eqPnDc}
\end{equation}
is cartesian for every $[0,m]\to [0,n]$.
In particular
$P_n^{(D,M)}
\to P_m^{(D,M)}$
is smooth if $[0,m]\to [0,n]$
is an injection.
\end{lm}

\begin{proof}
1.
It follows from Lemma \ref{lmXM}.

2.
It follows from assertion 1 and the 
functoriality (\ref{eqXMYN}).
The assertion in the case $M=D$ follows
from the last assertion in Lemma \ref{lmXM}.

3.
The cartesian diagram
(\ref{eqPnDc}) follows from the 
cartesian diagram (\ref{eqXMYNc}).
The smoothness for an injection follows
from the cartesian diagram
(\ref{eqPnDc}) and Lemma \ref{lmPnD}.3.
\end{proof}

We set
\begin{equation}
\widetilde X^{(M)}=P_0^{(D,M)}
\supset
\widetilde D^{(M)}=D_0^{(D,M)}
\label{eqXM}
\end{equation}
and define a subdivisor $\widetilde Z^{(M)}\subset
\widetilde D^{(M)}$
to be the union $\bigcup_{r_i>1}\widetilde D_i^{(M)}$.
By Lemma \ref{lmPDM},
the scheme
$\widetilde X^{(M)}$ is smooth over $k$
and its divisor
$\widetilde D^{(M)}$
has simple normal crossings.
Further, the morphism 
$\widetilde X^{(M)}\to X$ is flat.

Let $n\geqq  0$ be an integer.
We regard 
$\widetilde X^{(M)}=P_0^{(D,M)}$
as a closed subscheme of
$P_n^{(D,M)}$ by Lemma \ref{lmPDM}.2.
The linear combination
$\sum_{i=1}^hm_i(r_i-1)
D_{n,i}^{(D,M)}$
has integral coefficients
and defines a Cartier divisor of
$P_n^{(D,M)}$.
We consider closed immersions
$$\begin{CD}
\sum_{i=1}^hm_i(r_i-1)
D_{n,i}^{(D,M)}
@>>> P_n^{(D,M)}
@<<<
\widetilde X^{(M)}.
\end{CD}$$
We define
$P_n^{(R,M)}$ to be
the dilatation
defined by the immersions.
The canonical morphism
$P_n^{(R,M)}\to
P_n^{(D,M)}$
is an isomorphism
outside the inverse image of
$\widetilde Z^{(M)}$.
Define Cartier divisors of
$\widetilde X^{(M)}$ by 
\begin{equation}
\widetilde R^{(M)}
=r_1m_1\widetilde D_1^{(M)}
+\cdots+
r_hm_h\widetilde D_h^{(M)},
\qquad
\widetilde M^{(M)}
=m_1\widetilde D_1^{(M)}
+\cdots+
m_h\widetilde D_h^{(M)}.
\label{eqRM}
\end{equation}

\begin{lm}\label{lmPRM}
{\rm 1.}
The inverse images 
\begin{equation}
T_n^{(R,M)}
\subset D_n^{(R,M)}
\label{eqTnRM}
\end{equation}
of $\widetilde Z^{(M)}
\subset \widetilde D^{(M)}
\subset \widetilde X^{(M)}$
by the composition of
$P_n^{(R,M)}\to P_n^{(D,M)}$
with the $n+1$ projections
$P_n^{(D,M)}
\to \widetilde X^{(M)}$
do not depend on the projection.
The complement
$P_n^{(R,M)}\sm D_n^{(R,M)}$
is
$U^{n+1}\times_k{\mathbf G}_m^h$.

{\rm 2.}
The schemes $P_n^{(R,M)}$ 
form an oversimplicial scheme 
$P_\bullet^{(R,M)}$.
The morphisms 
$P_n^{(R,M)}\to P_n^{(D,M)}$
define a morphism 
$P_\bullet^{(R,M)}\to P_\bullet^{(D,M)}$
of oversimplicial schemes.

{\rm 3.}
For an injection $[0,m]\to [0,n]$,
the induced morphism
$P_n^{(R,M)}\to P_m^{(R,M)}$
is smooth.

{\rm 4.}
The scheme $P_n^{(R,M)}$
is smooth over $k$.
The inverse images $T_n^{(R,M)}
\subset D_n^{(R,M)}$
are divisors with simple normal crossings.
\end{lm}

\begin{proof}
The proof is similar to that of
Lemma \ref{lmPnD}.

1.
It follows from the construction.

2. 
It follows from the assertion 1
and the universality of dilatations.

3.
Let $f\colon [0,m]\to [0,n]$ be an injection.
By Corollary \ref{cordlf}.1 applied to
the commutative diagram
$$\begin{CD}
\sum_{i=1}^hm_i(r_i-1)
D_{n,i}^{(D,M)}
@>>>
P_n^{(D,M)}@<<<
\widetilde X^{(M)}
\\
@VVV@VVV @|\\
\sum_{i=1}^hm_i(r_i-1)
D_{m,i}^{(D,M)}
@>>>
P_m^{(D,M)}@<<<
\widetilde X^{(M)}
\end{CD}$$
and by Lemma \ref{lmPDM}.3,
the induced morphism
$P_n^{(R,M)}\to P_m^{(R,M)}$
is smooth.

4.
Taking $m=0$ in 3.,
we see that
the scheme $P_n^{(R,M)}$ is smooth over $k$
and that
the inverse images $T_n^{(R,M)}
\subset D_n^{(R,M)}$
are divisors with simple normal crossings
since $P_0^{(R,M)}=P_0^{(D,M)}=
\widetilde X^{(M)}$ is smooth over $k$
and that $\widetilde Z^{(M)}
\subset \widetilde D^{(M)}
\subset \widetilde X^{(M)}$
are divisors with simple normal crossings.
\end{proof}

In {\rm (\ref{eqtPR})},
the actions of
${\mathbf G}_m^h$
on the right terms
$U^{\bullet+1}\times_k{\mathbf G}_m^h$ 
on the top and the second lines
defined by the multiplication of
${\mathbf G}_m^h$
on the second factors ${\mathbf G}_m^h$
induce ${\mathbf G}_m^h$-actions
on the middle terms
$P_\bullet^{(D,M)}$ and
on $P_\bullet^{(R,M)}$.
They induce ${\mathbf G}_m^h$-actions on
$D_{i,\bullet}^{(D,M)}$,
$D_{i,\bullet}^{(R,M)}$
and on
$T_{i,\bullet}^{(R,M)}$.
The left horizontal arrows
are closed immersions of Cartier divisors
and the right horizontal arrows
are the open immersions
of the complements of the union.
The right squares 
of the diagram {\rm (\ref{eqtPR})} are cartesian.
The pull-backs of
the divisor $D_n^{(D)}\subset P_n^{(D)}$
are $M_n^{(D,M)}=\sum_{i=1}^h
m_iD_{n,i}^{(D,M)}
\subset P_n^{(D,M)}$
and $M_n^{(R,M)}=\sum_{i=1}^h
m_iD_{n,i}^{(R,M)}
\subset P_n^{(R,M)}$ respectively.

If $R=r_1D_1+\cdots+r_hD_h$ 
has integral coefficients,
we define
$P_n^{(R)}$ directly
as the dilatation
$(P_n^{(D)})^{(\sum_i(r_i-1)D_{n,i}^{(D)}\cdot X)}$
without introducing
an auxiliary divisor $M$.
The canonical map
$P_n^{(R)}\to P_n^{(D)}$ is
an isomorphism
outside the inverse image
of $Z=\bigcup_{r_i>1}D_i$.

\begin{lm}\label{lmPR}
Assume that $R$ has integral coefficients.

{\rm 1.}
The schemes $P_n^{(R)}$ 
form an oversimplicial scheme 
$P_\bullet^{(R)}$.
The morphisms 
$P_n^{(R)}\to P_n^{(D)}$
define a morphism 
$P_\bullet^{(R)}\to P_\bullet^{(D)}$
of oversimplicial schemes.

{\rm 2.}
For an injection
$[0,m]\to [0,n]$,
the morphism
$P_n^{(R)}\to P_m^{(R)}$
is smooth.

{\rm 3.}
The scheme $P_n^{(R)}$
is smooth over $k$.
The inverse images $T_n^{(R)}
\subset D_n^{(R)}$
of $Z\subset D\subset X$
by the composition
$P_n^{(R)}\to P_n^{(D)}\to X$
with $n+1$ projections
are the same 
and are divisors with simple normal crossings.
The complement
$P_n^{(R)}\sm D_n^{(R)}$
is $U^{n+1}$.

{\rm 4.}
For a divisor $M=m_1D_1+\cdots m_hD_h$
with integral coefficients $m_i\geqq  1$,
the diagram
\begin{equation}
\begin{CD}
P_\bullet^{(D,M)}@<<<
P_\bullet^{(R,M)}\\
@VV{\qquad\ \Box }V@VVV\\
P_\bullet^{(D)}@<<<
P_\bullet^{(R)}
\end{CD}
\label{eqPRRM}
\end{equation}
is cartesian
and the vertical arrows are flat.
If $M=D$, the vertical arrows are smooth.
\end{lm}

\begin{proof}
1.-3.
Similar to Lemma \ref{lmPRM}.

4.
For integers $n\geqq  0$,
the diagram
$$\begin{CD}
\sum_{i=1}m_i(r_i-1)D_{n,i}^{(D,M)}
@>>>
P_n^{(D,M)}@<<<
\widetilde X^{(M)}\\
@VV{\qquad\qquad\quad \Box }V@VV{\quad\quad \Box }V@VVV\\
\sum_{i=1}(r_i-1)D_{n,i}^{(D)}@>>>
P_n^{(D)}@<<<X
\end{CD}$$
is cartesian
and the vertical arrows are flat
by Lemma \ref{lmPDM}.
Hence, it suffices to apply
Lemma \ref{lmdlf}.1.
If $M=D$,
the morphism
$\widetilde X^{(M)}\to X$ is smooth.
\end{proof}

We show that the oversimplicial
schemes in the diagram (\ref{eqtPR})
are multiplicative.

\begin{lm}\label{lmprod}
{\rm 1.}
The 
oversimplicial schemes
$P_\bullet,P_\bullet^{(D)},
P_\bullet^{(D,M)}$
and 
$P_\bullet^{(R,M)}$
are multiplicative.
If $R$ has integral coefficients,
$P_\bullet^{(R)}$
is also multiplicative.

{\rm 2.}
The oversimplicial schemes
$D_{\bullet}^{(D)},
D_{\bullet}^{(D,M)},
D_{\bullet}^{(R,M)}$ and
$T_{\bullet}^{(R,M)}$
are strictly multiplicative.
If $R$ has integral coefficients,
$D_\bullet^{(R)}$ and
$T_\bullet^{(R)}$
are also strictly multiplicative.
\end{lm}

\begin{proof}
1.
Since $P_\bullet=X^{\bullet+1}$ is multiplicative,
the oversimplicial scheme
$P_\bullet^{(D)}$
is also multiplicative
by Corollary \ref{cordlf}.2.
By Lemma \ref{lmPDM}.3,
$P_\bullet^{(D,M)}$
is multiplicative.
Further by Corollary \ref{cordlf}.2.
$P_\bullet^{(R,M)}$
and 
$P_\bullet^{(R)}$
in the case $R$ has integral coefficients
are also multiplicative.

2.
Since $P_\bullet^{(D)},
P_\bullet^{(D,M)},
P_\bullet^{(R,M)}$
are multiplicative
and since
the diagrams
$$
\begin{CD}
D_m^{(D)}
@>>>
P_m^{(D)}
\\
@VV{\quad\quad \Box }V@VVV\\
D_n^{(D)}
@>>>
P_n^{(D)},
\end{CD}\quad
\begin{CD}
D_m^{(D,M)}
@>>>
P_m^{(D,M)}
\\
@VV{\quad\quad \Box }V@VVV\\
D_n^{(D,M)}
@>>>
P_n^{(D,M)},
\end{CD}\quad
\begin{CD}
D_m^{(R,M)}
@>>>
P_m^{(R,M)}
\\
@VV{\quad\quad \Box }V@VVV\\
D_n^{(R,M)}
@>>>
P_n^{(R,M)},
\end{CD}\quad
\begin{CD}
T_m^{(R,M)}
@>>>
P_m^{(R,M)}
\\
@VV{\quad\quad \Box }V@VVV\\
T_n^{(R,M)}
@>>>
P_n^{(R,M)}
\end{CD}
$$ are cartesian
for $[0,n]\to [0,m]$,
the oversimplicial schemes
$D_\bullet^{(D)},
D_\bullet^{(D,M)},
D_\bullet^{(R,M)}$
and
$T_\bullet^{(R,M)}$
are multiplicative.
Since the morphisms
$D_n^{(D)}\to 
D_0^{(D)}=D,
D_n^{(D,M)} \to
D_0^{(D,M)}=\widetilde D^{(M)},
D_n^{(R,M)} \to
D_0^{(R,M)}=\widetilde D^{(M)}$
and
$T_n^{(R,M)} \to
T_0^{(R,M)}=\widetilde Z^{(M)}$
are independent of
$[0]\to [0,n]$,
the multiplicative
oversimplicial schemes
$D_\bullet^{(D)},
D_\bullet^{(D,M)},
D_\bullet^{(R,M)}$ and
$T_\bullet^{(R,M)}$
are strictly multiplicative.
Similarly, $D_\bullet^{(R)}$ and
$T_\bullet^{(R)}$
are strictly multiplicative
in the case $R$ has integral coefficients.
\end{proof}

\subsection{Properties of the dilatations}

In order to study further properties
of $P_\bullet^{(R,M)}$
including the functoriality,
we give an alternative
description.
We regard
$\widetilde X^{(M)}$
as a closed subscheme of
$\widetilde X^{(M)}\times_k X^n$ 
by the section induced by 
the canonical map 
$\widetilde X^{(M)}\to X$
and the identity of
$\widetilde X^{(M)}$.
The pull-back $\widetilde R^{(M)}
=\sum_{i=1}^hm_ir_i\widetilde D_i^{(M)}$
of $R$ by $\widetilde X^{(M)}\to X$
is a Cartier divisor with integral coefficients.

\begin{lm}\label{lmPRMb}
{\rm 1.}
Let the notation be as above.
We regard
$\widetilde X^{(M)}$
as a closed subscheme of
$\widetilde X^{(M)}\times_k X^n$
embedded as the graph of
the product 
$\widetilde X^{(M)}\to X^n$
of the canonical map
and consider 
the closed immersions
\begin{equation}\begin{CD}
{\rm pr}_0^*\widetilde R^{(M)}
@>>> 
\widetilde X^{(M)}\times_k X^n
@<<< \widetilde X^{(M)}.
\end{CD}
\label{eq2.4}
\end{equation}
Let $\widetilde P^{(R)}\to
\widetilde X^{(M)}\times_k X^n$
denote the dilatation 
defined by the immersions
{\rm (\ref{eq2.4})}.
Then the canonical map
$P_n^{(R,M)}
\to
P_0^{(R,M)}\times_k X^n=
\widetilde X^{(M)}\times_k X^n$
is uniquely lifted to an open immersion
$P_n^{(R,M)}
\to \widetilde P^{(R)}$.
The image is
the complement of
the proper transforms of
the pull-backs of $D$
by the compositions 
$\widetilde X^{(M)}\times_k X^n
\to X^n\to X$with 
the $n$ projections.

{\rm 2.}
Assume further
that $R$ has integral coefficients
and consider 
the closed immersions
\begin{equation}\begin{CD}
{\rm pr}_0^*R
@>>> 
X^{n+1}@<<< X
\end{CD}
\end{equation}
Let $P\to X^{n+1}$
denote the dilatation 
defined by the immersions.
Then the canonical map
$P_n^{(R)}
\to
X^{n+1}$
is uniquely lifted to an open immersion
$P_n^{(R)}\to P$.
The image is
the complement of
the proper transforms of
the pull-backs of $D$
by the $n$ projections
$X^{n+1}\to X$
different from the $0$-th one.
\end{lm}

\begin{proof}
1.
First, we consider the case where
$R=D$.
Let $P_{n,0}^{(D)}\to X^{n+1}$ denote
the dilatation defined by the immersions
(\ref{eqPnD}) for $j=0$.
Then, since $\widetilde X^{(M)}\to X$
is flat, we have a cartesian diagram
\begin{equation}
\begin{CD}
P_{n,0}^{(D)}@<<< \widetilde P^{(D)}\\
@VV{\quad\qquad \Box }V @VVV\\
X^{n+1}@<<<\widetilde X^{(M)}\times_k X^n.
\end{CD}
\label{eqPMD0}
\end{equation}
The open subscheme $P_n^{(D)}
\subset P_{n,0}^{(D)}$
is obtained by removing
the proper transforms of
the pull-backs of $D$
by the $n$ projections
$X^{n+1} \to X$
different from the $0$-th projection.
Since $P_n^{(D,M)}$
is defined by the cartesian diagram
(\ref{eqPMD0}) with 
$P_{n,0}^{(D)}$ replaced by
$P_n^{(D)}$,
the assertion in this case $R=D$
follows.

We show the general case.
Let $p\colon \widetilde P^{(D)}
\to \widetilde X^{(M)}$
denote the projection
and regard
$\widetilde X^{(M)}$
as a closed subscheme of
$\widetilde P^{(D)}$
by the section lifting
$\widetilde X^{(M)}
\to
\widetilde X^{(M)}\times_kX^n$.
Then, the dilatation $\widetilde P^{(R)}$
is canonically identified with the
dilatation defined by the immersions
$$\begin{CD}
p^*(\widetilde R^{(M)}\sm
\widetilde M^{(M)})
@>>>
\widetilde P^{(D)}
@<<< 
\widetilde X^{(M)}.
\end{CD}$$
Thus the assertion for $R$
follows from that for $D$.

2. 
Similar to and easier than that of 1.
\end{proof}

We study the functoriality of the 
diagram (\ref{eqtPR}).
Let $f\colon X'\to X$ be a morphism of smooth
schemes over $k$
such that the inverse image
$U'=f^{-1}(U)$ is the complement
of a divisor $D'$ of $X'$ 
with simple normal crossings.
Let $D'_1,\ldots,D'_{h'}$
be the irreducible components
of $D'$ and set $f^*D_i
=\sum_{j=1}^{h'}e_{ij}D'_j$ for $i=1,\ldots,h$.

We set $R'=f^*(R)=
\sum_{i=1}^hr_i
\sum_{j=1}^{h'}e_{ij}D'_j=
\sum_{j=1}^{h'}r'_jD'_j$.
Let $M'=\sum_{j=1}^{h'}
m'_jD'_j$ be a divisor
with integral coefficients $m'_j\geqq  1$.
Then, the diagram
(\ref{eqtPR}) is defined for
$X',D',R'$ and $M'$.
We denote the schemes constructed for
$X',D',R'$ and $M'$
by putting $'$ as
$P^{\prime (R',M')}_n,
T^{\prime (R',M')}_n$ etc.

We assume that
$l_{ij}=e_{ij}m'_j/m_i$ 
is an integer for every $i=1,\ldots,h$
and $j=1,\ldots,h'$. 
Then, by the cartesian diagram (\ref{eqXMYNc}),
a canonical morphism
$\widetilde X^{\prime (M')}\to
\widetilde X^{(M)}$ is defined.
Since $m'_jr'_j=
\sum_{i=1}^hl_{ij}m_ir_i$ 
is an integer for every $j=1,\ldots,h'$,
further by Lemma \ref{lmPRMb}
and  the functoriality of dilatation,
a canonical morphism
\begin{equation}
P^{\prime (R',M')}_\bullet
\to
P^{(R,M)}_\bullet
\label{eqfunR}
\end{equation} 
of oversimplicial schemes is defined.
If $R$ and hence $R'$
have integral coefficients,
a canonical morphism
$P^{\prime (R')}_\bullet
\to
P^{(R)}_\bullet$ is defined similarly.

We study the 
schemes in (\ref{eqtPR}) more in detail.
We compute some normal bundles.
For a line bundle $L$
on a scheme $S$,
let $L^\times$
denote the complement
$L\sm S$ of the 0-section.

\begin{lm}\label{lmTR}
We have canonical isomorphisms
\begin{align}
T_XP_n
&\to TX^{n+1}/\Delta TX,
\label{eqNP}
\\
T_XP_n^{(D)}
&\to (TX^{n+1}/\Delta TX)(-D),
\label{eqNPD}
\\
T_{\widetilde X^{(M)}}P_n^{(D,M)}
&\to ((TX^{n+1}/\Delta TX)\times_X
\widetilde X^{(M)})(-\widetilde M^{(M)}),
\label{eqND}
\\
T_{\widetilde X^{(M)}}P_n^{(R,M)}
&\to ((TX^{n+1}/\Delta TX)\times_X
\widetilde X^{(M)})(-\widetilde R^{(M)})
\label{eqNR}
\end{align}
where $\Delta$ denotes the image
by the diagonal morphism.
If $R$ has integral coefficients,
we have a canonical isomorphism
\begin{equation}
T_XP_n^{(R)}
\to (TX^{n+1}/\Delta TX)(-R).
\label{eqNR1}
\end{equation}
\end{lm}

\begin{proof}
The diagonal map
$X\to X^{n+1}$
defines an exact sequence
$$0\to {\cal N}_{X/X^{n+1}}
\to
\Omega^1_{X^{n+1}/k}
\otimes_{{\cal O}_{X^{n+1}}}
{\cal O}_X\to
\Omega^1_{X/k}\to 0.$$
Since
$\Omega^1_{X^{n+1}/k}
\otimes_{{\cal O}_{X^{n+1}}}
{\cal O}_X
=
\Omega_{X/k}^{1\oplus n+1}$,
we obtain an isomorphism
(\ref{eqNP}).

In the notation of
proof of Lemma \ref{lmPRMb},
the scheme $P^{(D)}_n$
is an open subscheme of
the dilatation $P^{(D)}_{n,0}
\to X^{n+1}$.
Hence, the isomorphism
(\ref{eqnb})
gives an isomorphism
$T_XP_n^{(D)}
\to
T_XP_n\otimes
(T_DX)^{\otimes -1}$.
Thus (\ref{eqNPD}) follows from (\ref{eqNP}).

The diagram 
$$\begin{CD}
\widetilde X^{(M)}@>>>P_n^{(D,M)}\\
@VV{\quad\quad \Box }V@VVV\\
X@>>> P_n^{(D)}
\end{CD}$$
is cartesian and the vertical arrows are flat by Lemma \ref{lmPDM}.
Since the pull-back of $D$ to
$\widetilde X^{(M)}$ is $\widetilde M^{(M)}$
by Lemma \ref{lmPDM},
(\ref{eqNPD}) implies (\ref{eqND}).

In the notation of
proof of Lemma \ref{lmPRMb},
the scheme $P^{(R,M)}_n$
is an open subscheme of
the dilatation $\widetilde P^{(R)}
\to \widetilde X^{(M)}\times_kX^n$.
Hence, the isomorphism
(\ref{eqnb})
gives an isomorphism
$T_XP_n^{(R,M)}
\to
(T_{\widetilde X^{(M)}}
(\widetilde X^{(M)}\times_kX^n))
\otimes
(T_{\widetilde R^{(M)}}\widetilde X^{(M)})^{\otimes -1}$.
Since $\widetilde X^{(M)}\to X$
is flat, the normal bundle
$T_{\widetilde X^{(M)}}
(\widetilde X^{(M)}\times_kX^n)$
is the pull-back of
$T_XX^{n+1}=T_XP_n$
and is canonically isomorphic
to $TX^{n+1}/\Delta TX\times_X
\widetilde X^{(M)}$.
Thus (\ref{eqNR}) follows.
The isomorphism (\ref{eqNR1})
is proved similarly and more easily.
\end{proof}

For a subset $I\subset \{1,\ldots,h\}$,
we set $D_I=
\bigcap_{i\in I}D_i$
and $D_I^\circ
=D_I\sm
\bigcup_{i\notin I}
D_{I\amalg\{i\}}$.
Similarly, set
$\widetilde D_I^{(M)}
=
\bigcap_{i\in I}
\widetilde D_i^{(M)}$
and
$\widetilde D_I^{(M)\circ}
=\widetilde D_I^{(M)}
\sm
\bigcup_{i\notin I}
\widetilde D_{I\amalg\{i\}}^{(M)}$.

\begin{lm}\label{lmTR2}
{\rm 1.}
The normal bundle
$T_{\widetilde D_i^{(M)}}
\widetilde X^{(M)}$
is a trivial line bundle over
$\widetilde D_i^{(M)}$.
The action of
${\mathbf G}_m^h$ is the diagonal
action of that on the base
$\widetilde D_i^{(M)}$
and the multiplication
by the $i$-th component.

{\rm 2.}
For a subset $I\subset \{1,\ldots,h\}$,
the open subscheme
$\widetilde D_I^{(M)\circ}
\subset \widetilde D_I^{(M)}$
is a ${\mathbf G}_m^h$-torsor
over $D_I^\circ$ canonically isomorphic
to $\prod_{i\in I}(T_{D_i}X^\times\times_{D_i}D_I^\circ)
\times_{D_I^\circ} \prod_{i\notin I} {\mathbf G}_{m,D_I^\circ}$.
The ${\mathbf G}_m^h$-action
on the latter is by multiplication 
componentwise.

{\rm 3.}
We have a canonical isomorphism
\begin{equation}
D_n^{(D,M)}
\to
\widetilde D^{(M)}\times_D
D_n^{(D)}
\label{eqTnD}
\end{equation}
over $\widetilde D^{(M)}$
compatible with the
${\mathbf G}_m^h$-action 
defined on the base
$\widetilde D^{(M)}$ on the right hand side.

{\rm 4.}
For the scheme $T_n^{(R,M)}$
{\rm (\ref{eqTnRM})},
we have a canonical isomorphism
\begin{equation}
T_n^{(R,M)}
\to
((TX^{n+1}/\Delta TX)\times_X
\widetilde Z^{(M)})
(-\widetilde R^{(M)})
\label{eqTnR}
\end{equation}
over $\widetilde Z^{(M)}$.
The action of ${\mathbf G}_m^h$
is compatible with 
the diagonal action
on the base scheme $\widetilde Z^{(M)}$
and multiplication of
the product of
$m_ir_i$-th powers of $i$-th components
on the fiber.

If $R$ has integral coefficients,
we have a canonical isomorphism
\begin{equation}
T_n^{(R)}
\to
((TX^{n+1}/\Delta TX)(-R)\times_XZ.
\label{eqTnR0}
\end{equation}
\end{lm}

\begin{proof}
1.\ and 2.
Clear from the construction of
$\widetilde X^{(M)}$.

3. 
It follows from
the cartesian diagram
(\ref{eqPnDc}).

4.
Similarly as in the proof of
(\ref{eqNR}),
the isomorphism (\ref{eqPD})
gives a canonical isomorphism
\begin{equation}
T_n^{(R,M)}
\to
(T_{\widetilde X^{(M)}}
(\widetilde X^{(M)}\times X^n))
\times_{\widetilde X^{(M)}}
\widetilde Z^{(M)})
\otimes
(T_{\widetilde R^{(M)}}
\widetilde X^{(M)})^{\otimes -1}.
\label{eqTRn}
\end{equation}
Since the normal bundle
$T_{\widetilde X^{(M)}}
(\widetilde X^{(M)}\times X^n)$
is 
$(TX^{n+1}/\Delta TX)
\times_X\widetilde X^{(M)}$,
we obtain an isomorphism (\ref{eqTnR}).
The compatibility on
the ${\mathbf G}_m^h$-actions
follows from the assertion 1.

If $R$ has integral coefficients,
an isomorphism (\ref{eqTnR0})
is defined similarly and more easily.
\end{proof}

\begin{cor}\label{corTR}
The strictly multiplicative smooth oversimplicial schemes
$T_\bullet^{(R,M)}$
is an oversimplicial scheme
associated to the vector bundle
(Example {\rm \ref{eggpd}})
$$(TX\times_X
\widetilde Z^{(M)})
(-\widetilde R^{(M)})$$
over $\widetilde Z^{(M)}$.

If $R$ has integral coefficients,
$T_\bullet^{(R)}$
is an oversimplicial scheme
associated to the vector bundle
$TX(-R)\times_XZ$
over $Z$.
\end{cor}

\begin{proof}
By Proposition \ref{prgpd}
and Lemma \ref{lmTR2},
the strictly multiplicative
oversimplicial scheme
$T_\bullet^{(R,M)}$
is associated to a group scheme
$T_1^{(R,M)}=
((TX^2/\Delta TX)
\times_X
\widetilde Z^{(M)})
(-\widetilde R^{(M)})
=
(TX\times_X
\widetilde Z^{(M)})
(-\widetilde R^{(M)})$
over
$T_0^{(R,M)}=\widetilde Z^{(M)}$.
It is easily checked that
the group structure of
$T_1^{(R,M)}$
is defined by the 
addition on the vector bundle
$(TX\times_X
\widetilde Z^{(M)})
(-\widetilde R^{(M)})$.

In the case $R$ has integral coefficients,
the assertion on $T_\bullet^{(R)}$
is proved in the same way.
\end{proof}

The oversimplicial schemes
$P_\bullet^{(D,M)}$ and
$P_\bullet^{(R,M)}$ depend on
$M$ as follows.

\begin{lm}\label{lmPDq}
{\rm 1.}
If $M=D$, the scheme
$P^{(D,M)}_n$ 
is canonically isomorphic
to the ${\mathbf G}_m^h$-torsor
$\prod_{i=1}^hL(D^{(D)}_{n,i})^\times$ over $P^{(D)}_n$.

{\rm 2.}
Let $m_i'=l_im_i\geqq  1$ be integers
for $i=1,\ldots,h$
and set $M=m_1D_1+\cdots+m_hD_h$
and $M'=m'_1D_1+\cdots+m'_hD_h$.
Then, the diagram
\begin{equation}
\begin{CD}
P^{(R,M')}_n@>>>{\mathbf A}^h\\
@VV{\quad\quad \Box }V @VV{(t_i)\mapsto (t_i^{l_i})}V\\
P^{(R,M)}_n@>>>{\mathbf A}^h
\end{CD}
\label{eqdncP}
\end{equation}
is cartesian 
and the vertical arrows are
finite flat.
The vertical arrows are
compatible with the ${\mathbf G}_m^h$-actions
and the morphism
${\mathbf G}_m^h\to {\mathbf G}_m^h$
sending $(t_i)$ to $(t_i^{l_i})$.

If $R=D$,
the isomorphisms
{\rm (\ref{eqTnD})} for
$M$ and $M'$ make
a commutative diagram
\begin{equation}
\begin{CD}
D^{(D,M')}_n@>>>
\widetilde D^{(M')}\times_DD^{(D)}_n
\\
@VVV @VVV\\
D^{(D,M)}_n@>>>
\widetilde D^{(M)}\times_DD^{(D)}_n.
\end{CD}
\label{eqdncT}
\end{equation}
The isomorphisms
{\rm (\ref{eqTnR})} for
$M$ and $M'$ make
a commutative diagram
\begin{equation}
\begin{CD}
T^{(R,M')}_n@>>>
((TX^{n+1}/\Delta TX)\times_X
\widetilde Z^{(M')})
(-\widetilde R^{(M')})\\
@VVV @VVV\\
T^{(R,M)}_n@>>>
((TX^{n+1}/\Delta TX)\times_X
\widetilde Z^{(M)})
(-\widetilde R^{(M)}).
\end{CD}
\label{eqdncTR}
\end{equation}
\end{lm}

\begin{proof}
1.
It follows from the last assertion in Lemma \ref{lmXM}.

2.
For $n=0$ and $R=D$,
the cartesian diagram
(\ref{eqdncP}) follows from the construction of
$\widetilde X^{(M)}
=P_0^{(D,M)}$
and
$\widetilde X^{(M')}
=P_0^{(D,M')}$.
The rest follows from this
and Lemma \ref{lmPRMb}.
\end{proof}

\subsection{Ramification of Galois covering
and an additive structure}

We keep the notation fixed at the beginning of the section.
Let $G$ be a finite group and
$V\to U=X\sm D$ be a $G$-torsor.
The quotients $V^{n+1}/\Delta G$
of the $(n+1)$-fold fibered products
$V^{n+1}$ over $k$
by the diagonal action of $G$
for integers $n\geqq  0$
define a finite \'etale morphism 
$V^{\bullet+1}/\Delta G
\to U^{\bullet+1}$
of oversimplicial schemes.

\begin{lm}\label{lmDV}
The morphism
$V^{\bullet+1}/\Delta G
\to U^{\bullet+1}$
is a multiplicative morphism
of oversimplicial schemes.
\end{lm}

\begin{proof}
It suffices to show that
an additive cocartesian diagram
{\rm (\ref{eqadd})}
defines a cartesian diagram
$$\begin{CD}
V^{m+1}/\Delta G
\times_{U^{m+1}}U^{n+1}
@<<<
V^{n+1}/\Delta G\\
@VV{\quad\qquad\qquad\quad \Box }V @VVV\\
U^{n+1}
@<<<
V^{l+1}/\Delta G
\times_{U^{l+1}}U^{n+1}.
\end{CD}$$
The fibered product
$V^{m+1}/\Delta G
\times_{U^{m+1}}U^{n+1}$
is canonically identified
with the quotient
of $V^{n+1}$
by $\Delta G\times G^{l+1}$.
Similarly
$V^{l+1}/\Delta G
\times_{U^{l+1}}U^{n+1}$
is canonically identified
with the quotient
of $V^{n+1}$
by 
$G^{m+1}\times \Delta G$.
Hence it follows from
$(\Delta G\times G^{l+1})
\cap (G^{m+1}\times \Delta G)
=\Delta G$
in $G^{n+1}$.
\end{proof}

We consider the morphisms
$$\begin{CD}
P_\bullet^{(R,M)} @<<<
U^{\bullet+1}\times{\mathbf G}_m^h
@<<<
V^{\bullet+1}/\Delta G\times{\mathbf G}_m^h
\end{CD}
$$
of oversimplicial schemes over $k$.
For an integer $n\geqq  0$,
let $Q_n^{(R,M)}$ denote the normalization of
$P_n^{(R,M)}$ in the finite \'etale covering
$V^{n+1}/\Delta G\times{\mathbf G}_m^h\to
U^{n+1}\times{\mathbf G}_m^h$.
Then, by Lemma \ref{lmkey}.1,
the schemes $Q_n^{(R,M)}$ for integers $n\geqq  0$
form an oversimplicial scheme 
$Q_\bullet^{(R,M)}$ over $k$
and we obtain a cartesian diagram
\begin{equation}
\begin{CD}
Q_\bullet^{(R,M)} @<<<
V^{\bullet+1}/\Delta G\times{\mathbf G}_m^h
\\
@VV{\quad\quad\quad \Box }V @VVV\\
P_\bullet^{(R,M)} @<<<
U^{\bullet+1}\times{\mathbf G}_m^h
\end{CD}
\label{eqPQUV}
\end{equation}
of oversimplicial schemes.
For $n\geqq 0 $,
the canonical map
$Q_0^{(R,M)}=\widetilde X^{(M)}\to Q_n^{(R,M)}$
is a lifting of the lifting
$P_0^{(R,M)}=\widetilde X^{(M)}\to P_n^{(R,M)}$of the diagonal map
$X\to X^{n+1}$.

\begin{df}\label{dfbR}
Let $V\to U=X\sm D$ be a $G$-torsor 
for a finite group $G$.
Let $R=r_1D_1+\cdots+r_hD_h$
and $M=m_1D_1+\cdots+m_hD_h$
be linear combinations with
rational coefficients $r_i\geqq  1$
and 
integral coefficients $m_i\geqq  1$
such that $m_ir_i$ is an integer
for every irreducible components
$D_1,\ldots,D_h$ of $D$.
We say that the ramification
of $V$ over $U$ along $D$
(resp.\ at a point $x$ of $D$)
is {\rm bounded by} $R+$,
if the finite morphism
$$Q_1^{(R,M)}\to P_1^{(R,M)}$$
is \'etale on a neighborhood of the
image of
$Q_0^{(R,M)}=\widetilde X^{(M)}\to 
Q_1^{(R,M)}$
(resp.\ of the image 
by $Q_0^{(R,M)}=\widetilde X^{(M)}\to 
Q_1^{(R,M)}$
of the inverse image of $x$ 
by $\widetilde X^{(M)} \to X$).
\end{df}

Definition \ref{dfbR} is a non-logarithmic
variant of \cite[Definition 7.3]{Tohoku}.
We show that the condition in Definition \ref{dfbR}
is independent of the choice of a divisor $M$.
We also show that
if the ramification is bounded by $R+$
then
if the ramification is bounded by $R'+$
for $R'\geqq R$.

\begin{lm}\label{lmbR}
Let $M'=m'_1D_1+\cdots+m'_hD_h$
be linear combinations with
integral coefficients $m_i\geqq  1$
such that $m'_i=l_im_i$ is divisible by $m_i$
for every $1,\ldots,h$
and
let $R'=r'_1D_1+\cdots+r'_hD_h$
be linear combinations with
rational coefficients $r'_i\geqq  r_i$
such that $m'_ir'_i$ is an integer
for every $1,\ldots,h$.

{\rm 1.}
The cartesian diagram
$$\begin{CD}
U^{\bullet+1}\times{\mathbf G}_m^h
@<<<
V^{\bullet+1}/\Delta G\times{\mathbf G}_m^h\\
@VV{\quad\quad\qquad \Box }V @VVV\\
U^{\bullet+1}\times{\mathbf G}_m^h
@<<<
V^{\bullet+1}/\Delta G\times{\mathbf G}_m^h
\end{CD}$$
of oversimplicial schemes over $k$
where the vertical arrows are
defined by the map
on the second factor ${\mathbf G}_m^h$
sending $(t_i)$ to $(t_i^{l_i})$
is extended to
a commutative diagram
\begin{equation}
\begin{CD}
P_\bullet^{(R',M')}
@<<<
Q_\bullet^{(R',M')}\\
@VVV @VVV\\
P_\bullet^{(R,M)}@<<<
Q_\bullet^{(R,M)}
\end{CD}
\label{eqQqq'}
\end{equation}
of oversimplicial schemes over $k$.

{\rm 2.}
Let $x$ be a point of
$X^{(M)}=Q_0^{(R,M)}$ and $x'$ be a point of
$X^{(M')}=Q_0^{(R',M')}$
above $x$ by the map defined in {\rm 1.}
If $Q_1^{(R,M)}\to P_1^{(R,M)}$
is \'etale on the image
of $x$ by
$Q_0^{(R,M)}\to Q_1^{(R,M)}$,
then $Q_1^{(R',M')}\to P_1^{(R',M')}$
is \'etale on the image
of $x'$ by
$Q_0^{(R',M')}\to Q_1^{(R',M')}$.
If $R'=R$, conversely if
then $Q_1^{(R,M')}\to P_1^{(R,M')}$
is \'etale on the image
of $x'$ by
$Q_0^{(R,M')}\to Q_1^{(R,M')}$
then
$Q_1^{(R,M)}\to P_1^{(R,M)}$
is \'etale on the image
of $x$ by
$Q_0^{(R,M)}\to Q_1^{(R,M)}$.
\end{lm}

\begin{proof}
1.
It follows from
the functoriality of normalizations,
Lemma \ref{lmnsf}.1.

2.
The commutative diagram (\ref{eqQqq'})
defines a commutative diagram
$$\begin{CD}
\widetilde X^{(M')}=\ &
P_0^{(R',M')}@>>>
P_1^{(R',M')}\\
&@VVV@VVV\\
\widetilde X^{(M)}=\ &
P_0^{(R,M)}@>>>
P_1^{(R,M)}
\end{CD}$$
and its lifting
$$\begin{CD}
\widetilde X^{(M')}=\ &
Q_0^{(R',M')}@>>>
Q_1^{(R',M')}\\
&@VVV@VVV\\
\widetilde X^{(M)}=\ &
Q_0^{(R,M)}@>>>
Q_1^{(R,M)}
\end{CD}$$
If  $Q_1^{(R,M)}\to P_1^{(R,M)}$
is \'etale on the image of $x$,
then
$Q_1^{(R',M')}\to P_1^{(R',M')}$
is \'etale on the image of $x'$ 
by Lemma \ref{lmnsf}.2.

Assume $R'=R$.
Then, the diagrams above
satisfy the conditions (1)--(3) of Lemma \ref{lmet}.
Hence, if $Q_1^{(R,M')}\to P_1^{(R,M')}$
is \'etale on the image of $x'$,
then
$Q_1^{(R,M)}\to P_1^{(R,M)}$
is \'etale on the image of $x$
by Lemma \ref{lmet}.
\end{proof}

\begin{thm}\label{thmmain}
Let $X$ be a smooth separated scheme
over a perfect field $k$ and
$D$ be a divisor with normal crossings.
Let $V\to U$ be a $G$-torsor
on the complement $U=X\setminus D$
for a finite group $G$.
Let $R=r_1D_1+\cdots+r_hD_h$
and $M=m_1D_1+\cdots+m_hD_h$
be linear combinations with
rational coefficients $r_i\geqq  1$
and 
integral coefficients $m_i\geqq  1$
such that $m_ir_i$ is an integer
for every irreducible components
$D_1,\ldots,D_h$ of $D$.
We consider the commutative
diagram
$$
\begin{CD}
Q_\bullet^{(R,M)} @<<<
V^{\bullet+1}/\Delta G\times{\mathbf G}_m^h
\\
@VVV @VVV\\
P_\bullet^{(R,M)} @<<<
U^{\bullet+1}\times{\mathbf G}_m^h
\end{CD}
\leqno{(\ref{eqPQUV})}
$$
of oversimplicial schemes over $k$.

Assume that the ramification
of $V$ over $U$ along $D$
is bounded by $R+$.
For each integer $n\geqq  0$,
let $W_n^{(R,M)}$
be the largest open subscheme of
$Q_n^{(R,M)}$ \'etale over $P_n^{(R,M)}$.
Then 
$W_n^{(R,M)}$
form an oversimplicial
open subscheme 
$W_\bullet^{(R,M)}$
of
$Q_\bullet^{(R,M)}$.
Further,
the morphism
$W_\bullet^{(R,M)}
\to
P_\bullet^{(R,M)}$ is multiplicative
and 
the oversimplicial scheme 
$W_\bullet^{(R,M)}$ is multiplicative.
\end{thm}

\begin{proof}
We apply Lemma \ref{lmkey}.2
to the \'etale and multiplicative
morphism of oversimplicial schemes
$V^{\bullet+1}/\Delta G
\times_k{\mathbf G}_m^h
\to P_\bullet^{(R,M)}$.
By Lemma \ref{lmPRM}.3,
the assumption (1) in Lemma \ref{lmkey}.2
that
$P_n^{(R,M)}\to P_m^{(R,M)}$
are smooth for injections
$[0,m]\to [0,n]$ is satisfied.

Since $V/G\to U$ is an isomorphism,
the morphism
$Q_0^{(R,M)}\to P_0^{(R,M)}=
\widetilde X^{(M)}$
is an isomorphism.
The assumption
that the ramification is
bounded by $R+$
means that the assumption (2) in Lemma \ref{lmkey}.2
that the morphism
$V/G \times{\mathbf G}_m^h
\to (V\times V)/\Delta G\times
{\mathbf G}_m^h$
is extended to
$\widetilde X^{(M)}=W_0^{(R,M)}\to W_1^{(R,M)}$
is also satisfied.

Thus by 
Lemma \ref{lmkey}.2,
$W_n^{(R,M)}$
form an oversimplicial
open subscheme 
$W_\bullet^{(R,M)}$
of
$Q_\bullet^{(R,M)}$
and the morphism
$W_\bullet^{(R,M)}
\to
P_\bullet^{(R,M)}$ is multiplicative.
Since
$P_\bullet^{(R,M)}$ is multiplicative
by Lemma \ref{lmprod}.1,
the oversimplicial scheme 
$W_\bullet^{(R,M)}$ is multiplicative
by Lemma \ref{lmgrpd}
(2)$\Rightarrow$(1).
\end{proof}

\begin{cor}\label{cormain}
Define an oversimplicial
scheme 
$E_\bullet^{(R,M)}$
by the cartesian diagram
$$\begin{CD}
E_\bullet^{(R,M)}
@>>>W_\bullet^{(R,M)}\\
@VV{\quad\quad \Box }V @VVV\\
T_\bullet^{(R,M)}
@>>>P_\bullet^{(R,M)}.
\end{CD}$$
Then, the oversimplicial
scheme 
$E_\bullet^{(R,M)}$
is strictly multiplicative
and the left vertical arrow
$E_\bullet^{(R,M)}\to
T_\bullet^{(R,M)}$
of oversimplicial schemes
is \'etale and multiplicative.

The multiplicative oversimplicial scheme
$E_\bullet^{(R,M)}$
is associated to 
a smooth group scheme
$E_1^{(R,M)}$
over
$E_0^{(R,M)}=\widetilde Z^{(M)}$.
Further, the multiplicative morphism
$E_\bullet^{(R,M)}\to
T_\bullet^{(R,M)}$
is associated
to an \'etale morphism
\begin{equation}
E_1^{(R,M)}=E^{(R,M)}\to 
T_1^{(R,M)}=T^{(R,M)}
\label{eqisog}
\end{equation}
of smooth group schemes over
$\widetilde Z^{(M)}$.
\end{cor}

\begin{proof}
Since the morphism
$W_\bullet^{(R,M)}
\to P_\bullet^{(R,M)}$
is an \'etale multiplicative morphism
of multiplicative oversimplicial schemes
by
Theorem \ref{thmmain},
its base change
$E_\bullet^{(R,M)}\to
T_\bullet^{(R,M)}$ is
also an \'etale multiplicative morphism
of multiplicative oversimplicial smooth schemes.
Since
$T_\bullet^{(R,M)}$ is strictly multiplicative
by Lemma \ref{lmprod}.2,
$E_\bullet^{(R,M)}$
is also strictly multiplicative.
Hence
by Proposition \ref{prgpd}.2,
the oversimplicial scheme
$E_\bullet^{(R,M)}$
is associated to
a smooth group scheme 
$E_1^{(R,M)}$ over $E_0^{(R,M)}=
\widetilde Z^{(M)}$
and 
the morphism $E_\bullet^{(R,M)}\to
T_\bullet^{(R,M)}$ is
associated to
an \'etale morphism 
$E_1^{(R,M)}\to T_1^{(R,M)}$ of
group schemes over $\widetilde Z^{(M)}$.
\end{proof}

\begin{pr}\label{prmain}
{\rm 1.}
There exists a unique open
subgroup scheme
$$E^{(R,M)0}
\subset E^{(R,M)}$$
such that for every point of
$x\in \widetilde Z^{(M)}$,
the fiber
$E^{(R,M)0}\times_{\widetilde Z^{(M)}}x$
is the connected component of
$E^{(R,M)}\times_{\widetilde Z^{(M)}}x$
containing the unit section.
The group scheme
$E^{(R,M)0}$
over $\widetilde Z^{(M)}$
is commutative and killed by $p$.
The restriction of {\rm (\ref{eqisog})}
\begin{equation}
E^{(R,M)0}
\to 
T^{(R,M)}
=
(TX\times_X\widetilde Z^{(M)})
(-{\widetilde R^{(M)}})
\label{eqisog0}
\end{equation}
is an \'etale and surjective 
morphism of commutative group schemes.
The kernel of {\rm (\ref{eqisog0})}
is an \'etale commutative group scheme
annihilated by $p$.

{\rm 2.}
Let $M'$ be as in Lemma {\rm \ref{lmbR}}.
Then the diagram 
\begin{equation}
\begin{CD}
E^{(R,M')0}
@>>> 
T^{(R,M')}
@>>>\widetilde Z^{(M')}
\\
@VV{\quad\quad\quad \Box }V@VV{\quad\quad \Box }V@VVV\\
E^{(R,M)0}
@>>> 
T^{(R,M)}
@>>>\widetilde Z^{(M)}
\end{CD}
\label{eqTqq'}
\end{equation}
is cartesian.
The morphism
$E^{(R,M)0}
\to T^{(R,M)}$
is finite if and only if
$E^{(R,M')0}
\to T^{(R,M')}$ is finite.
\end{pr}

\begin{proof}
1.
Since $E^{(R,M)}$
is a smooth group scheme,
the open
subgroup scheme
$E^{(R,M)0}$ exists by \cite[Theorem 3.10 (i)
$\Rightarrow$(iv)]{SGA3}.
By Lemma \ref{lmExt},
$E^{(R,M)0}$  is commutative
and killed by $p$.
Consequently,
the kernel
${\rm Ker}(E^{(R,M)0}\to T^{(R,M)})$
of an \'etale morphism
is an \'etale commutative group scheme
annihilated by $p$.
Since
$E^{(R,M)0}$ is an open subscheme of
$E^{(R,M)}$,
the morphism (\ref{eqisog0}) is \'etale.
Since the fibers of
the vector bundle
$T^{(R,M)}$ are connected,
it is surjective.
The description of
$T^{(R,M)}$ is given in (\ref{eqTnR}).

2.
The commutative diagram (\ref{eqQqq'})
defines a commutative diagram
\begin{equation}
\begin{CD}
E^{(R,M')}
@>>> 
T^{(R,M')}@>>>\widetilde Z^{(M')}
\\
@VVV@VVV@VVV\\
E^{(R,M)}
@>>> 
T^{(R,M)}@>>>\widetilde Z^{(M)}
\end{CD}
\label{eqTqq'0}
\end{equation}
without $0$.
The right square is cartesian
by Lemma \ref{lmPDq}.2
and hence the left vertical arrow
induces an open immersion
$E^{(R,M')}
\to E^{(R,M)}\times
_{\widetilde D^{(M)}}\widetilde D^{(M')}$.
Thus, the left square in (\ref{eqTqq'}) is 
also cartesian.

Since (\ref{eqTqq'}) is cartesian,
the finiteness of the morphism
$E^{(R,M)0}
\to
T^{(R,M)}$
implies the finiteness of 
$E^{(R,M')0}
\to
T^{(R,M')}$.
Since $\widetilde Z^{(M')}
\to \widetilde Z^{(M)}$ is finite surjective,
the converse holds.
\end{proof}

If $R$ has integral coefficients,
we can work with $P_\bullet^{(R)}$
without introducing an auxiliary divisor $M$,
which is much easier.
In fact, in this case we may take $M=D$
and then the vertical arrows
in the cartesian diagram
(\ref{eqPRRM}) are smooth.

\begin{df}\label{dfnd}
We say that the ramification of
$V$ over $U$
is {\em non-degenerate} along $D$ at multiplicity $R$
if the \'etale surjective morphism
{\rm (\ref{eqisog0})} is {\rm finite}. 
\end{df}

Definition \ref{dfnd} is
a non-logarithmic variant of \cite[Definition 2.27]{mu}.
Contrary to \cite[Theorem 4.1]{Kato} for the logarithmic version,
Yatagawa shows that
there exist rank 1 sheaves on surfaces
such that the ramification along the
boundary are {\em not} non-degenerate
even after any successive blow-up.
She actually has a complete list of such cases.

Proposition \ref{prmain}.2 implies that 
Definition \ref{dfnd} is independent
of the choice of $M$
such that $m_ir_i$ is an integer for every $i=1,\ldots,h$.
If $D$ is smooth and if
$r=r_1>1$,
the right vertical arrow 
$\widetilde Z^{(M')}\to \widetilde Z^{(M)}$
of (\ref{eqTqq'})
is an isomorphism by Lemma \ref{lmPDq}.2.\
and $\widetilde Z^{(M)}$ is canonically identified
with $T_DX^\times$ by Lemma \ref{lmPDq}.1.
We will study the relation of Definition \ref{dfnd}
with pull-back in Proposition \ref{prfun}.

\begin{ex}\label{egAS}
{\rm (cf.\ \cite[Proposition 11.8]{aml})
Let $X$ be the affine plane ${\mathbf A}^2_k=
{\rm Spec}\ k[x,y]$
and $U$ be the complement of
the smooth divisor $D$ defined by $x=0$.
Let $r> 1$ be an integer and set $R=rD$.
Then, the dilatation $P_1^{(R)}
={\rm Spec}\ k[x,y,x',y']
\left[\dfrac{x'-x}{x^r},\dfrac{y'-y}{x^r},\dfrac x{x'}\right]$
is isomorphic to
${\rm Spec}\ k[x,y,u,v]
\left[\dfrac1{1+ux^{r-1}}\right]$
by $u\mapsto \dfrac{x'-x}{x^r},
v\mapsto\dfrac{y'-y}{x^r}$.
Set $G= {\mathbf F}_p$.

{\rm 1.}
Let $n\geqq  1$ be an integer prime to $p$
and $V$ be the $G$-torsor
over $U$ defined by the Artin-Schreier equation
$t^p-t=\dfrac1{x^n}$.
Then, 
$V\times V/\Delta G$
is the $G$-torsor
over $U\times U$ defined by the Artin-Schreier equation
$t^p-t=\dfrac1{x^{\prime n}}-\dfrac1{x^n}
=\dfrac1{x^n}
\left(\dfrac 1{(1+ux^{r-1})^n}-1\right)$.
For $r=n+1$,
the right hand side is a regular function
on $P_1^{(R)}$ and hence
the ramification of $V$ over $U$
along $D$ is bounded by $R+$.
Further,  the right hand side
is congruent to
$-nu$ modulo $x$.
Hence, 
$E^{(R)}\to T^{(R)}$
is the $G$-torsor defined by
the Artin-Schreier equation
$t^p-t=-nu$
and the ramification of $V$ over $U$
along $D$ is {\it non-degenerate}.
By identifying $T^{(R)}$
with $TX(-R)\times_XD$,
the linear form $-nu$
is identified with the differential form
$-n \dfrac{dx}{x^{n+1}}=d\dfrac 1{x^n}$.

{\rm 2.}
Let $n\geqq  1$ be an integer divisible $p$
and $V$ be the $G$-torsor
over $U$ defined by the Artin-Schreier equation
$t^p-t=\dfrac y{x^n}$.
Then, 
$V\times V/\Delta G$
is the $G$-torsor
over $U\times U$ defined by the Artin-Schreier equation
\begin{equation}t^p-t=\dfrac{y'}{x^{\prime n}}-\dfrac y{x^n}
=\dfrac{y'}{x^n}
\left(\dfrac 1{(1+ux^{r-1})^n}-1\right)
+\dfrac {vx^r}{x^n}.
\label{eqASp}
\end{equation}
For $r=n$,
the right hand side is a regular function
on $P_1^{(R)}$ and hence
the ramification of $V$ over $U$
along $D$ is bounded by $R+$.

We put $n=n'n''$ where $n'$ is 
the prime-to-$p$ part of $n$
and $n''$ is a power of $p$.
Then, 
we have 
$(1+ux^{n-1})^n
\equiv 1+n'u^{n''}x^{(n-1)n''}
\bmod x^{2(n-1)n''}$.
Since $n\leqq (n-1)n''$
and the equality is equivalent to
$n=n''=2$,
the right hand side of (\ref{eqASp})
is congruent to
$v$ modulo $x$ if $n\neq 2$
and to
$yu^2+v$ modulo $x$ if $n=2$.
Hence, 
$E^{(R)}\to T^{(R)}$
is the $G$-torsor defined by
the Artin-Schreier equation
$t^p-t=v$ if $n\neq 2$
and $t^p-t=\sqrt y u+v$ 
defined over the radicial covering of
degree $2$
of $D$ if $n=2$.
Consequently
the ramification of $V$ over $U$
along $D$ is {\it non-degenerate}.
By identifying $T^{(R)}$
with $TX(-R)\times_XD$,
the linear form $v$
is identified with the differential form
$\dfrac{dy}{x^n}=d\dfrac y{x^n}$.}
\end{ex}

\subsection{Ramification and the cotangent bundle}

We keep the notation
$V\to U=X\sm D$ over $k$
in the previous subsection.
We relate ramification
to the cotangent bundle
using the finite \'etale surjective morphism
{\rm (\ref{eqisog1})}, 
assuming that the ramification 
is non-degenerate along $D$.

Let $M$ and $R$ be
as in the first paragraph of Section 2
and
assume that the ramification of
$V\to U$ along $D$
is bounded by $R+$ 
(Definition \ref{dfbR}) and
is non-degenerate at the multiplicity $R$
(Definition \ref{dfnd}) .
We fix $M$ 
and write the finite \'etale morphism (\ref{eqisog0}) 
of smooth group schemes over
$\widetilde Z=\widetilde Z^{(M)}$ as
\begin{equation}
E^{(R)0}
\to 
T^{(R)}=(TX\times_X\widetilde Z)
(-{\widetilde R})
\label{eqisog1}
\end{equation}
by abuse of notation.

The kernel $\widetilde G^{(R)}$
of the finite \'etale morphism
(\ref{eqisog1})
is a commutative finite
\'etale group scheme killed
by $p$.
By Proposition \ref{prExt}
and Lemma \ref{lmconn},
the extension
\begin{equation}
\begin{CD}
0@>>>\widetilde G^{(R)}
@>>>
E^{(R)0}
@>>>
T^{(R)}
@>>> 0
\end{CD}
\label{eqGET}
\end{equation}
defines a closed immersion
\begin{equation}
\widetilde G^{(R)\vee}
\to T^{(R)\vee}
\label{eqcf}
\end{equation}
of a finite \'etale
${\mathbf F}_p$-vector space scheme
to the dual vector bundle
defined over a finite radicial
covering 
$F^n\colon \widetilde Z^{(p^{-n})}
\to \widetilde Z$.
Here and in the following,
for a scheme over $S$ over $k$,
we define a scheme
$S^{(p^{-n})}$ over $k$
and a radicial covering
$F^n\colon S^{(p^{-n})}\to S$
by the diagram
$$\begin{CD}
S^{(p^{-n})}@>>> S@>{F_S^n}>>S\\
@VV{\quad\quad \Box }V @VVV @VVV\\
k@>>{F_k^{-n}}> k@>>{F_k^n}>k
\end{CD}$$
as follows.
The right square is
the usual commutative diagram
with the $n$-th powers of
the absolute Frobeniuses $F_k$ and $F_S$.
The left square is the base change
by the {\it inverse} of $F_k^n$.
The composition of
the lower line is the identity
and that of the upper line defines
$F^n\colon S^{(p^{-n})}\to S$.

\begin{df}\label{dfcf}
Assume that the ramification of
$V$ over $U$ along $D$
is bounded by $R+$.
We call the injection
{\rm (\ref{eqcf})}
of commutative group schemes
defined over 
$F^n\colon \widetilde Z^{(p^{-n})}
\to \widetilde Z$ for a sufficiently large integer $n$
the {\rm characteristic form}
of $V$ over $U$
at multiplicity $R$ and write
\begin{equation}
{\rm Char}_R(V/U)\colon
\widetilde G^{(R)\vee}
\to T^{(R)\vee}
=
(T^*X\times_X\widetilde Z)
({\widetilde R}).
\label{eqdfcf}
\end{equation}
\end{df}

The logarithmic variant of
the characteristic form
for an abelian covering is
defined and studied by Kato
in \cite{rank1} and \cite{Kato}
and called the refined Swan conductor.
If $M'$ is a multiple of $M$,
the characteristic form on $\widetilde Z^{(M')}$
is the pull-back of that on $\widetilde Z^{(M)}$.
For the Artin-Schreier covering $V\to U$
in Example \ref{egAS}.1,
the characteristic form is the morphism
sending the generator $1$ of $\widetilde G^{(R)\vee}
={\mathbf F}_p$ to the differential form
$-ndx\otimes x^{-(n+1)}$.
In the case $p=n=2$ in Example \ref{egAS}.2,
it is the morphism 
sending $1$ to $(\sqrt y dx+dy)\otimes x^{-2}$ defined
on the radicial covering $D^{(2^{-1})}\to D$.

We study the functoriality of
characteristic form.
Let $X'$ be another smooth scheme
over $k$
and $f\colon X'\to X$ be
a morphism over $k$.
Assume that $f^{-1}(U)$ is 
the complement $U'=X'\sm D'$
of a divisor $D'\subset X'$
with normal crossings.
Let $D'_1,\ldots,D'_{h'}$ be the
irreducible components of $D'$
and set $f^*D_i
=\sum_{j=1}^{h'}e_{ij}D'_j$ for $i=1,\ldots,h$.

We set $R'=f^*(R)=
\sum_{i=1}^hr_i
\sum_{j=1}^{h'}e_{ij}D'_j=
\sum_{j=1}^{h'}r'_jD'_j$.
Let $M'=\sum_{j=1}^{h'}
m'_jD'_j$ be a divisor
with integral coefficients $m'_j\geqq  1$.
We assume that
$l_{ij}=e_{ij}m'_j/m_i$ 
is an integer for every $i=1,\ldots,h$
and $j=1,\ldots,h'$. 
Then a canonical morphism 
$P^{\prime (R',M')}_\bullet
\to
P^{(R,M)}_\bullet$
(\ref{eqfunR}) is defined. 
We also fix $M'$ and 
we write $\widetilde Z'=\widetilde Z^{\prime (M')}$ etc.\ by abuse of notation.

\begin{df}\label{dfnc}
Assume that the ramification of
$V$ over $U$ along $D$
is bounded by $R+$
and is non-degenerate at multiplicity $R$.
Let $f\colon X'\to X$
be a morphism of smooth schemes
over $k$
such that  $f^{-1}(U)$ is 
the complement $U'=X'\sm D'$
of a divisor $D'\subset X'$
with simple normal crossings.

We say that 
$f\colon X'\to X$ is {\rm non-characteristic} 
with respect to the ramification of
$V\to U$ along $D$ at multiplicity $R$,
if the composition
\begin{equation}
\begin{CD}
\widetilde G^{(R)\vee}\times_{\widetilde Z}
\widetilde Z'
@>>> 
T^{(R)\vee}\times_{\widetilde Z}
\widetilde Z'
@>>>
T^{\prime(R')\vee}\\
@.\|@.\| \\
@.
(T^*X\times_X\widetilde Z)(\widetilde R)
\times_{\widetilde Z}
\widetilde Z'
@.
(T^*X'\times_{X'}\widetilde Z')(\widetilde R')
\end{CD}
\label{eqncf}
\end{equation}
defined over $F^n\colon 
\widetilde Z^{\prime (p^{-n})}
\to \widetilde Z'$
is {\rm injective}.
\end{df}

Proposition \ref{prfun}.2 (2)$\Rightarrow$(1)
below implies that
a smooth morphism $X'\to X$ is non-characteristic
since the canonical morphism $T^{\prime (R')}\to 
T^{(R)}\times_{\widetilde Z}\widetilde Z'$
is a surjection of
vector bundles in that case.

\begin{lm}\label{lmnc}
Assume that the ramification of
$V\to U$ along $D$ is 
bounded by $R+$ and
is non-degenerate
at multiplicity $R$
and let $f\colon X'\to X$ be
a morphism of smooth schemes over $k$
as above.
Then, the following  conditions are equivalent:

{\rm (1)}
$f\colon X'\to X$ is non-characteristic
with respect to the ramification of
$V\to U$ along $D$ at multiplicity $R$.

{\rm (2)}
For every point
$x$ of $\widetilde Z'$,
the fiber
$(E^{(R)0}
\times_{T^{(R)}}T^{(R')})
\times_{\widetilde Z'}x$
is a connected group scheme.
\end{lm}

\begin{proof}
By the definition of characteristic form,
the composition
$\widetilde G^{(R)\vee}\times_{\widetilde Z}
\widetilde Z'
\to 
T^{(R)\vee}\times_{\widetilde Z}
\widetilde Z'
\to 
T^{(R')\vee}$
corresponds to
the extension
$E^{(R)0}
\times_{T^(R)}T^{(R')}$ of
$T^{(R')}$ by
$G^{(R)}\times_{\widetilde Z}
\widetilde Z'$
obtained as the pull-back 
by $T^{(R')}
\to T^{(R)}\times_{\widetilde Z}
\widetilde Z'$
of the base change 
of (\ref{eqGET})
by $\widetilde Z'\to 
\widetilde Z$.
Thus, the injectivity of
the composition
$\widetilde G^{(R)\vee}\times_{\widetilde Z}
\widetilde Z'
\to 
T^{(R')\vee}$
is equivalent to the condition (2)
by Lemma \ref{lmconn}.
\end{proof}

We define a finite \'etale $G$-torsor
$V'\to U'$
by the cartesian diagram 
$$\begin{CD}
X'@<<< U'@<<< V'\\
@VfV{\quad\quad \Box }V @VV{\quad\quad \Box }V @VVV\\
X@<<< U@<<< V.
\end{CD}$$
Recall that $R'=f^*(R)$.
Then, by the functoriality of
dilatations (\ref{eqfunR}), we obtain a cartesian diagram
\begin{equation}
\begin{CD}
P_\bullet^{\prime (R',M')}@<<< 
U^{\prime \bullet+1}\times{\mathbf G}_m^{h'}@<<< 
V^{\prime \bullet+1}/\Delta G\times
{\mathbf G}_m^{h'}\\
@VV{\quad\quad\quad \Box }V 
@VV{f^{\bullet+1}\times (\prod_js_j^{l_{ij}}) \quad \Box }V 
@VVV\\
P_\bullet^{(R,M)}@<<< 
U^{\bullet+1}\times{\mathbf G}_m^h@<<< 
V^{\bullet+1}/\Delta G\times{\mathbf G}_m^h.
\end{CD}
\label{eqXX'R}
\end{equation}

\begin{pr}\label{prfun}
Assume that the ramification of
$V\to U$ along $D$ is 
bounded by $R+$.
Let $f\colon X'\to X$ be
a morphism of smooth schemes over $k$
and $M$ and $M'$ be as above.

{\rm 1.}
The ramification of $V'$ over $U'$
along $D'$ is bounded by $R'+$
and, for the \'etale oversimplicial scheme
$W^{(R,M)}_\bullet$ over
$P^{(R,M)}_\bullet$
defined in Theorem {\rm \ref{thmmain}},
we have an open immersion
\begin{equation}
\begin{CD}
W^{(R,M)}_\bullet
\times_
{P^{(R,M)}_\bullet}
P^{\prime (R',M')}_\bullet
@>>>
W^{\prime (R',M')}_\bullet
\end{CD}
\label{eqfun}
\end{equation}
of multiplicative 
oversimplicial schemes
\'etale over
$P^{\prime (R',M')}_\bullet$.

{\rm 2.}
Further assume that the ramification of
$V$ over $U$
is non-degenerate along $D$ at multiplicity $R$.
Then, the following conditions are equivalent.

{\rm (1)}
$f\colon
X'\to X$ 
is non-characteristic with respect to the
ramification of $V\to U$ along $D$ at multiplicity $R$.

{\rm (2)}
The open immersion {\rm (\ref{eqfun})}
induces an isomorphism
\begin{equation}
\begin{CD}
E^{\prime (R')0}
@>>>
E^{(R)0}\times_{T^{(R)}}T^{\prime (R')}.
\end{CD}
\label{eqfunE}
\end{equation}

{\rm (3)}
The ramification of
$V'$ over $U'$ along $D'$
is non-degenerate at multiplicity $R'$
and there exists an isomorphism
$\widetilde G^{(R)\vee}\times_{\widetilde Z}
\widetilde Z'\to
\widetilde G^{\prime(R')\vee}$
that makes
a commutative diagram
\begin{equation}
\begin{CD}
\widetilde G^{\prime(R')\vee}
@>{{\rm Char}_{R'}(V'/U')}>>
T^{\prime (R')\vee}
\\
@AAA@AAA\\
\widetilde G^{(R)\vee}\times_{\widetilde Z}
\widetilde Z'
@>{{\rm Char}_R(V/U)}>>
T^{(R)\vee}\times_{\widetilde Z}
\widetilde Z'.
\end{CD}
\label{eqncch}
\end{equation}
\end{pr}

\begin{proof}
1.
It suffices to apply
the functoriality
Lemma \ref{lmnsf}.2 to 
the diagram (\ref{eqXX'R}).

2.
The open immersion (\ref{eqfun})
induces an open immersion
$E^{(R,M)}
\times_
{T^{(R,M)}}T^{\prime (R',M')}
\to
E^{\prime (R',M')}$
and hence an open immersion
(\ref{eqfunE}).

By Lemma \ref{lmnc},
the condition (1) is equivalent
to that,
for every point $x'$ of
$\widetilde Z'$
the fiber
$(E^{(R)0}\times_{T^{(R)}}T^{\prime (R')})
\times_{\widetilde Z'}x'$
is connected.
Hence, by the uniqueness of $E^{\prime (R')0}$
in Proposition \ref{prmain}.1,
the condition (1) is equivalent to
the condition (2).

An isomorphism
$\widetilde G^{(R)\vee}\times_{\widetilde Z}
\widetilde Z'\to
\widetilde G^{\prime(R')\vee}$
making the diagram (\ref{eqncch})
is equivalent to an isomorphism
$$
\begin{CD}
0@>>>
\widetilde G^{\prime(R')}
@>>>
E^{\prime (R')0}
@>>>
T^{\prime (R')}
@>>>0\\
@.@VVV@VV{\rm(\ref{eqfunE})}V@|@.
\\
0@>>>
\widetilde G^{(R)}
\times_{\widetilde Z}
\widetilde Z'
@>>>
E^{(R)0}
\times_
{T^{(R)}}T^{\prime (R')}
@>>>
T^{\prime (R')}
@>>>0
\end{CD}$$
of extensions
over $\widetilde Z'$.
Hence the condition (3)
is equivalent to that
the \'etale morphism
$E^{\prime(R')0} \to T^{\prime(R')}$
is finite and that
(\ref{eqfunE}) is an isomorphism.
By the assumption that
the ramification of
$V$ over $U$
is non-degenerate along $D$ at multiplicity $R$,
the \'etale morphism
$E^{(R)0} \to T^{(R)}$
is finite
and the target of
(\ref{eqfunE}) 
is finite over $T^{\prime(R')}$.
Thus, the condition (3)
is also equivalent to 
the condition (2).
\end{proof}

\begin{cor}\label{corfun}
Assume that the ramification of
$V\to U$ along $D$ is 
bounded by $R+$ and
is non-degenerate
at multiplicity $R$
as above.
Let $X''\to X'\to X$ be
morphisms of smooth schemes
over $k$ such
that the inverse images of $U$ are 
the complement of divisors 
of $X'$ and of $X''$ with
simple normal crossings.
Then, the following conditions
are equivalent.

{\rm (1)}
The composition 
$X''\to X$ is non-characteristic
with respect to the ramification of
$V$ over $U$ along $D$
at multiplicity $R$.

{\rm (2)}
$X'\to X$ is non-characteristic
with respect to the ramification of
$V$ over $U$ along $D$
at multiplicity $R$
on a neighborhood of the image
of $X''\to X'$
and 
$X''\to X'$ is non-characteristic
with respect to the ramification of
$V'$ over $U'$ along $D'$
at multiplicity $R'$.
\end{cor}

\begin{proof}
We consider 
the open immersions 
\begin{equation}
\begin{CD}
E^{\prime\prime (R'')0}
@>>>
E^{\prime(R')0}
\times_
{T^{\prime(R')}}
T^{\prime \prime(R'')}
@>>>
E^{(R)0}
\times_
{T^{(R)}}T^{\prime \prime(R'')}
\end{CD}
\label{eqfun2}
\end{equation}
over $\widetilde Z''$
in (\ref{eqfunE})
for $X''\to X'\to X$.
By Proposition \ref{prfun}.2,
the condition (1) is
equivalent to
that the composition of (\ref{eqfun2})
is an isomorphism.
Since the morphisms in (\ref{eqfun2})
are open immersions,
it is equivalent to
that the first map is an isomorphism
and that
 (\ref{eqfunE})
is an isomorphism
on the image of $X''\to X'$.
Hence the assertion follows.
\end{proof}

We study the restrictions to curves.
Let $\Sigma\subset TX$
denote the union of the images of
the hyperplane bundles defined as the zero locus
of the image by the injection
${\rm Char}_R(V/U)\colon
\widetilde G^{(R)\vee}
\to T^{(R)\vee}$
of non-zero sections
defined over $F^n\colon \widetilde Z^{(p^{-n})}
\to \widetilde Z$.

\begin{lm}\label{lmcut}
Assume that the ramification of
$V\to U$ along $D$ is 
bounded by $R+$ and
is non-degenerate
at multiplicity $R$
as above.
Let $C$ be a smooth curve
and
$f\colon C\to X$ be a morphism over $k$
such that the pull-back $f^*D$ is a divisor of $C$.
Then, for $x\in f^*(D)$,
the following conditions are equivalent.

{\rm (1)}
$f\colon C\to X$
is non-characteristic 
with respect to the ramification of
$V$ over $U$ along $D$
at multiplicity $R$
on a neighborhood of $x$.

{\rm (2)}
The image of the morphism
$f_*\colon T_xC\to T_{f(x)}X$
on the tangent space
is not contained in
$\Sigma$ defined above.
\end{lm}

\begin{proof}
We apply Proposition \ref{prfun}.2
to $C\to X$.
Since the right vertical arrow in
(\ref{eqncch}) is the dual of
$f_*\colon T_xC\to T_{f(x)}X$,
the condition (2) is equivalent
to that in Proposition \ref{prfun}.2.
\end{proof}

Assume that $D$ is irreducible
and that a finite covering $V\to U$
is wildly ramified along $D$.
Then, it follows immediately from Lemma
\ref{lmcut} that there exists
a curve $C$ in $X$ meeting
$D$ transversely at a closed point
$x$ such that the pull-back of $V\to U$
to $C$ is wildly ramified at $x$.

\begin{cor}\label{corfuncut}
Assume that the ramification of
$V\to U$ along $D$ is 
bounded by $R+$ and
is non-degenerate
at multiplicity $R$
as above.
Let 
$f\colon X'\to X$ be a morphism
of smooth schemes such
that the inverse image $U'$ of $U$ is
the complement of a divisor 
$D'$ of $X'$ with
simple normal crossings.
Then, 
the following conditions are equivalent.

{\rm (1)}
$f\colon X'\to X$
is non-characteristic
with respect to the ramification of
$V$ over $U$ along $D$
at multiplicity $R$.

{\rm (2)}
For every closed point
$x'$ of $D'$,
there exists a smooth curve $C$ defined on 
a neighborhood of $x'$ in $X'$
meeting each component of $D'$ transversally at $x'$
such that the composition 
$C'\to X'\to X$ is non-characteristic
with respect to the ramification of
$V$ over $U$ along $D$
at multiplicity $R$.
\end{cor}

\begin{proof}
(1)$\Rightarrow$(2):
By Proposition \ref{prfun}.2,
the ramification of the
pull-back $V'=V\times_UU'$
along $D'$ is non-degenerate
at multiplicity $R'$.
By Lemma \ref{lmcut} (2)$\Rightarrow$(1), 
for every closed point
$x'$ of $D'$,
there exists a smooth curve $C$ defined on 
a neighborhood of $x'$ in $X'$
meeting  each component of $D'$ transversally at $x'$
such that the composition 
$C'\to X'$ is non-characteristic
with respect to $V'$ over $U'$.
By Corollary \ref{corfun} (2)$\Rightarrow$(1),
the composition
$C'\to X$ is non-characteristic
with respect to $V$ over $U$.

(2)$\Rightarrow$(1):
It follows from 
Corollary \ref{corfun}  (1)$\Rightarrow$(2).
\end{proof}

\subsection{Ramification groups}\label{sslf}

We briefly recall from \cite{AS1},
\cite{AS2} and \cite{mu} the definition 
and basic properties of
the filtration by ramification groups
of the absolute Galois group
of a local field with not necessarily perfect
residue field.
Let $K$ be a complete discrete 
valuation field
and ${\cal O}_K$
be the valuation ring.
We fix a separable closure $\bar K$
of $K$. Its residue field $\bar F$
is an algebraic closure of 
the residue field $F$.
Let $L$ be a finite \'etale algebra
over $K$
and let $T$ denote the 
normalization of 
$S={\rm Spec}\ {\cal O}_K$
in $L$.
We consider a cartesian diagram
\begin{equation}
\begin{CD}
Q@<<< T\\
@VV{\quad\ \Box }V @VVV\\
P@<<< S
\end{CD}\label{eqPQST}
\end{equation}
of schemes over $S$
satisfying the condition:
\begin{itemize}
\item[(P)]
The horizontal
arrows are closed immersions,
the vertical arrows
are finite flat
and $P$ and $Q$
are smooth over $S$.
\end{itemize}

Let $r>0$ be a rational number.
Let $K'$ be a finite separable extension of $K$
of ramification index $e$
contained in $\bar K$ such that
$er$ is an integer.
Let $S'$ denote the
normalization of $S$ in $K'$
and $F'$ be the residue field of $K'$.
Let $D\subset P$
and $D'=P\times_S{\rm Spec}\ F'
\subset P_{S'}=P\times_SS'$
denote the closed fibers
and 
define the dilatation
$P_{S'}^{(er)}$ to be 
$P_{S'}^{(erD'\cdot S')}$.
Let $Q_{S'}^{(er)}$
denote the normalization
of the base change
$Q\times_PP_{S'}^{(er)}$.
As a consequence of
Epp's theorem, 
there exists a finite separable
extension $K'$ such 
that the geometric closed fiber
$Q_{S'}^{(er)}\times_{S'}{\rm Spec}\ \bar F$
of the normalization is reduced.
Further the finite morphism
\begin{equation}
Q_{\bar F}^{(r)}
=
Q_{S'}^{(er)}
\times_{S'}{\rm Spec}\ \bar F
\to
P_{\bar F}^{(r)}
=
P_{S'}^{(er)}
\times_{S'}{\rm Spec}\ \bar F
\label{eqPQr}
\end{equation}
on the geometric closed fibers
is independent of $K'$.

\begin{df}{\rm (\cite[Definition 1.8]{mu})}\label{dfbddL}
Let $K$ be a complete discrete valuation
field, $L$ be a finite \'etale $K$-algebra
and $r>0$ be a rational number.
Then, we say that the ramification of $L$ over $K$
is  {\rm bounded by } $r+$
(resp.\ $r$)
if there exist a 
cartesian diagram {\rm (\ref{eqPQST})}
satisfying the condition {\rm (P)}
above and a finite separable
extension $K'$ over $K$ such that
$Q_{\bar F}^{(r)}
\to P_{\bar F}^{(r)}$ {\rm (\ref{eqPQr})} is 
a finite \'etale covering
(resp.\ 
a totally decomposed 
finite \'etale covering).
\end{df}

We recall main results
from \cite{AS1} and \cite{AS2}.
The condition in Definition \ref{dfbddL}
holds for one cartesian
diagram {\rm (\ref{eqPQST})}
satisfying the condition {\rm (P)}
for an extension $K'$
if and only if it holds
for every such diagram
for every finite separable extension over $K'$.
The full subcategory of
the category of finite \'etale
$K$-algebras consisting
of those such that the ramification
is bounded by $r$ (resp.\ by $r+$)
form a Galois subcategory
and its fundamental group
is the quotient 
$G_K/G_K^r$
(resp.\ $G_K/G_K^{r+}$)
by the ramification group
$G_K^r$
(resp.\ $G_K^{r+}=\overline{
\bigcup_{s>r}G_K^s}$)
of the absolute Galois group
$G_K={\rm Gal}(\bar K/K)$.
The subgroup
$G_K^1$ is the inertia subgroup
$I_K={\rm Gal}(K^{ur}/K)$
and
$G_K^{1+}$ is its $p$-Sylow subgroup $P_K$.

The diagram
(\ref{eqPQST})
induces a commutative diagram
\begin{equation}
\begin{CD}
Q_{S'}^{(er)}
@<<<
\bar T_{S'}\\
@VVV@VVV\\
P_{S'}^{(er)}
@<<<S'
\end{CD}
\label{eqPQS'}
\end{equation}
where the normalization
$\bar T_{S'}$
of $T\times_SS'$
is the disjoint union of
copies of $S'$ indexed by
embeddings
$L\to \bar K$ over $K$
if $K'$ contains the Galois closure of $L$
over $K$.
We call the closed point of
$P^{(r)}_{\bar F}$ defined by
the bottom horizontal arrow
of  (\ref{eqPQS'}) the origin
and let it be denoted by $0$.
Then, if the ramification of
$L$ is bounded by $r+$,
the diagram (\ref{eqPQS'}) induces
a cartesian diagram
\begin{equation}
\begin{CD}
Q^{(r)}_{\bar F}
@<<<
{\rm Mor}_K(L,\bar K)\\
@VV{\quad\quad\ \Box }V@VVV\\
P^{(r)}_{\bar F}
@<<<0.
\end{CD}
\label{eqPQFbar}
\end{equation}

Assume $L$ is a Galois extension of
$K$ such that the ramification
is bounded by $r+$.
We fix an embedding $L\to \bar K$
and identify the Galois group
$G={\rm Gal}(L/K)$ with
the set ${\rm Mor}_K(L,\bar K)$
of embeddings.
Let $G^{(r)}\subset G$
denote the ramification subgroup
defined as the image of
$G^r_K\subset G_K$ by
the surjection $G_K\to G={\rm Gal}(L/K)$.
Then, $G^{(r)}\subset G$
is identified with the intersection
$Q^{(r)0}_{\bar F}\cap G$
with the connected component
$Q^{(r)0}_{\bar F}$
containing the image of
the fixed embedding
$1\in {\rm Gal}(L/K)=
{\rm Mor}_K(L,\bar K)$.
Further $Q^{(r)0}_{\bar F}$
is a $G^{(r)}$-torsor
over $P^{(r)}_{\bar F}$ compatible
with its canonical action on the
subset $G^{(r)}$.
Thus, we obtain a canonical surjection
\begin{equation}\pi_1(P_{\bar F}^{(r)},0)
\to G^{(r)}.
\label{eqpiGr}
\end{equation}

\begin{pr}\label{prequiv}
Let $X$ be a smooth scheme over 
a perfect field $k$ of characteristic $p>0$
and $D$ be a smooth irreducible divisor of $X$.
Let $\xi$ be the generic point 
of $D$ and
$K$ be the fraction field
of the completion $\hat {\cal O}_{X,\xi}$
of the discrete valuation ring.
Let $G$ be a finite group and
$V\to U$ be a $G$-torsor.
Then for the finite \'etale $K$-algebra
$L=\Gamma(V\times_UK,{\cal O})$
and a rational number $r> 1$,
the following conditions are equivalent:

{\rm (1)}
The ramification of $V$ over $U$ at $\xi$
is bounded by $R+=rD+$.

{\rm (2)}
The ramification of $L$ over $K$
is bounded by $r+$.
\end{pr}

\begin{proof}
Set $S={\rm Spec}\
{\cal O}_K$
and let $Y$ be the normalization
of $X$ in $V$.
Since $T={\rm Spec}\
{\cal O}_L
\to Y\times_XS$ is an isomorphism,
the cartesian diagram 
$$\begin{CD}
Q&\ =Y\times_kS@<<< T\\
@VV{\quad\quad\quad \Box }V&@VVV\\
P&\ =X\times_kS@<<< S
\end{CD}$$
satisfies the condition (P) on the diagram
(\ref{eqPQST}).
Let $K'$ be a finite separable extension
of $K$
of ramification index
$e_{K'/K}=e$ such that
$er$ is an integer,
that $K'$ contains $L$
as a subfield and that
the geometric closed fiber
of $Q^{(er)}_{S'}$ is reduced.

Set $M=eD$.
In order to link
Definitions \ref{dfbR} and  \ref{dfbddL},
we construct a diagram
\begin{equation}
\xymatrix{
Q_1^{(R,M)}\ar[d]
&
Q^{(e,R)}_{S'}
\ar[d]\ar[l]\ar[r]
&
Q^{(er)}_{S'}\ar[d]\\
P_1^{(R,M)}\ar[d]
&
P^{(e,R)}_{S'}
\ar[d]\ar[l]\ar[r]
&
P^{(er)}_{S'}\ar[ld]\\
X&S'\ar[l]
}
\label{eqXS'}
\end{equation}
such that
the base change
by $U\to X$ is
\begin{equation}
\xymatrix{
(V\times_k V)/\Delta G
\times_k{\mathbf G}_m
\ar[d]
&
V\times_k{\rm Spec}\ K'\times_k{\mathbf G}_m
\ar[d]\ar[l]\ar[r]
&
Y\times_kK'
\ar[d]\\
U\times_kU\times_k{\mathbf G}_m
\ar[d]_{{\rm pr}_2}
&
U\times_k{\rm Spec}\ K'\times_k{\mathbf G}_m
\ar[d]_{{\rm pr}_2}\ar[l]\ar[r]
&
X\times_kK'
\ar[ld]\\
U&{\rm Spec}\ K'\ar[l]
}
\label{eqUK'}
\end{equation}
as follows.
Set $P_{S'}=P\times_SS'=X\times_kS'$
and let $D'=P_{S'}\times_{S'}{\rm Spec}\ F'$
be the closed fiber.
The scheme $P_{S'}^{(e)}$ for $r=1$
is the dilatation
$P_{S'}^{(eD'\cdot S')}$
with respect to
the Cartier divisor $eD'$
and the transpose $S'\subset P_{S'}=
X\times_kS'$ of the graph
of the composition $S'\to S\to X$.
Define $P^{(D)}_{S'}
\subset P_{S'}^{(e)}=
P_{S'}^{(eD'\cdot S')}$
to be the complement of the proper transform of
the inverse image of $D$
by the projection $P_{S'}\to X$.
The canonical map
$P^{(D)}_{S'}\to S'$
is smooth and
we have $P_1^{(D)}\gets 
P_1^{(D)}\times_XS
\gets
P^{(D)}_{S'}$
by the functoriality of dilatation.
Let $D^{(D)}_{S'}
=P^{(D)}_{S'}\times_{S'}
{\rm Spec}\ F'$
denote the closed fiber.

Similarly to the definition
of $P_1^{(D,M)}$
and $P_1^{(R,M)}$,
we define
$P^{(e,D)}_{S'}$ and
$P^{(e,R)}_{S'}$ as follows.
We consider the dilatation
of $P^{(D)}_{S'}
\times{\mathbf A}^1$
with respect to 
$e$-times the 0-section
regarded as a Cartier divisor
and the closed subscheme
$D^{(D)}_{S'}
\times{\mathbf A}^1$.
We define 
$P^{(e,D)}_{S'}$ by further
removing the proper transform of
$D^{(D)}_{S'}
\times{\mathbf A}^1$.
The scheme
$P^{(e,D)}_{S'}$ is a ${\mathbf G}_m$-bundle
over $P^{(D)}_{S'}$ and
hence the canonical map
$P^{(e,D)}_{S'}\to S'$ is smooth.
Let $T^{(e,D)}_{F'}$ denote the closed fiber
and let $\widetilde S'=
S'\times_{P^{(D)}_{S'}}
P^{(e,D)}_{S'}$ be the ${\mathbf G}_m$-bundle
over $S'$.
Then, $\widetilde S'$
is an open subscheme of the dilatation
of $(S'\times {\mathbf A}^1)^{(e(0)\cdot 
(D'\times {\mathbf A}^1))}$
obtained by removing the proper
transform of
$D'\times {\mathbf A}^1$.
Regard it as a closed subscheme
of $P^{(e,D)}_{S'}$.

Let $P^{(e,R)}_{S'}$
be the dilatation of
$P^{(e,D)}_{S'}$
with respect to the
Cartier divisor
$e(r-1)T^{(e,D)}_{F'}$
and the closed subscheme
$\widetilde S'\subset
P^{(e,D)}_{S'}$.
By the functoriality of
the dilatation
and by the canonical
isomorphism $(P^{(e)}_{S'})^{(e(r-1)D'\cdot S')}
\to
P^{(er)}_{S'}$,
we obtain the lower part of the commutative diagram
(\ref{eqXS'}).

By the assumption that
$K'$ contains $L$ as a subfield,
we obtain the upper left square of
the diagram (\ref{eqUK'}) 
as a cartesian square.
By taking the normalizations
of the schemes on the middle line
of (\ref{eqXS'})
in those on the top line of (\ref{eqUK'}),
we complete the construction
of the commutative diagram
(\ref{eqXS'}).
The scheme
$P^{(e,R)}_{S'}$ is a ${\mathbf G}_m$-bundle
over the open subscheme
$P^{(er)}_{S'}
\times_{P^{(e)}_{S'}}
P^{(D)}_{S'}$ of $P^{(er)}_{S'}$ 
and hence the canonical map
$P^{(e,R)}_{S'}\to 
P^{(er)}_{S'}$ is smooth.
The closed immersion
$\widetilde S'\to P^{(e,D)}_{S'}$
is canonically lifted to a closed immersion
$\widetilde S'\to P^{(e,R)}_{S'}$.

We show that the conditions (1)
and (2) are equivalent.
We consider the diagram
\begin{equation}
\begin{CD}
\widetilde X^{(M)}=\ &
Q_0^{(R,M)}
@<<<
\widetilde S'
@>>>
S'\\
&@VVV@VVV@VVV\\
&Q_1^{(R,M)}
@<<<
Q^{(e,R)}_{S'}
@>>>
Q^{(er)}_{S'}\\
&@VVV@VVV@VVV\\
&P_1^{(R,M)}
@<<<
P^{(e,R)}_{S'}
@>>>
P^{(er)}_{S'}
\end{CD}
\label{eqPS'}
\end{equation}
where the lower half is the same as the upper half of
(\ref{eqXS'}) and the upper vertical arrows
are the canonical sections.
The condition (1) is equivalent
to that the lower left vertical arrow in (\ref{eqPS'})
is \'etale on a neighborhood
of the image of the inverse image of 
the generic point $\xi$ of $D$
by $\widetilde X^{(M)}\to X$.
This is equivalent to the same condition
for the base change of the left column of (\ref{eqPS'})
by $X\gets S$.
The condition (2) is equivalent
to the condition
that the lower right vertical arrow of (\ref{eqPS'})
is \'etale on a neighborhood
of the image of the closed point $\xi'$
of $S'$.
Since the bottom right horizontal arrow
$P^{(e,R)}_{S'}\to 
P^{(er)}_{S'}$ is smooth,
the right half of the diagram
(\ref{eqPS'}) is cartesian.
Hence 
the condition (2)
is equivalent to that
the lower middle vertical arrow
in (\ref{eqPS'})
is \'etale on a neighborhood
of the image 
of the inverse image in $\widetilde S'$ of $\xi'$.
Since the bottom horizontal
arrow in the diagram
$$\begin{CD}
Q_1^{(R,M)}
\times_XS
@<<<
Q_{S'}^{(e,R)}\\
@VVV@VVV\\
P_1^{(R,M)}
\times_XS
@<<<
P^{(e,R)}_{S'}
\end{CD}
$$
is finite, it suffices to apply
Lemma \ref{lmet} to this diagram.
\end{proof}

\begin{cor}\label{corGr}
Let $K$ be a complete discrete valuation
field of characteristic $p>0$.
Let $L$ be a finite Galois extension of $K$
and $G={\rm Gal}(L/K)$ be
the Galois group.
Let $r>1$ be a rational number
such that
the ramification of $L$ over $K$
is bounded by $r+$
and define an
$\bar F$-vector space $N^{(r)}
={\mathfrak m}_{\bar K}^r/
{\mathfrak m}_{\bar K}^{r+}$
of dimension $1$ by
${\mathfrak m}_{\bar K}^r
=\{a\in \bar K\mid
{\rm ord}_K a\geqq  r\}$,
${\mathfrak m}_{\bar K}^{r+}
=\{a\in \bar K\mid
{\rm ord}_K a> r\}$.

{\rm 1.}
The last graded piece
$G^{(r)}$ is an abelian group
killed by $p$.

{\rm 2.}
Assume that the residue field
$F$ is finitely generated over
a perfect subfield $k$.
Then, there is a canonical injection
$${\rm Hom}(G^{(r)},{\mathbf F}_p)
\to
{\rm Hom}_{\bar F}(N^r,
\Omega^1_{({\cal O}_K/{\mathfrak m}_K^2)/k}
\otimes_F \bar F).$$
\end{cor}

\begin{proof}
1.
By a standard limit argument
as in the proof of \cite[Theorem 2.15]{AS2},
we may assume that the residue field
$F$ is finitely generated over
a perfect subfield $k$.
There exist a $G$-torsor
$V\to U=X\sm D$
and an isomorphism
$\hat {\cal O}_{X,\xi}\to {\cal O}_K$
as in Proposition \ref{prequiv}.
Then, by Proposition \ref{prequiv},
after replacing $X$ by a neighborhood of
$\xi$ if necessary,
the ramification of $V$ over $U$
along $D$ is bounded by $R+$.
Further, by the definition of
the ramification group recalled above,
we have an isomorphism
$\widetilde G^{(R)}_{\bar \xi}\to G^{(r)}$.
Hence the assertion
follows by Proposition \ref{prmain}.1.

2.
It suffices to take a uniformizer
and the stalk 
of the characteristic form (\ref{eqdfcf})
at the corresponding $\bar F$-point
of $\widetilde Z$.
\end{proof}

Let $X$ be a smooth scheme
over a perfect field $k$
and $D$ be a smooth irreducible
divisor of $X$.
Let $V\to U=X\sm D$ 
be a finite \'etale Galois covering
not tamely ramified along $D$.
Then by Proposition
\ref{prequiv},
after shrinking $D$,
there exists a rational number $r>1$
such that, for $R=rD$, 
the ramification
of $V$ over $U$ along $D$ is bounded by $R+$
and is non-degenerate at multiplicity $R$.
By Lemma \ref{lmcut},
for every closed point $x$,
there exists a curve $C\subset X$
meeting $D$ transversally
at $x$
such that the pull-back of
$V$ to $C$ is wildly ramified
at $x$.

\subsection{Logarithmic variant}\label{slog}

In \cite{mu} and \cite{Tohoku},
a logarithmic variant of the contents
of previous subsections is studied
by a different method.
We indicate how to treat
the logarithmic variant by the method
of this article and
how to translate the results there.

We keep the notation set at the beginning
of this section.
We consider a commutative
diagram of oversimplicial schemes 
\begin{equation}
\begin{CD}
T_{\bullet \log D}^{(R-D,M)}
@>{\subset}>>
P_{\bullet \log D}^{(R-D,M)}
@<{\supset}<<
U^{\bullet+1}\times_k{\mathbf G}_m^h\\
@VVV @VVV @|\\
T_{\bullet \log D}^{(M)}
@>{\subset}>>
P_{\bullet \log D}^{(M)}
@<{\supset}<<
U^{\bullet+1}\times_k{\mathbf G}_m^h\\
@VVV @VVV @VVV\\
D_{\bullet \log D}
@>{\subset}>>
P_{\bullet \log D}
@<{\supset}<<
U^{\bullet+1}\\
@. @VVV @|\\
&X^{\bullet+1}=&
P_\bullet
@<{\supset}<<
U^{\bullet+1}\\
\end{CD}
\label{eqtPRl}
\end{equation}
constructed as follows.
For integer $n\geqq  0$,
let $P'_{n\log D}$ be the blow-up
of $P_n=X^{n+1}$
at $D_1^{n+1},\ldots, D_h^{n+1}$ and
define $P_{n \log D}$
to be the complement of
the proper transforms of
the inverse image of $D$
by the $n+1$ projections.
By the universality of blow-up,
the schemes $P_{n \log D}$
form an oversimplicial scheme
$P_{\bullet \log D}$.
By replacing
$P_{\bullet}^{(D)}$
by $P_{\bullet \log D}$
in the construction of
(\ref{eqtPR}), we define
(\ref{eqtPRl}).
By the functoriality of
dilatation, we have canonical morphisms
$P_{\bullet \log D}\gets
P_{\bullet}^{(D)}$
and 
$P_{\bullet \log D}^{(R-D,M)}\gets
P_{\bullet}^{(R,M)}\gets
P_{\bullet \log D}^{(R,M)}$.

We define the condition that the logarithmic 
ramification of a $G$-torsor $V$
over $U$ is bounded by $(R-D)+$
as in Definition \ref{dfbR}
by replacing $P_1^{(R,M)}$
by $P_{1\log D}^{(R-D,M)}$.
If the ramification is bounded by $R+$,
the logarithmic ramification is bounded by $R+$.
If the logarithmic ramification is bounded by 
$(R-D)+$,
the ramification is bounded by $R+$.
The latter conditions are equivalent if
$X$ is a curve since the canonical map
$P_{\bullet \log D}\gets
P_{\bullet}^{(D)}$
is an isomorphism in this case.

We verify that the definition 
of bounding log ramification here
is equivalent to
\cite[Definition 7.3]{Tohoku}
as follows.
We take a smooth scheme
$X'$ over $k$ and a faithfully flat
morphism $f\colon X'\to X$ over $k$
such that the inverse image
$U'=f^{-1}(U)$ is the complement 
of a divisor $D'$ with simple normal crossings,
that $f\colon X'\to X$ 
is log smooth with respect to 
the log structures defined by $D=X\sm U$ and $D'$
and that $f^*R=R'$ has integral coefficient
as in \cite[Proposition 7.7]{Tohoku}.
As in the proof of Proposition \ref{prequiv},
we set $M=m_1D_1+\cdots+m_hD_h$ 
such that $f^*(m_1^{-1}D_1+\cdots+m_h^{-1}D_h)=D'$
and consider morphisms
\begin{equation}
\begin{CD}
P_{1\log D}^{(R,M)}
@<<<
P_{1\log D'}^{\prime (R',D')}
@>>>
P_{1\log D'}^{(R')}
\end{CD}
\label{eqlog}
\end{equation}
where $P_{1\log D'}^{(R')}$
is defined as $P_{1}^{(R')}$ without introducing
an auxiliary divisor $M'$.
The smoothness of the morphisms
(\ref{eqlog})
implies that the definition
using $P_{1\log D}^{(R,M)}$
is equivalent to the criterion
\cite[Proposition 7.7]{Tohoku}
using $P_{1\log D'}^{(R')}$.

If we assume a strong form of resolution
of singularities,
for a Galois covering $V\to U$,
we can conclude that there exists
$R$ such that the ramification
of $V\to U$ is bounded by $R+$.
More precisely, we consider the
following condition
on a smooth separated scheme $U$
of finite type over $k$:
\begin{itemize}
\item[(RS)]
If $X$ is a proper scheme over $k$
containing $U$ as a dense open subscheme,
there exist a proper smooth scheme
$X'$ over $k$ and
a morphism $f\colon X'\to X$
such that $U'=f^{-1}(U)\to U$
is an isomorphism and
$U'$ is the complement of a divisor
with simple normal crossings.
\end{itemize}

\begin{pr}\label{prNpS}
Let $U$ be a separated smooth scheme 
of finite type over $k$
satisfying the condition {\rm (RS)} above
and $V\to U$ be a finite \'etale Galois
covering.
Then, 
there exist a proper smooth scheme $X$ over $k$
containing $U$ as the complement of a divisor $D$
with simple normal crossings
and a linear combination
$R=r_1D_1+\cdots+r_hD_h$
of irreducible components of $D$
with rational coefficients
$r_i\geqq  1$
such that the ramification of $V$ over $U$
along $D$
is bounded by $R+$.
\end{pr}

\begin{proof}
By \cite[Proposition 7.22]{Tohoku},
there exist $X$ and $R'$
such that the log ramification of $V$ over $U$
along $D$
is bounded by $R'+$.
It suffices to set $R=R'+D$.
\end{proof}

\section{Characteristic cycles}

We keep the notation in the previous section.
Namely, $k$ denotes a perfect field
of characteristic $p>0$
and 
$X$ denotes a smooth separated scheme over $k$.
Let $D$ be a divisor of $X$ 
with simple normal crossings 
and $U=X\sm D$ be the complement.
Let $D_1,\ldots,D_h$ be
the irreducible components of $D$
and $R=r_1D_1+\cdots+r_hD_h$
be a linear combination
with rational coefficients
$r_i\geqq  1$ for every $i=1,\ldots, h$.
Let $M=m_1D_1+\cdots+m_hD_h$
be a linear combination
with integral coefficients
$m_i\geqq  1$ such that
$m_ir_i$ is an integer for every $i=1,\ldots, h$.
Let $Z\subset D$ be the
union of irreducible components $D_i$
such that $r_i>1$.

\subsection{Definition of characteristic cycles}

Let $\Lambda$
be a local ring
over ${\mathbf Z}[\frac 1p,\zeta_p]$
and ${\cal F}$
be a locally constant sheaf
of free $\Lambda$-modules
of finite rank
on $U=X\sm D$.
Set $d=\dim X$.
We will define
the characteristic cycle
of ${\cal F}$
as a cycle of dimension $d$
of the cotangent bundle $T^*X$
under a certain non-degenerate hypothesis.

We prepare some terminology about
representations of
the absolute Galois group $G_K$
of a local field $K$ as in \S\ref{sslf}.
Let ${\mathbf V}$ be a continuous
representation of
the profinite group $G_K$ 
on a discrete free $\Lambda$-module
of finite rank.
For a rational number $r> 1$,
we say that ${\mathbf V}$ is {\it isoclinic of slope} $r$
if $G_K^{r+}$ acts trivially on ${\mathbf V}$
and if the $G_K^{r}$-fixed part is $0$.
For $r=1$,
we say ${\mathbf V}$ is isoclinic of slope $1$
if 
the wild inertia group $P_K=G_K^{1+}$
acts trivially on ${\mathbf V}$
or, equivalently, if it is tamely ramified.
Since $P_K$ is a pro-$p$ group
and is of order prime to $\ell$,
there exists a unique decomposition
${\mathbf V}=\bigoplus_{r\geqq  1}{\mathbf V}^{(r)}$,
called the {\it slope decomposition}
such that ${\mathbf V}^{(r)}$ is isoclinic of slope $r$.

First, we consider the isoclinic case.
In this case, 
the characteristic class
will be defined by (\ref{eqchar})
using the image of the locally constant function
(\ref{eqrank}) by the map (\ref{eqch2}).
Recall that a locally constant sheaf ${\cal F}$ 
on $U$ of free $\Lambda$-modules
of finite rank
is trivialized on a finite \'etale Galois
covering and 
corresponds to a continuous
representation of the fundamental
group of $U$ if it is connected.

\begin{df}\label{dfic}
Let $\Lambda$ be a local ring over
${\mathbf Z}[\frac 1p]$
and ${\cal F}$ be a locally constant sheaf
on $U$ of free $\Lambda$-modules
of finite rank.
Let $D_1,\ldots, D_h$ be 
the irreducible components
of the divisor $D$ with simple normal
crossings,
let $\xi_i$ be the generic point of $D_i$
and
$\bar \eta_i$ be the geometric
point of $U$ defined by
a separable closure $\bar K_i$
of the local field
$K_i={\rm Frac}\ \hat{\cal O}_{X,\xi_i}$
for $i=1,\ldots,h$.

{\rm 1.}
We say that ${\cal F}$ is {\rm isoclinic
of slope} $R=r_1D_1+\cdots+
r_hD_h$
if the representation
${\cal F}_{\bar \eta_i}$ of
the absolute Galois group $G_{K_i}$
is isoclinic of slope $r_i$
for every $i=1,\ldots,h$.
We say that
the ramification 
of an isoclinic sheaf ${\cal F}$ 
of slope $R$ is
{\em non-degenerate} along $D$ if
there is a finite \'etale Galois torsor
$V\to U$ such that
the pull-back of ${\cal F}$
on $V$ is constant
and that
the ramification of
$V$ over $U$ along $D$
is bounded by $R+$
and is non-degenerate along $D$ at
multiplicity $R$.

{\rm 2.}
We say that ${\cal F}$ is {\rm polyisoclinic}
if there exists a finite number
of divisors $R_j\geqq  D$
with rational coefficients
and a decomposition
${\cal F}=\bigoplus_j{\cal F}_j$
such that
each ${\cal F}_j$ is isoclinic
of slope $R_j$.
We call such a decomposition
an {\rm isoclinic decomposition}.

We say that
the ramification
of a polyisoclinic sheaf ${\cal F}$ is
{\rm non-degenerate} along $D$  if
there exists
an isoclinic decomposition
${\cal F}=\bigoplus_j{\cal F}_j$
such that
the ramification 
of each isoclinic sheaf ${\cal F}_j$ is
{\rm non-degenerate} along $D$
at multiplicity $R_j$.

{\rm 3.}
We say that
the ramification 
of ${\cal F}$ is
{\rm non-degenerate} along $D$  if
it is \'etale locally polyisoclinic
and the ramification is
{\rm non-degenerate} along $D$.

{\rm 4.}
We say that ${\cal F}$ is {\rm totally wildly ramified} along $D$,
if for every irreducible component $D_i$ of $D$,
the slopes of the representations
of the absolute Galois group of the local field $K_i$
are strictly greater than $1$.
\end{df}

A sheaf ${\cal F}\neq0$ is tamely ramified
along $D$ if and only if
${\cal F}$ is isoclinic of slope $D$.
The non-degenerate condition
roughly says that the ramification along $D$ is uniformly controlled
at the generic points of irreducible
components of $D$.
The following Lemma implies that
the condition is satisfied if we remove
a sufficiently large closed subscheme
of codimension $\geqq 2$.

\begin{lm}\label{lmic}
Any locally constant sheaf ${\cal F}$
on $U$
is polyisoclinic
on an \'etale neighborhood 
of a geometric point
above the generic point 
of each irreducible component of $D$.
\end{lm}

\begin{proof}
Since the wild inertia subgroup
$P_{K_i}$ is a pro-$p$ group,
any representation
of $G_{K_i}$ admits a slope decomposition.
Since $G_{K_i}$ is the same
as the absolute Galois group
of the fraction field of the henselization
${\cal O}_{X,\xi_i}^h$ at the generic point
$\xi_i$ of $D_i$,
the assertion follows.
\end{proof}

Let $\Lambda$ be a local ring over
${\mathbf Z}[\frac 1p,\zeta_p]$
and ${\cal F}$ be a locally constant sheaf
on $U$ of free $\Lambda$-modules
of finite rank.
Let $V$ be a $G$-torsor over $U$
for a finite group $G$
such that the pull-back of ${\cal F}$
on $V$ is constant and
let ${\mathbf V}$ be a represention of $G$
on a free $\Lambda$-module
such that ${\cal F}$ is corresponding to ${\mathbf V}$.
Define a locally constant sheaf
${\cal H}$ on $U\times_kU$ by
$${\cal H}={\cal H}om({\rm pr}_2^*{\cal F},
{\rm pr}_1^*{\cal F}).$$
The locally constant sheaf ${\cal H}$
on $U\times_kU$
is trivialized by the $G\times G$-torsor
$V\times_kV$ over
$U\times_kU$ and corresponds
to the representation 
${\rm End}({\mathbf V})$ of $G\times G$.

We consider the diagram
\begin{equation}
\begin{CD}
E^{(R)0}\ &\subset
E^{(R)}@>{i_W}>>
W^{(R,M)}_1
@<{j_W}<<
V\times_kV/\Delta G
\times{\mathbf G}_m^h\\
@V{\pi}VV&@VVV @VVV\\
T^{(R)}
&@>{i_P}>>P_1^{(R,M)}@<{j_P}<< U\times 
U\times {\mathbf G}_m^h\\
@V{{\rm pr}_2}VV&@VVV @VVV\\
\widetilde Z&
@>i>>\widetilde X^{(M)}@<j<< 
U\times {\mathbf G}_m^h.
\end{CD}
\label{eqEWTP}
\end{equation}
By abuse of notation,
the pull-backs of ${\cal H}$ and ${\cal E}nd({\cal F})$
are also denoted by ${\cal H}$ and ${\cal E}nd({\cal F})$.
Since the middle vertical arrows are smooth,
this abuse of notation should not cause
problem by smooth base change theorem.

To describe the sheaf
$i_P^*j_{P*}{\cal H}$
on $T^{(R)}$,
we introduce some notation.
Let ${\cal L}_\psi$
be the locally constant sheaf of rank $1$
on $T^{(R)}\times_{\widetilde Z}T^{(R)\vee}$
defined by the Artin-Schreier covering
$t^p-t=\langle x,f\rangle$
and the injection $\psi\colon
{\mathbf F}_p
\to \Lambda^\times$
sending $1$ to $\zeta_p$.
Let
$\widetilde Z^{(p^{-n})}\to \widetilde Z$
be a radicial covering where
the characteristic form
$\widetilde G^{(R)\vee}\to T^{(R)\vee}$
is defined and let
${\cal L}={\cal L}_{\psi,\widetilde G^{(R)\vee}}$ be the pull-back by the characteristic form
on the base change of
$T^{(R)}\times_{\widetilde Z}\widetilde G^{(R)\vee}$
by
$\widetilde Z^{(p^{-n})}\to \widetilde Z$.

\begin{lm}\label{lmH}
Let ${\cal F}$ be a locally constant sheaf
on $U=X\sm D$
isoclinic of slope $R$ along $D$.
Let $V\to U$ be a $G$-torsor 
such that the pull-back of ${\cal F}$ on $V$
is constant
and that 
the ramification of $V$ over $U$
along $D$ is bounded by $R+$
and is non-degenerate along $D$ at multiplicity $R$.
Let ${\mathbf V}$ be a representation of $G$
corresponding to ${\cal F}$.

There exists an idempotent
$e_{\cal F}\in \Gamma(
T^{(R)}\times_{\widetilde Z}
\widetilde G^{(R)\vee},i_P^*j_{P*}{\cal E}nd({\cal F}))$
and canonical isomorphisms
\begin{align}
{\rm pr}_{1*}(
(e_{\cal F}\cdot i_P^*j_{P*}{\cal E}nd({\cal F}))
\otimes {\cal L})
&\to
i_P^*j_{P*}{\cal H},
\label{eqH}\\
{\rm pr}_{1*}(
(e_{\cal F}\cdot Ri_P^! j_{P*}{\cal E}nd({\cal F}))
\otimes {\cal L})
&\to
Ri_P^!j_{P*}{\cal H}
\nonumber
\end{align}
of $i_P^*j_{P*}{\cal E}nd({\cal F})$-modules
on $T^{(R)}$.
\end{lm}

\begin{proof}
Let 
$$\iota\in
\Gamma(V\times_kV/\Delta G,
{\cal H})$$ denote
the image of
the identity $1\in {\rm End}_G({\mathbf V})$
by the canonical map
\begin{equation}
{\rm End}_G({\mathbf V})
\to
 \Gamma(V\times_kV,
{\cal H}om({\rm pr}_2^*{\cal F},
{\rm pr}_1^*{\cal F}))^G
=
\Gamma(V\times_kV/\Delta G,
{\cal H}).
\label{eq1}
\end{equation}
The diagram (\ref{eqEWTP})
defines canonical maps
\begin{align}
\Gamma(V\times_kV/\Delta G,
{\cal H})
\to
\Gamma(V\times_kV/\Delta G
\times{\mathbf G}_m^h,
{\cal H})
& =
\Gamma(W^{(R,M)}_1,j_{W*}{\cal H})
\label{eqrank2}
\\
&
\to
\Gamma(E^{(R)0},i_W^*j_{W*}{\cal H}).
\nonumber
\end{align}
Since $W_1^{(R,M)}\to X$ is smooth, 
the section $\iota$ defines an isomorphism
\begin{equation}
j_{W*}{\rm pr}_2^*{\cal F}
\to
j_{W*}{\rm pr}_1^*{\cal F}
\label{eqp1p2}
\end{equation}
 on $W_1^{(R,M)}$ and
hence its image in
$\Gamma(E^{(R)0},i_W^*j_{W*}{\cal H})$
defines an isomorphism
\begin{equation}
i_W^*j_{W*}{\cal E}nd({\cal F})\to
i_W^*j_{W*}{\cal H}
\label{eqisoi}
\end{equation} of
$i_W^*j_{W*}{\cal E}nd({\cal F})$-modules
on the extension $E^{(R)0}$
of the vector bundle $T^{(R)}$
by the finite \'etale group scheme
$\widetilde G^{(R)}$.
We obtain the first isomorphism in
(\ref{eq1}) by applying Corollary \ref{corGG}
to the extension
$0\to \widetilde G^{(R)}\to E^{(R)0}\to T^{(R)}\to 0$
and to the
${\cal A}=i_P^*j_{P*}{\cal E}nd({\cal F})$-module
${\cal M}=i_P^*j_{P*}{\cal H}$.

Since $W_1^{(R,M)}\to X$ is smooth, 
the canonical morphism
$$Ri_P^!j_{P*}{\cal E}nd({\cal F})
\otimes_{i_P^*j_{P*}{\cal E}nd({\cal F})}
i_P^*j_{P*}{\cal H}
\to
Ri_P^!j_{P*}{\cal H}$$ on $T^{(R)}$ is an isomorphism,
similarly as the isomorphism (\ref{eqisoi}).
Thus the first isomorphism in (\ref{eq1})
induces the second one.
\end{proof}

Let $\bar x\to \widetilde Z$ be a geometric point
and $\chi
\in \widetilde G^{(R)\vee}_{\bar x}$
be the geometric point corresponding
to a character $\widetilde G^{(R)}_{\bar x}\to \Lambda^\times$.
Then, the stalk $e_\chi
\in (j_*{\cal E}nd({\cal F}))_{\chi}
\subset 
{\rm End}({\mathbf V})$
of $e_{\cal F}$ at $\chi$
is the projector to the
$\chi$-part
${\mathbf V}_\chi
=\{a\in {\mathbf V}\mid \sigma a
=\chi(\sigma)a
\text{ for } \sigma \in \widetilde G^{(R)}_{\bar x}\}$.
In an $\ell$-adic setting,
we may replace the construction
in the proof of Lemma \ref{lmH}
by the Fourier transform
as in \cite{mu}.

Let ${\cal I}dem({\cal F})
\subset {\cal E}nd({\cal F})$
denote the subsheaf
consisting of idempotents.
The morphism ${\rm rank}\colon
{\cal I}dem({\cal F})
\to {\mathbf Z}$
induces a map
${\rm rank}\colon i^*j_*{\cal I}dem({\cal F})
\to {\mathbf Z}$
and
\begin{equation}
\Gamma(
T^{(R)}\times_{\widetilde Z}
\widetilde G^{(R)\vee},
i^*j_*{\cal I}dem({\cal F)})
\to
\Gamma(T^{(R)}\times_{\widetilde Z}
\widetilde G^{(R)\vee},
{\mathbf Z})
\overset\cong\gets
\Gamma(\widetilde G^{(R)\vee},
{\mathbf Z})
\label{eqrankv}
\end{equation}
where the second arrow is an isomorphism
since $T^{(R)}$ is a vector bundle
over $\widetilde Z$.
As the image of the idempotent $e_{\cal F}$
by the map (\ref{eqrankv}),
we define a locally constant function
\begin{equation}
{\rm rk}_{\cal F}\colon
\widetilde G^{(R)\vee}
\to {\mathbf Z}.
\label{eqrank}
\end{equation}
We define the characteristic cycle
of an isoclinic sheaf
with non-degenerate ramification
along the boundary
out of the locally constant function
${\rm rk}_{\cal F}$ in
(\ref{eqrank}).

\begin{lm}\label{lmGm}
Let $S$ be a scheme and
$T$ be a ${\mathbf G}_m^h$-torsor
over $S$.
Let $X$ be a scheme over $S$
with a ${\mathbf G}_m^h$-action and $\pi\colon 
X\to T$ be a finite \'etale morphism
over $S$.
Let $m_1\geqq  1,\ldots,m_h\geqq  1$
be integers and
assume that the morphism
$\pi\colon X\to T$ 
and the map
${\mathbf G}_m^h\to {\mathbf G}_m^h$
sending $(t_i)$ to $(t_i^{m_i})$
are compatible with
the ${\mathbf G}_m^h$-actions.

Then, the quotient
$X/{\mathbf G}_m^h\to
T/{\mathbf G}_m^h= S$
is finite \'etale.
\end{lm}

\begin{proof}
By induction on $h$,
we may assume $h=1$.
Since the assertion is local on
$S$, we may assume 
$S={\rm Spec}\ A$
is affine,
$T={\rm Spec}\ A[t^{\pm 1}]$
is the trivial ${\mathbf G}_m$-torsor
and $X={\rm Spec}\ B$
is affine.
Then, the ${\mathbf G}_m$-action on $X$
defines a grading $B=\bigoplus_{
n\in {\mathbf Z}}B_n$
on a finite \'etale $A[t^{\pm 1}]$-algebra $B$.
The grading on $B$ is 
compatible with the natural one on
$A[t^{\pm 1}]$ multiplied by $m=m_1$.

We show that
the quotient
$X/{\mathbf G}_m=
{\rm Spec}\ B_0$
is finite \'etale over $A$.
Since $B$ is finite \'etale over $A[t^{\pm 1}]$,
its direct summand $B_0$ is
finite flat over $A$.
To show it is \'etale,
we may assume $A$ is an algebraically
closed field.
Then,
$B_0$ is an $A$-subalgebra
of the \'etale $A$-algebra
$B\otimes_{A[t^{\pm1}]}A
=B_0\oplus\cdots\oplus B_{m-1}$
defined by $t\mapsto 1$,
it is also \'etale over $A$.
\end{proof}

Let $\widetilde G^{(R)\vee} \sm \widetilde Z$
denote the complement of
the $0$-section.
By Lemma \ref{lmGm},
the quotient
$$PG^{(R)\vee}=
(\widetilde G^{(R)\vee} \sm \widetilde Z)/{\mathbf G}_m^h$$
is a finite \'etale scheme over 
$Z=\widetilde Z/{\mathbf G}_m^h$.
Then, since the injection
$\widetilde G^{(R)\vee}\to T^{(R)\vee}
=(T^*X\times_X\widetilde Z)(\widetilde R)$
defined over a radicial covering $F^n\colon 
\widetilde Z^{(p^{-n})}\to \widetilde Z$
is compatible with the 
${\mathbf G}_m^h$-actions,
it  induces a morphism
\begin{equation}
PG^{(R)\vee}=
(\widetilde G^{(R)\vee} \sm \widetilde Z)/{\mathbf G}_m^h
\to 
(T^{(R)\vee} \sm \widetilde Z)/{\mathbf G}_m^h
={\mathbf P}(T^{(R)\vee})
\to {\mathbf P}(T^*X)
\label{eqpcf}
\end{equation}
on the quotients
defined over a radicial covering $F^n
\colon Z^{(p^{-n})}\to Z$.
Since the fibers of the map
$\widetilde G^{(R)\vee} \sm \widetilde Z
\to
PG^{(R)\vee}$ are connected,
the pull-back
\begin{equation}
\Gamma(PG^{(R)\vee}
,{\mathbf Z})
\to
\Gamma(
\widetilde G^{(R)\vee}
\sm \widetilde Z,{\mathbf Z})
\label{eqlcf}
\end{equation}
is an isomorphism.

Let ${\mathbf L}(1)$ denote the tautological line bundle
on the projective space bundle
${\mathbf P}(T^*X)$.
We consider the commutative diagram
\begin{equation}
\begin{CD}
{\mathbf L}(1)\times_{{\mathbf P}(T^*X)}
PG^{(R)\vee}@>>>
T^*X\\
@VVV @VVV\\
PG^{(R)\vee}@>>> X
\end{CD}
\label{eqLTX}
\end{equation}
where the top vertical arrow
is the composition of the first projection
with the blow-up ${\mathbf L}(1)\to T^*X$ at
the $0$-section.
Then, the pull-back by the flat map
${\mathbf L}(1)\times_{{\mathbf P}(T^*X)}
PG^{(R)\vee} \to PG^{(R)\vee}$
and the push-forward by the proper map
${\mathbf L}(1)\times_{{\mathbf P}(T^*X)}
PG^{(R)\vee} \to
T^*X$
define maps
\begin{equation}
\Gamma(PG^{(R)\vee},
{\mathbf Z})
\to
\Gamma({\mathbf L}(1)\times_{{\mathbf P}(T^*X)}
PG^{(R)\vee} ,
{\mathbf Z})
\to Z_d(T^*X)
\label{eqcycl}
\end{equation}
to the free abelian group of $d$-dimensional cycles
of $T^*X$. Let
\begin{equation}
L\colon
\Gamma(\widetilde G^{(R)\vee}\sm \widetilde Z,
{\mathbf Z})
\to Z_d(T^*X)\otimes_{\mathbf Z}{\mathbf Q}=
Z_d(T^*X)_{\mathbf Q}
\label{eqch2}
\end{equation}
denote the composition of
the inverse of (\ref{eqlcf}) and (\ref{eqcycl})
divided by the degree $p^{n(d-1)}$
of the radicial covering $F^n\colon Z^{(p^{-n})}\to Z$
on which 
$\widetilde G^{(R)\vee}\to
T^{(R)\vee}$ is defined.
It is independent of the integer $n\geqq  0$
such that (\ref{eqpcf}) is 
defined over $F^n\colon Z^{(p^{-n})}\to Z$.
The morphism
$L\colon
\Gamma(\widetilde G^{(R)\vee}\sm \widetilde Z,
{\mathbf Z})
\to
Z_d(T^*X)_{\mathbf Q}$
is the sum of the compositions
of the restriction maps with
\begin{equation}
L_i\colon
\Gamma(\widetilde G^{(R)\vee}
\times_{\widetilde Z}\widetilde D_i
\sm \widetilde Z_i,
{\mathbf Z})
\to Z_d(T^*X)_{\mathbf Q}
\label{eqch2i}
\end{equation}
for $i=1,\ldots,h$ such that
$r_i>1$.
Recall that for a set $I
\subset\{1,\ldots,h\}$
of indices,
$D_I=\bigcap_{i\in I}D_i$
denotes the intersection
and $T^*_{D_I}X
\subset T^*X\times_X{D_I}
\subset T^*X$ the conormal bundle.

\begin{df}\label{dfccy}
Let ${\cal F}$ be
a locally constant constructible sheaf
of free $\Lambda$-modules
on $U$.
Assume that the ramification of
${\cal F}$ is
non-degenerate along $D$.
We define the characteristic cycle
${\rm Char}({\cal F})$ of ${\cal F}$
as an effective $d$-cycle with rational coefficients
on the cotangent bundle $T^*X$
and 
the total dimension divisor
$DT({\cal F})$ of ${\cal F}$
as a multiple of $D$
with rational coefficient
as follows.

{\rm 1.}
If ${\cal F}$ is isoclinic of slope $R
=r_1D_1+\cdots+r_hD_h$,  
define ${\rm Char}({\cal F})$ by 
\begin{equation}
{\rm Char}({\cal F})
=
(-1)^d\cdot
\left({\rm rank}\ {\cal F}
\cdot 
\sum_{I:r_i=1\text{ for } i\in I}[T^*_{D_I}X]
+
\sum_{r_i>1}
r_i\cdot L_i({\rm rk}_{\cal F})
\right).
\label{eqchar}
\end{equation}
If ${\cal F}$ is polyisoclinic
and if ${\cal F}=\bigoplus_j{\cal F}_j$
is an isoclinic decomposition,
define ${\rm Char}({\cal F})$
by 
$${\rm Char}({\cal F})
=\sum_j
{\rm Char}({\cal F}_j).$$
In general,
define 
${\rm Char}({\cal F})$
by \'etale descent.

{\rm 2.}
If ${\cal F}$ is isoclinic of slope $R$,  
define $DT({\cal F})$ by 
\begin{equation}
DT({\cal F})
=\sum_ir_i\cdot {\rm rank}\ {\cal F}\cdot D_i
\label{eqDT}
\end{equation}
If ${\cal F}$ is polyisoclinic
and if ${\cal F}=\bigoplus_j{\cal F}_j$
is an isoclinic decomposition,
define ${\rm Char}({\cal F})$
by 
$$
DT({\cal F})=\sum_j
DT({\cal F}_j).$$
In general,
define 
$DT({\cal F})$
by \'etale descent.
\end{df}

In Definition \ref{dfccy}.1,
for an irreducible component $D_i$ with
$r_i>0$,
the fiber of the last term $L_i({\rm rk}_{\cal F})$
in (\ref{eqchar})
at the generic point $\xi_i$ of $D_i$ is
given by
$\sum_\chi n(\chi)\cdot \dfrac{[L(\chi)]}{[F(\chi):F_i]}$
in the following notation.
Let ${\mathbf V}$ be the representation of
the absolute Galois group $G_{K_i}$
defined by ${\cal F}$
and let ${\mathbf V}=\bigoplus_\chi \chi^{\oplus n(\chi)}$
be the decomposition by characters
of the graded quotient ${\rm Gr}^{r_i}G_{K_i}
=G_{K_i}^{r_i}/G_{K_i}^{r_i+}$.
The characteristic form (\ref{eqdfcf})
induces a morphism
$({\rm Gr}^{r_i}G_{K_i})^\vee
={\rm Hom}({\rm Gr}^{r_i}G_{K_i},{\mathbf F}_p)
\to 
(\Omega^1_{X/k}\otimes_{{\cal O}_X}
{\cal O}_{\widetilde Z}(\widetilde R))_{\zeta_i}
\otimes_{{\cal O}_{{\widetilde Z},\zeta_i}}
\kappa(\zeta_i)$
where $\zeta_i$ denote the generic point
of the fiber of $\widetilde Z\to X$ at $\xi_i$.
The construction of (\ref{eqch2}) implies
that, for each non-trivial character $\chi
({\rm Gr}^{r_i}G_{K_i})^\vee$,
its image defines a line $L(\chi)$ in the fiber
of the cotangent bundle $T^*X$
at $\xi_i$ defined over a
finite extension $F(\chi)$ of $F_i=\kappa(\xi_i)$.

We expect that ${\rm Char}({\cal F})$
has integral coefficients (Conjecture
\ref{cnHA}).
We will show that $DT({\cal F})$
has integral coefficients in
Proposition \ref{prHA}.

If ${\cal F}$ is tamely ramified along $D$,
we have $R=D$ and
the characteristic cycle
${\rm Char}({\cal F})$
is Lagrangian \cite[Definition A.1.2]{KS}.
However, it is not Lagrangian in general
(see Example \ref{exArt}.2 below).
In general,
the total dimension divisor
$DT({\cal F})$ is much less precise than
the characteristic cycle
${\rm Char}({\cal F})$.
However for curves,
they are equivalent.

\begin{ex}\label{exArt}
{\rm
1.
Assume $\dim X=1$
and $D=\{x\}$.
Recall that, for
a representation $V$ of $G_K$,
the sum
$\dim V+{\rm Sw}\ V$
of the rank and the Swan conductor
is called the total dimension
and denoted
$\dim {\rm tot} V$.
We have
\begin{align}
&{\rm Char}({\cal F})
=-\Bigl({\rm rank}\ {\cal F}\cdot
[T^*_XX]+\dim {\rm tot}_x({\cal F})
\cdot [T_D^*X]\Bigr),
\label{eqdt}
\\
&DT({\cal F})
=
\dim {\rm tot}_x({\cal F})
\cdot D.
\nonumber
\end{align}
The minus sign in
the formula for
${\rm Char}({\cal F})$
comes from the fact that a curve
has odd dimension.
In particular,
${\rm Char}({\cal F})$
and $DT({\cal F})$
have integral coefficients
by the classical theorem of
Hasse-Arf \cite[Th\'eor\`eme 1, Chapitre VI]{CL}.

2.
Let $X={\mathbf A}^2_k=
{\rm Spec}\ k[x,y]$
and let $U$ be the complement
of the divisor $D$ defined by $x=0$.
Then by Example \ref{egAS}.1,
for the locally constant sheaf ${\cal F}$
of rank $1$ defined by
the Artin-Schreier covering
$t^p-t=\dfrac 1{x^n}$ where
$p\nmid n$,
we have
${\rm Char}({\cal F})
=
[T^*_XX]+
(n+1)
\cdot [T_D^*X]$
since $T_D^*X$ is spanned by
$dx$.

Similarly by Example \ref{egAS}.2,
for ${\cal F}$
defined by
$t^p-t=\dfrac y{x^n}$ where
$p\mid n$,
we have
${\rm Char}({\cal F})
=
[T^*_XX]+
n
\cdot [L]$
where $L$ is the sub line bundle
of the restriction
$T^*X\times_XD$
generated by the section $dy$
unless $p=n=2$.
If $p=n=2$,
we have
${\rm Char}({\cal F})
=
[T^*_XX]+
F_*[L]$
where $L$ is the sub line bundle
of the pull-back 
$T^*X\times_XD^{(2^{-1})}$
by $F\colon D^{(2^{-1})}
\to D\to X$
generated by the section $
\sqrt y dx+dy$.
}
\end{ex}

\subsection{Pull-back and local acyclicity}

We define a condition
for a morphism $f\colon X'\to X$
of smooth schemes
to be non-characteristic 
with respect to
a locally constant sheaf
on the complement
of a divisor on $X$.

\begin{df}\label{dfncF}
Let $X$ and $X'$ be schemes
smooth of dimension $d$ and $d'$ over $k$
respectively
and $U\subset X$ 
and $U'\subset X'$ be the complements
of divisors $D\subset X$ 
and $D'\subset X'$
with simple normal crossings.
Let $f\colon X'\to X$ be a morphism
over $k$ such that
$f^{-1}(U)=U'$.

{\rm 1.}
Let $L=\sum_il_iL_i\in 
Z_d(T^*X)_{\mathbf Q}$ be
a conic cycle with
rational coefficients $l_i>0$.
We say that
$f\colon X'\to X$
is {\rm non-characteristic}
with respect to $L$
if the intersection of
the inverse image of
$L_i$ by the projection
$T^*X\times_XX'\to T^*X$
with
${\rm Ker} 
(T^*X\times_XX'\to T^*X')$
is contained in
the $0$-section of $T^*X\times_XX'$ for every $i$.

If $f\colon X'\to X$
is {\rm non-characteristic}
with respect to $L$,
we define 
$f^*(L)\in  
Z_{d'}(T^*X')_{\mathbf Q}$
to be the linear combination
$\sum_il_i[f^*L_i]$
of the images of
$[L_i\times_XX']
\in  
Z_{d'}(T^*X\times_XX')$.

{\rm 2.}
{\rm (cf.\ \cite[Definition 5.4.12]{KS})}
Let ${\cal F}$
be a locally constant constructible sheaf
of $\Lambda$-modules
on $U$
such that
the ramification of
${\cal F}$ is
non-degenerate along $D$.
We say
$f\colon X'\to X$
is {\rm non-characteristic}
with respect to ${\cal F}$
if 
$f\colon X'\to X$
is {\rm non-characteristic}
with respect to 
${\rm Char}({\cal F})\in 
Z_d(T^*X)_{\mathbf Q}$.
\end{df}

We prove a compatibility
of the characteristic cycle
with the pull-back 
by a non-characteristic morphism $f\colon X\to Y$.

\begin{pr}\label{prncF}
Let $X$ be a smooth scheme over
$k$
and ${\cal F}$ be a locally constant
sheaf of free $\Lambda$-modules
of finite rank
on the complement $U=X\sm D$
of a divisor $D$ with normal crossings.
Let $f\colon X'\to X$ be a 
morphism
of smooth schemes over $k$
such that $U'=f^{-1}(U)$ 
is the complement 
of a divisor $D'\subset X'$ with normal crossings.

Assume that either of the following
conditions is satisfied.

{\rm (s)}
${\cal F}$ is totally wildly ramified
along $D$ and
the ramification of ${\cal F}$
along $D$ is non-degenerate.

{\rm (t)}
${\cal F}$ is tamely ramified
along $D$.

\noindent
Then, if $f\colon X'\to X$ is
non-characteristic 
with respect to ${\cal F}$,
the ramification
of the pull-back $f_U^*{\cal F}$
is non-degenerate along $D'$
and 
\begin{equation}
(-1)^{\dim X'}{\rm Char}(f_U^*{\cal F})
=(-1)^{\dim X}f^*{\rm Char}({\cal F}),\quad
DT(f_U^*{\cal F})
=f^*DT({\cal F}).
\label{eqChf*}
\end{equation}
\end{pr}

\begin{proof}
(s)
Since the assertion is \'etale local,
we may assume that
${\cal F}$ is isoclinic of
slope $R$.
Let $V\to U$ be a finite
\'etale Galois covering trivializing
${\cal F}$ such that
the ramification along $D$
is bounded by $R+$ and
is non-degenerate at multiplicity $R$.
Then, by Proposition \ref{prfun}.2,
the ramification along $D'$
of the \'etale Galois covering 
$V'=V\times_UU'\to U'$
is bounded by $R'+$ and
is non-degenerate at multiplicity $R'$
for $R'=f^*R$.
Further, the characteristic form 
${\rm Char}_{R'}(V'/U')$
is the pull-back of
${\rm Char}_{R}(V/U)$.
Thus, 
the assertion follows.

(t)
Let $I$ be a subset of $\{1,\ldots,h\}$.
The conormal bundle
$T^*_{D_I}X$ is locally spanned by
$dt_i$ for $i\in I$ such that
$t_i$ defines $D_i$.
Then by the condition that
$X'\to X$ is non-characteristic
with respect to $T^*_{D_I}X$,
it follows that
the inverse images
$D'_i=f^*D_i$ are smooth and
meet transversally each other.
Hence, the pull-back
$f^*(\sum_{r_i=1}D_i)$
is 
$\sum_{r_i=1}D'_i$
and the assertion on $DT({\cal F})$ follows.
Further, if we put
$D'_I=\bigcap_{i\in I}D'_i$,
then the pull-back of
$[T^*_{D_I}X]$
is $[T^*_{D'_I}X']$
and the assertion on
${\rm Char}({\cal F})$ also follows.
\end{proof}

\begin{cor}\label{corncF}
Let $X$ be a smooth scheme over
$k$
and ${\cal F}$ be a locally constant
sheaf of free $\Lambda$-modules
of finite rank
on the complement $U=X\sm D$
of a divisor $D$
with simple normal crossings.
Assume that one of the conditions
{\rm (s)} and {\rm (t)}
in Proposition {\rm \ref{prncF}}
is satisfied.

{\rm 1.}
Let $C$ be a smooth curve,
$f\colon C\to
X$ be a morphism over $k$
and $x$ be an isolated
point of $f^{-1}(D)$.
Then,
we have
\begin{equation}
\dim {\rm tot}_xf^*{\cal F}
\leqq (DT({\cal F}),C)_x
\label{eqcv}
\end{equation}
where the right hand side
denotes the intersection number at $x$.

{\rm 2.}
Let $\Sigma\subset TX$
be the union of the images of the hyperplane
bundles defined as zero-locus
of non-vanishing sections
of the characteristic cycle
${\rm Char}({\cal F})$.
Let $C$ be a smooth curve
and
$f\colon C\to X$ be a morphism over $k$.
For a point $x\in f^{-1}(D)$ of
the inverse image,
the following conditions are equivalent.

{\rm (1)}
$f\colon C\to X$
is non-characteristic on a neighborhood of $x$
with respect to ${\cal F}$.

{\rm (2)}
We have an equality
in {\rm (\ref{eqcv})}.

{\rm (3)}
The image of the map 
$f_*\colon T_xC\to T_{f(x)}X$
on the tangent space
is not contained in
$\Sigma$ defined above.

{\rm 3}.
Let $f\colon X'\to X$ be a 
morphism
of smooth schemes over $k$
such that $f^{-1}(U)$ 
is the complement of a divisor $D'$
with simple normal crossings.
Then, the following conditions are equivalent.

{\rm (1)}
$f\colon X'\to X$
is non-characteristic with respect to ${\cal F}$.

{\rm (2)}
For every closed point
$x'$ of $D'$,
there exists a curve $C$ on $X'$
meeting $D'$ transversally at $x'$
such that the composition 
$C\to X'\to X$ is non-characteristic
with respect to ${\cal F}$.
\end{cor}

\begin{proof}
By the second equality in (\ref{eqdt})
in Example \ref{exArt}.1,
the left hand side of (\ref{eqcv})
is $DT$ of the pull-back ${\cal F}|_C$
to $C$.

1.
Since the assertion is \'etale local,
we may assume that
${\cal F}$ is isoclinic of slope $R$.
Then, the inequality (\ref{eqcv})
follows from Proposition \ref{prfun}.1.

2.
The equivalence (1)$\Leftrightarrow$(3)
is clear from the definition
of non-charactericity.
If $f\colon C\to X$ is
non-characteristic
with respect to ${\cal F}$, 
then (\ref{eqcv})
is an equality by 
Proposition \ref{prncF}.
Hence, we have
(1)$\Rightarrow$(2).

We show (2)$\Rightarrow$(1)
or equivalently (2)$\Rightarrow$(3).
By the inequality (\ref{eqcv}),
the equality in (\ref{eqcv})
is equivalent to those for
all isoclinic components.
Hence we may assume ${\cal F}$
is isoclinic of slope $R$.
Let $V$ be a $G$-torsor
over $U$ for a finite group
$G$ such that
the ramification of $V$ is bounded
by $R+$
and that
${\cal F}$ is corresponding
to a faithful representation ${\mathbf V}$ of $G$.

If ${\cal F}$ is tamely ramified along $D$,
the equality (\ref{eqcv})
is equivalent to the condition that $C$
meets $D$ transversally.
Thus in this case (2) implies (1).

Assume that
the condition (s) in Proposition \ref{prncF}
is satisfied.
Then, since the ramification of
$V'=V\times_XC$ over 
$U'=U\times_XC$ is bounded by $R'+$,
the equality
$\dim {\rm tot}_x{\cal F}|_C
= (DT_D({\cal F}),C)_x$
is equivalent to that
every character $\chi$ of
$G^{(R)}$ such that
${\rm rk}_{\cal F}(\chi)\neq 0$
induces a non-trivial character
of $G^{\prime(R')}$.
This means that
a non-zero differential form
${\rm Char}_R(\chi)$ in
the fiber at $x$ of the line $L_\chi$ 
for such a character $\chi$
does not vanish on 
the tangent line $T_xC$
and the assertion follows.

3.
By 2.(3)$\Rightarrow$(1),
the condition (1) implies (2).
Conversely, by 2.(1)$\Rightarrow$(3),
the condition (2) implies (1).
\end{proof}

On the integrality of
the coefficients of
the characteristic cycle
and of the total dimension divisor,
we deduce the following from
the classical Hasse-Arf theorem.

\begin{pr}\label{prHA}
Assume that the ramification of
${\cal F}$ along $D$ is non-degenerate.

{\rm 1.}
The characteristic cycle
${\rm Char}({\cal F})$
has coefficients in ${\mathbf Z}[\frac 1p]$.

{\rm 2.}
The total dimension divisor
$DT({\cal F})$
has integral coefficients.
\end{pr}

For the total dimension divisor
$DT({\cal F})$,
the integrality is proved
by Xiao Liang \cite[Corollary 4.4.3]{Liang} by
using $p$-adic differential equations.

\begin{proof}
By the definition of characteristic cycle,
it suffices to prove the case where
the condition (s) in Proposition \ref{prncF}
is satisfied. Further 
we may assume that $D$ is irreducible.
Since the question is
local on an \'etale neighborhood of
the generic point $\xi$
of $D$,
we may assume that
${\cal F}$ is isoclinic
of slope $rD$
and $r>1$.
Further
we may assume that the support of
the characteristic cycle
consists of the 0-section
and the image of one line bundle
defined over the radicial covering
$F^n\colon D^{(p^{-n})}\to D$ for an integer $n\geqq 0 $.

1.
Let $m$ be the prime-to-$p$ part
of the denominator of $r$.
Then, the subgroup $\mu_m
\subset {\mathbf G}_m$
stabilizes the line bundle
and acts faithfully on
the characteristic form.
Hence ${\rm rank}\ {\cal F}$
is divisible by $m$
and the assertion follows.

2. It follows from Corollary \ref{corncF}.2.
\end{proof}

\begin{cn}\label{cnHA}
The characteristic cycle
${\rm Char}({\cal F})$
has integral coefficients.
\end{cn}

We will later state Conjecture
\ref{cnvan} stronger than 
Conjecture \ref{cnHA}.

Corollary \ref{corncF}
immediately implies
the following characterization
on the support of
the characteristic cycle
and the total dimension divisor.

\begin{pr}\label{prcv}
Keep the assumptions in 
Corollary {\rm \ref{corncF}}.

{\rm 1.}
There exists a unique linear combination
$L=\sum_{i=1}l_iD_i$
with integral coefficients $l_i\geqq  0$
satisfying the following properties 
{\rm (1)} and {\rm (2)}:
{\rm (1)} For every irreducible
curve in $X$ not contained in $D$
and every point $x\in D\cap C$,
we have
\begin{equation}
{\rm dim\ tot}_x{\cal F}|_C
\leqq (L,C)_x.
\label{eqcvm}
\end{equation}
{\rm (2)}
For each irreducible component 
$D_i$ of $D$,
there exists a 
smooth curve $C$ in $X$ meeting $D_i$
transversally and meeting no other
irreducible components of $D$
such that
we have an equality in 
{\rm (\ref{eqcvm})}.

The unique linear combination
$L$ above
equals the total dimension divisor
$DT({\cal F})$.

{\rm 2.}
There exists a closed
subset $\Sigma\subset TX$
characterized by the following
property:
For a point $x\in D$
and a   
smooth curve $C$ in $X$ meeting 
every irreducible component of $D$
transversally at $x\in C$,
an equality in 
{\rm (\ref{eqcv})} is equivalent to
that
the image of the map 
$f_*\colon T_xC\to T_{f(x)}X$
on the tangent space
is not contained in
$\Sigma$.

{\rm 3.}
The closed subset $\Sigma\subset TX$
in {\rm 2.}\
is the union of the following two parts:

{\rm (1)}
Finitely many hyperplane bundles 
$H_1,\ldots,H_l$ defined over finite \'etale
schemes of the radicial covering
$F^n\colon Z^{(p^{-n})}\to Z$ for an integer $n\geqq 0 $.

{\rm (2)}
The linear span of the union of
the tangent bundles
$TD_i\subset TX$ for
$r_i=1$.

The support of
the characteristic cycle
${\rm Char}({\cal F})$
is the union of the following two parts:

{\rm (1$'$)}
The union of the orthogonals $L_1,\ldots, L_l
\subset T^*X$
of  $H_1,\ldots,H_l\subset TX$.

{\rm (2$'$)}
The linear span of the union of
the conormal bundles
$T^*_{D_i}X\subset T^*X$ for
$r_i=1$.
\end{pr}

The assertion 2.\ implies
that for most smooth curve $C\subset X$
transversal to $D$ satisfies
the equality in (\ref{eqcvm}) in assertion 1.
The linear span in 3 (2$'$) arises from the
definition of the tame part of
the characteristic cycle in (\ref{eqchar}).

We introduce another condition
for a smooth morphism $f\colon X\to Y$
of smooth schemes
to be non-characteristic 
with respect to
a locally constant sheaf
on the complement
of a divisor on $X$.
In the rest of this section,
we further assume that $\Lambda$ is
a noetherian ring annihilated
by an integer invertible in $k$,
to consider nearby cycles.

\begin{df}\label{dfncf}
Let $X$ be a smooth scheme over $k$
and let ${\cal F}$ be
a locally constant sheaf of 
free $\Lambda$-modules of finite rank
on the complement $U=X\sm D$
of a divisor $D$ with simple normal crossings.
Assume that the ramification of
${\cal F}$ along $D$ is non-degenerate.
Let $f\colon X\to Y$ be a {\rm smooth} morphism of smooth schemes over $k$.

We say that $f$ is {\rm non-characteristic}
with respect to the ramification of
${\cal F}$ along $D$ if
the inverse image of
${\rm Char}({\cal F})$
by $T^*Y\times_YX\to T^*X$
is contained in
the $0$-section.
\end{df}

The proposition below partially answers
a question raised by Deligne \cite{bp}.
We recall the definition of local acyclicity
\cite[D\'efinition 2.12]{fini}.
Let $f\colon X\to S$ be a morphism
of schemes and ${\cal K}$ be
a complex of sheaves on the \'etale site of $X$.
Let $\bar x\to X$ be a geometric point of $X$
and let $\bar s\to S$
denote its image.
Then, we say that $f\colon X\to S$ 
is locally acyclic at $\bar x$ relatively to
${\cal K}$ if, for every 
geometric point $\bar t$
of the strict localization $S_{\bar s}$,
the canonical map
${\cal K}_{\bar x}\to R\Gamma(
X_{\bar x}\times_{S_{\bar s}}\bar t,
{\cal K})$
is a quasi-isomorphism.
We say that $f\colon X\to S$ 
is locally acyclic relatively to
${\cal K}$ if it is locally acyclic at every 
geometric point $\bar x\to X$.
It is universally locally acyclic relatively to
${\cal K}$ if every base change
is locally acyclic.

The following well-known fact is 
a consequence of the local acyclicity
of smooth morphism.

\begin{lm}\label{lmacy}
Let $f\colon X\to S$ be a smooth morphism
of schemes,
$D$ be a divisor of $X$
with simple normal crossings
relatively to $f\colon X\to S$
and $j\colon U\to X$ be the
open immersion of the complement
$U=X\sm D$.
Let ${\cal F}$ be a locally constant
constructible sheaf of $\Lambda$-modules
tamely ramified along $D$.
Then, $f\colon X\to S$ is universally locally acyclic
relatively to $j_!{\cal F}$
and to $Rj_*{\cal F}$.
\end{lm}

For the sake of convenience of the reader,
we sketch a proof,
similar to that of \cite[1.3.3 (i)]{app}.

\begin{proof}
Since the assertion
is local on $X$, we may assume that
there is a smooth map
$X\to {\mathbf A}^n_S$
defined by functions $t_1,\ldots,t_n$ on $X$
such that
$D$ is defined by $t_1\cdots t_n$.
Further, we may assume that
the pull-back of ${\cal F}$
to the inverse image $U'\subset X'$
of $U$ by the finite flat covering $\pi\colon
X'\to X$
defined by
$s_1^m=t_1,\ldots,s_n^m=t_n$
for a integer $m\geqq  1$
invertible on $S$ 
is constant.

The canonical map
${\cal F}\to \pi_*\pi^*{\cal F}$
is injective
and the cokernel is
locally constant and
tamely ramified along $D$.
Further the morphism
$X'\to {\mathbf A}^n_S$
defined by $s_1,\ldots,s_n$ is smooth
and hence $X'\to S$ is smooth
and $U'$ is the complement of a divisor $D'$
of $X'$ with simple normal crossings
relatively to $X'\to S$.
Thus ${\cal F}$ admits a
resolution by the direct images
of constant sheaves as above
and the assertion is reduced to the case
where ${\cal F}=\Lambda$
is constant.

Let $D_1,\ldots,D_n$
be the irreducible components
of $D$ and, 
for a subset $I\subset \{1,\ldots,n\}$,
set $D_I=\bigcap_{i\in I}D_i$.
Then, we have
an exact sequence
$0\to j_!\Lambda
\to \Lambda_X
\to 
\bigoplus_{i=1}^n\Lambda_{D_i}
\to
\cdots
\to \bigoplus_{|I|=q}\Lambda_{D_I}
\to
\cdots.$
By the relative purity,
we also have an isomorphism
$\bigoplus_{|I|=q}\Lambda_{D_I}(-q)
\to R^qj_*\Lambda$.
By devissage, it suffices to
show that $X\to S$ is locally acyclic
relatively to $\Lambda_{D_I}$.
Since the intersections
$D_I$ are smooth over $S$,
it follows from 
the local acyclicity of smooth morphism \cite{Artin}.
\end{proof}

\begin{pr}\label{pracy}
Let $X$ be a smooth scheme over $k$
and let ${\cal F}$ be
a locally constant sheaf of 
free $\Lambda$-modules of finite rank
on the complement $U=X\sm D$
of a divisor $D$ with simple normal crossings.
Let $j\colon U=X\sm D
\to X$ denote the open immersion
and let $f\colon X\to Y$ be a smooth morphism
of smooth schemes over $k$.
Assume that $f\colon X\to Y$ is 
{\em non-characteristic}
with respect to the ramification of
${\cal F}$ along $D$ and
that either of the following
conditions is satisfied.

{\rm (s)}
${\cal F}$ is totally wildly ramified
along $D$ and
the ramification of ${\cal F}$
along $D$ is non-degenerate.
Further 
the restriction $D\to Y$ of $f$ is {\rm flat}.

{\rm (t)}
${\cal F}$ is tamely ramified
along $D$.

\noindent
Then
$f\colon X\to Y$ is universally locally acyclic
relatively to $j_!{\cal F}$.
In the case {\rm (t)}, 
$f\colon X\to Y$ is universally locally acyclic
relatively to $Rj_*{\cal F}$.
\end{pr}

By the proper base change theorem,
if we further assume that $f$ is proper,
Proposition implies that
$R^qg_!{\cal F}$
and
$R^qg_*{\cal F}$
are locally constant for $g=f\circ j$
under the assumptions.

\begin{proof}
Since $f\colon X\to Y$ is assumed smooth,
the locally constant sheaf ${\cal F}$ on 
$U$ is universally locally acyclic by the local acyclicity
of smooth morphism \cite{Artin}.
Thus, it suffices to prove the
assertion for each point on $D$.

Assume (t) is satisfied.
Then, by the definition 
of the characteristic cycle and
of non-characteristic morphism,
the divisor $D$ has simple normal crossings
relatively to $f\colon X\to Y$.
Further, the sheaf ${\cal F}$ is
tamely ramified along $D$.
Hence, it follows from
Lemma \ref{lmacy} in this case.

Assume (s) is satisfied.
First we prove the case where
$f\colon X\to Y$ is of relative dimension 1.
Since $D$ is flat over $Y$, 
it is quasi-finite over $Y$.
For each closed point
$y$ of $Y$,
the fiber $X_y$ is a smooth curve
and the immersion $X_y\to X$ is
non-characteristic with respect to ${\cal F}$.
Let $x$ be a closed point $X_y\cap D$.
By Corollary \ref{corncF}.2,
the total dimension 
$\dim {\rm tot}_x({\cal F}|_{X_y})$
at $x$ of the restriction on the fiber
is equal to the intersection number
$(DT({\cal F}),X_y)_x$.
Thus the function $\varphi$ 
defined as the sum of total dimensions
in \cite{DL} is locally constant
after shrinking $X$ and $Y$ if necessary.
Since the condition that $\varphi$
is locally constant
implies the local acyclicity by
\cite[Th\'eor\`eme 2.1.1]{DL},
the assertion follows.

We prove the general case (s) by 
reducing it to the case of relative dimension 1.
It suffices to show that, for every closed
point $x$ of $D$,
there exists an open neighborhood of 
$x$ such that the restriction of
$f\colon X\to Y$ is universally locally
acyclic relatively to $j_!{\cal F}$.
Let $x\in D$ be a closed point of $D$
and $y\in Y$ be its image.
Since the assertion is \'etale local,
we may assume that $k$ is algebraically closed.
The fiber at $x$ of the support of
the characteristic cycle ${\rm Char}\ {\cal F}$
is a union of lines $L_j$.

By the assumption that
$f\colon X\to Y$ is non-characteristic,
their intersections with the image
of the injection
$T_y^*Y\to T_x^*X$ reduce to $0$.
Therefore, as $k$ is algebraically closed,
there exist functions
$t_1,\ldots,t_n$ defined on a neighborhood of $X$
such that
$dt_1,\ldots,dt_n$ form 
a basis of the cokernel
$T_x^*X/T_y^*Y$ and that 
the intersections of $L_j$
with 
$T_y^*Y\oplus 
\langle dt_1,\ldots,dt_{n-1}\rangle\subset T_x^*X$ reduce to $0$.
After shrinking $X$,
we define a morphism
$g\colon X\to P={\mathbf A}^{n-1}\times_kY$
by $t_1,\ldots,t_{n-1}$.
Then, $g\colon X\to P$ is smooth 
of relative dimension $1$ and
non-characteristic with respect to
the ramification of ${\cal F}$ along $D$.
By modifying $t_i$ if necessary
without changing $dt_i$,
we may assume that the
intersection of $D$
with the fiber $g^{-1}(g(x))$ is
reduced to the point $x$.
Hence, after shrinking $X$
if necessary, the morphism
$g\colon X\to P$ is
universally locally acyclic relatively
to $j_!{\cal F}$.
Since $P\to Y$ is smooth,
the assertion follows from
\cite[Corollaire 2.7]{app} and
the local acyclicity of smooth morphism.
\end{proof}

The following statement is
conjectured by Deligne
in \cite{bp}, in a more general setting.

\begin{cn}\label{cnvan}
Let $X$ be a smooth scheme
over $k$
and let ${\cal F}$ be
a locally constant sheaf of 
free $\Lambda$-modules of finite rank
on the complement $U=X\sm D$
of a divisor $D$ with simple normal crossings.
Assume that the ramification of
${\cal F}$ along $D$ is non-degenerate.
Let $j\colon U=X\sm D
\to X$ denote the open immersion.

Let $f\colon 
X\to {\mathbf A}^1$ be a flat morphism
satisfying either of the 
conditions {\rm (s)} and {\rm (t)}
in Proposition {\rm \ref{pracy}}
and let $f\in \Gamma(X,{\cal O}_X)$
also denote the function
defining the morphism $f\colon 
X\to {\mathbf A}^1$.
Assume that
the intersection of
${\rm Char}({\cal F})$
and the section 
$df\colon X\to T^*X$
is supported in
the fiber $T^*_xX$ of 
a closed point $x$ of $X$.
Then, 
the total dimension of
the space of vanishing cycles
is given by the intersection number:
\begin{equation}
-{\rm dim\ tot}_x
\phi_x(f,j_!{\cal F})
=
({\rm Char}({\cal F}),
df(X))_{T^*X}.
\end{equation}
\end{cn}

Conjecture \ref{cnvan} implies
the integrality of the
coefficients of ${\rm Char}({\cal F})$ 
(Conjecture \ref{cnHA}).
Proposition \ref{pracy} implies that
Conjecture \ref{cnvan} holds
if the intersection of
${\rm Char}(j_!{\cal F})$
and the section 
$df\colon X\to T^*X$
is empty.
If $x$ is in $U$,
it is a consequence of the Milnor formula
proved in \cite{SGA7}.
In the case where $X$ is a surface,
Conjecture \ref{cnvan} is proved
under a certain assumption
in \cite{La}.

\subsection{Characteristic cycle and characteristic class}

We study the relation of the characteristic cycle
and the characteristic cycle defined in \cite{cc}.
We briefly recall the definition of the characteristic class 
\cite[Definition 2.1.1]{cc} specialized to the following situation.
Let $k$ be a field
and $X$ be a smooth separated scheme
of finite type of dimension $d$ over $k$.
Let $\ell$ be a prime number
invertible in $k$.
We assume that a local ring $\Lambda$ 
is finite over ${\mathbf Z}_\ell[\zeta_p]$
or is a finite extension of ${\mathbf Q}_\ell(\zeta_p)$,
to consider cohomology.
Let $j\colon U\to X$ be an
open immersion
and ${\cal F}$ be 
a smooth sheaf of flat $\Lambda$-modules
of finite rank on $U$.

Let $j_1\colon U\times_kX\to X\times_kX$
and
$j_2\colon U\times_kU\to U\times_kX$
denote the open immersions.
Let ${\cal H}_0$
denote the smooth sheaf
${\cal H}om({\rm pr}_2^*{\cal F},
{\rm pr}_1^*{\cal F})$ on $U\times_kU$
and set
${\cal H}=j_{1!}Rj_{2*}{\cal H}_0$
on $X\times_kX$.
We consider the
commutative diagram
$$\begin{CD}
X@<j<<U\\
@V{\delta}VV @VV{\delta_U}V\\
X\times_kX
@<{j\times j}<<
U\times_kU
\end{CD}$$
where the vertical arrows
are the diagonal immersions and 
regard $X$ and $U$
as closed subschemes of
$X\times_kX$ and $U\times_kU$
respectively.
By the relative purity
and the canonical isomorphism
$\delta_U^*{\cal H}_0\to {\cal E}nd({\cal F})$,
we identify
$H^{2d}_U(U\times_kU,
{\cal H}_0(d))
=
H^0(U,
{\cal E}nd({\cal F}))$.
Then, it is shown in \cite[(1.22)]{cc} that
the restriction map
\begin{equation}
\begin{CD}
H^{2d}_X(X\times_kX,
{\cal H}(d))
@>>>
H^{2d}_U(U\times_kU,
{\cal H}_0(d))
=
H^0(U,
{\cal E}nd({\cal F}))
=
{\rm End}_U({\cal F})
\end{CD}
\label{eqH0}
\end{equation}
is an isomorphism.
The pull-back to the diagonal 
and the trace map define
\begin{equation}
\begin{CD}
H^{2d}_X(X\times_kX,
{\cal H}(d))
@>>>
H^{2d}(X,j_!{\cal E}nd({\cal F})(d))
@>>>
H^{2d}(X,\Lambda(d))
\end{CD}
\label{eqH1}
\end{equation}
since the isomorphism
$\delta_U^*{\cal H}_0\to {\cal E}nd({\cal F})$
induces an isomorphism
$\delta^*{\cal H}(d)
\to j_!{\cal E}nd({\cal F})(d)$.
The characteristic class
$$C(j_!{\cal F})\in
H^{2d}(X,\Lambda(d))$$
is defined as
the image by (\ref{eqH1})
of the inverse image of 
the identity $1_{\cal F}
\in {\rm End}_U({\cal F})$
by the isomorphism
(\ref{eqH0}).
If $X$ is proper over $k$,
the Lefschetz trace formula
implies that
the Euler number
$\chi_c(U_{\bar k},{\cal F})
=\sum_q(-1)^q
\dim H^q_c(U_{\bar k},{\cal F})$
is equal to the image of
the characteristic class
$C(j_!{\cal F})$
by the trace map
${\rm Tr}\colon
H^{2d}(X,\Lambda(d))
\to \Lambda.$

Assume that $k$ is perfect.
We compute the characteristic class
assuming that the ramification
is isoclinic and non-degenerate.
Let $U=X\sm D$ be the complement
of a divisor with simple normal crossings
and $R=r_1D_1+\cdots+r_hD_h$ 
and $M=m_1D_1+\cdots+m_hD_h$ be
as in the first paragraph of Section 2.
Recall that
in the case where $R$
has non integral coefficients,
the definition of
the characteristic cycle
involves the oversimplicial
scheme $P_\bullet^{(R,M)}$
that carries a canonical ${\mathbf G}_m^h$-action.
In order to treat the cohomology of the quotient
by the ${\mathbf G}_m^h$-action,
we will use the cohomology
of classifying space which
we recall briefly.
The canonical ${\mathbf G}_m^h$-equivariant map
$\widetilde X^{(M)}=P_0^{(R,M)}
\to X$ induces a morphism
$[\widetilde X^{(M)}/{\mathbf G}_m^h]
\to X$ of algebraic stacks.

The cohomology of
the quotient stack
$[\widetilde X^{(M)}/{\mathbf G}_m^h]$
(\cite[(12.10)]{LB})
is computed as that of
the simplicial scheme
$[\widetilde X^{(M)}/{\mathbf G}_m^h]_\bullet$
\cite[(6.1.2.1)]{Hdg3}
by \cite[Proposition (12.4.5)]{LB}.
The canonical map
$[\widetilde X^{(M)}/{\mathbf G}_m^h]
\to X$ of stacks
is interpreted
as an augmentation morphism
$[\widetilde X^{(M)}/{\mathbf G}_m^h]_\bullet
\to X$ to the constant simplicial scheme.

\begin{lm}\label{lmcs}
There exists a constant 
$c\geqq  1$ independent of
$M$ or $\Lambda$
such that the kernel and the cokernel
of the pull-back
\begin{equation}
H^{2d}(X, \Lambda(d))
\to
H^{2d}([\widetilde X^{(M)}/{\mathbf G}_m^h],
\Lambda(d))
\label{eqiso}
\end{equation}
are killed by $(m_1\cdots m_h)^c$.
\end{lm}

\begin{proof}
By the stratifications
$(D_I^\circ)$ of $X$
and
$(\widetilde D_I^{(M)\circ})$ of 
$\widetilde X^{(M)}$,
it suffices to show the assertion 
for 
$H^q(D_I^\circ, \Lambda(r))
\to
H^q([\widetilde D_I^{(M)\circ}/{\mathbf G}_m^h],
\Lambda(r))$ for $q\geqq 0$.
For $M=D$,
$\widetilde X^{(D)}$ is a
${\mathbf G}_m^h$-torsor over $X$
and the morphisms
$[\widetilde D_I^{(M)\circ}/{\mathbf G}_m^h]
\to D_I^\circ$ are isomorphisms.
Hence
(\ref{eqiso}) is an isomorphism
for $M=D$.

We show the general case.
It suffices to show the assertion for
the map on the cohomology defined by
$[\widetilde D_I^{(M)\circ}/{\mathbf G}_m^h]
\to
[\widetilde D_I^{(D)\circ}/{\mathbf G}_m^h]$.
The canonical map
$\widetilde D_I^{(M)\circ }\to
\widetilde D_I^{(D)\circ}$
fits in a cartesian diagram
$$\begin{CD}
\widetilde D_I^{(M)\circ }
@>>>
\prod_{i\notin I}{\mathbf G}_m\\
@VV{\quad\quad\quad \Box }V @VV{(t_i)\mapsto (t_i^{m_i})}V\\
\widetilde D_I^{(D)\circ}
@>>>
\prod_{i\notin I}{\mathbf G}_m
\end{CD}$$
and is compatible with
the map
${\mathbf G}_m^h\to
{\mathbf G}_m^h$
sending $(t_i)$ to $(t_i^{m_i})$
by Lemma \ref{lmPDq}.1.
By computing the cohomology using 
the simplicial schemes
$[\widetilde D_I^{(M)\circ }/{\mathbf G}_m^h]_\bullet$
and 
$[\widetilde D_I^{(D)\circ }/{\mathbf G}_m^h]_\bullet$,
it suffices to apply
the K\"unneth formula.
\end{proof}

We compute
the pull-back of the characteristic class
$C(j_!{\cal F})$
in $H^{2d}([\widetilde X^{(M)}/{\mathbf G}_m^h],\Lambda(d))$
using the Chern class of
the tangent bundle.

\begin{pr}\label{prcc}
{\rm (cf.\ \cite[Theorem 3.4]{mu})}
Assume that the ramification of
${\cal F}$ along $D$
is isoclinic of slope 
$R=r_1D_1+\cdots+r_hD_h$,
is bounded by $R+$
and is non-degenerate at multiplicity $R$.
Assume that ${\cal F}$ is totally wildly ramified
along $D$.
Then the pull-back of the characteristic class
$C(j_!{\cal F})$
in $H^{2d}([\widetilde X^{(M)}/{\mathbf G}_m^h],\Lambda(d))$
equals 
\begin{equation}
{\rm rank}({\cal F})\cdot
c_d\bigl((TX\times_X\widetilde X^{(M)})
(-\widetilde R^{(M)})\bigr).
\label{eqcc}
\end{equation}
If the denominators of the coefficients
in $R$ are invertible in $\Lambda$,
it is further equal to the pull-back of the cycle class
\begin{equation}
{\rm rank}({\cal F})\cdot
\bigl((X,X)_{T^*X}
+\bigl(c(\Omega^1_{X/k})(1-R)^{-1}[R]
\bigr)_{\dim 0}
\bigr).
\label{eqccc}
\end{equation}
\end{pr}

\begin{proof}
By the assumption on $R$,
we have $Z=D$ and 
$T_1^{(R,M)}$ is equal to the complement
$D_1^{(R,M)}=
P_1^{(R,M)}\sm
(U\times_kU\times_k{\mathbf G}_m^h)$.
We consider the cartesian diagram
\begin{equation}
\begin{CD}
[T_1^{(R,M)}/{\mathbf G}_m^h]
@>i>>
[P_1^{(R,M)}/{\mathbf G}_m^h]
@<{\tilde j}<<
[(U\times_kU\times_k{\mathbf G}_m^h)
/{\mathbf G}_m^h]\
&=&\
U\times_kU\\
@A0A{\quad\quad\quad \Box }A
@AA{\tilde \delta}A{ \Box }&&@AA{\delta_U}A\\
[\widetilde D^{(M)}/{\mathbf G}_m^h]
@>>>
[\widetilde X^{(M)}/{\mathbf G}_m^h]
@<{j^{(M)}}<<
[(U\times_k{\mathbf G}_m^h)
/{\mathbf G}_m^h]\
&=&
U
\end{CD}\label{eqcd}
\end{equation}
where the vertical arrows are 
regular closed immersions
and the right horizontal arrows
are the open immersions of
complements of the images of
the left horizontal arrows.
The cycle class of the middle vertical arrow
is defined as a cohomology class
\begin{equation}
[\ [\widetilde X^{(M)}/{\mathbf G}_m^h]\ ]
\in 
H^{2d}_{[\widetilde X^{(M)}/{\mathbf G}_m^h]}(
[P_1^{(R,M)}/{\mathbf G}_m^h],
\Lambda(d))
\label{eqcccl}
\end{equation}
by cohomological purity
\cite[Proposition 4.10.1]{LO}.

Set $\widetilde {\cal H}=
\tilde j_*{\cal H}_0$
on 
$[P_1^{(R,M)}/{\mathbf G}_m^h]$.
Since 
${\cal H}_0$ is isomorphic
to the pull-back of
${\cal E}nd({\cal F})$
on $(V\times_kV)/\Delta G$
and $W^{(R,M)}_1\to P_1^{(R,M)}$
(Theorem \ref{thmmain})
is \'etale,
the base change map
\begin{equation}
\tilde\delta^* \widetilde {\cal H}=
\tilde\delta^* \tilde j_*{\cal H}_0
\to 
j^{(M)}_*\delta_U^*{\cal H}_0
=j^{(M)}_*{\cal E}nd({\cal F})
\label{eqjMEnd}
\end{equation}
is an isomorphism.
Hence the restriction map
$$
H^0([\widetilde X^{(M)}/{\mathbf G}_m^h],
\tilde\delta^* \widetilde {\cal H})
\to 
H^0(U,{\cal E}nd({\cal F}))
={\rm End}_U({\cal F})$$
is an isomorphism
and the identity of ${\cal F}$
defines a section
\begin{equation}
1_{\cal F}
\in 
H^0([\widetilde X^{(M)}/{\mathbf G}_m^h],
\tilde\delta^*\widetilde {\cal H}).
\label{eq1F}
\end{equation}
The cup product of
(\ref{eqcccl}) and (\ref{eq1F})
defines a cohomology class
\begin{equation}
[\ [\widetilde X^{(M)}/{\mathbf G}_m^h]\ ]
\cdot
1_{\cal F}
\in 
H^{2d}_{[\widetilde X^{(M)}/{\mathbf G}_m^h]}(
[P_1^{(R,M)}/{\mathbf G}_m^h],
\widetilde {\cal H}(d)).
\label{eqcppt}
\end{equation}

Let $\varphi\colon
[P_1^{(R,M)}/{\mathbf G}_m^h]\to 
X\times X$ be the canonical map.
We show that the
pull-back
$$\begin{CD}
\varphi^*
\colon
H^{2d}_X(X\times X,{\cal H}(d))
@>>>
H^{2d}_{\varphi^{-1}(X)}
([P_1^{(R,M)}/{\mathbf G}_m^h],
\widetilde {\cal H}(d))
\end{CD}$$
of the identity
$1_{\cal F}
\in
H^{2d}_X(X\times X,{\cal H}(d))$
has the same image as
the cup product (\ref{eqcppt}).
The top line in (\ref{eqcd})
defines an exact sequence
\begin{equation*}
H^{2d}_{[T_1^{(R,M)}/{\mathbf G}_m^h]}
([P_1^{(R,M)}/{\mathbf G}_m^h],
\widetilde {\cal H}(d))
\to
H^{2d}_{\varphi^{-1}(X)}
([P_1^{(R,M)}/{\mathbf G}_m^h],
\widetilde {\cal H}(d))
\to
H^{2d}_U
(U\times U,{\cal H}_0(d))
\end{equation*}
since $\varphi^{-1}(X)\sm
[T_1^{(R,M)}/{\mathbf G}_m^h]=
[U\times {\mathbf G}_m^h
/{\mathbf G}_m^h]=
U\subset U\times U$.

Since both the identity
$1_{\cal F}
\in
H^{2d}_X(X\times X,{\cal H}(d))$
and the cup product (\ref{eqcppt}) have 
the identity $1_{\cal F}$
as their images in
$H^{2d}_U
(U\times U,{\cal H}_0(d))
={\rm End}_U({\cal F})$,
it suffices to show
$H^{2d}_{[T_1^{(R,M)}/{\mathbf G}_m^h]}
([P_1^{(R,M)}/{\mathbf G}_m^h],
\widetilde {\cal H}(d))
=0$.
Let $p\colon 
[T_1^{(R,M)}/{\mathbf G}_m^h]
\to [\widetilde D^{(M)}/{\mathbf G}_m^h]$
denote the canonical map.
Then, it is reduced to showing
\begin{equation}
Rp_*Ri^!
\widetilde {\cal H}=0.
\label{eqRp0}
\end{equation}

Since the assertion is \'etale local on
$[\widetilde D^{(M)}/{\mathbf G}_m^h]$,
we may assume that $\widetilde G^{(R)}$
is constant
and regard the idempotent $e_{\cal F}$
as a family of idempotent $(e_\chi)$
indexed by non-trivial characters
$\chi \in\widetilde G^{(R)\vee},\chi\neq 0$.
By the second isomorphism
in (\ref{eq1}) in Lemma \ref{lmH}, 
we have a canonical isomorphism
$$
\bigoplus_{\chi\neq 0} 
p^*Ri^!j_*e_\chi
{\cal E}nd({\cal F})
\otimes {\cal L}_\chi
\to
Ri^!
\widetilde {\cal H}.$$
Since
$Rp_*{\cal L}_\chi=0$
for non-trivial character $\chi\neq 0$
of $\widetilde G^{(R)}$,
it induces an isomorphism
$Rp_*Ri^!
\widetilde {\cal H}
\to
\bigoplus_{\chi\neq 0} Ri^!j_*
e_\chi
{\cal E}nd({\cal F})
\otimes Rp_*{\cal L}_\chi
=0$
and (\ref{eqRp0}) is proved.

Let
$\delta\colon
[\widetilde X^{(M)}/{\mathbf G}_m^h]
\to [P_1^{(R,M)}/{\mathbf G}_m^h]$
and
$\delta_U\colon
U=[U\times {\mathbf G}_m^h/{\mathbf G}_m^h]
\to 
U\times U=[U\times U
\times {\mathbf G}_m^h/{\mathbf G}_m^h]$
denote the diagonal morphisms.
By the isomorphism (\ref{eqjMEnd}),
we obtain a variant 
\begin{align}
H^{2d}_{\varphi^{-1}(X)}
([P_1^{(R,M)}/{\mathbf G}_m^h],
\widetilde {\cal H}(d))
\to&
H^{2d}([\widetilde X^{(M)}/{\mathbf G}_m^h],j^{(M)}_*{\cal E}nd({\cal F})(d))\label{eqH2}
\\
&\to
H^{2d}([\widetilde X^{(M)}/{\mathbf G}_m^h],\Lambda(d))
\nonumber\end{align}
of (\ref{eqH1}).
By the definition of the characteristic cycle
$C(j_!{\cal F})$ recalled above,
its image in $H^{2d}([\widetilde X^{(M)}/{\mathbf G}_m^h],\Lambda(d))$ is equal to the
image of the cup-product (\ref{eqcppt}) by (\ref{eqH2}).
The trace map
sends the identity to the rank
and the pull-back of the 
cycle class
$[\widetilde X^{(M)}/{\mathbf G}_m^h]$
by the lifting of
the diagonal map
is the self-intersection product
of
$[\widetilde X^{(M)}/{\mathbf G}_m^h]$
in 
$[P_1^{(R,M)}/{\mathbf G}_m^h]$
and is the top Chern class
$c_d((TX\times_X\widetilde X^{(M)})
(-\widetilde R^{(M)}))$
of the normal bundle.
Hence we obtain (\ref{eqcc}).

To deduce (\ref{eqccc}) from 
(\ref{eqcc}), it suffices to apply the formula
\begin{align*}
c_d(E\otimes L^{-1})
&=
(-1)^d\sum_{i=0}^d
c_i(E^\vee)c_1(L)^{d-i}
\\&=
(-1)^d\Bigl(c_d(E^\vee)
+\sum_{i=1}^d
c_i(E^\vee)(1-c_1(L))^{-1}
c_1(L)\Bigr)_{\deg d}
\end{align*}
for a vector bundle
$E$ of rank $d$,
the dual $E^\vee$
and a line bundle $L$.
\end{proof}

\begin{cor}\label{corcc}
Assume that the ramification of
${\cal F}$ along $D$
is isoclinic of slope $R=r_1D_1+\cdots+r_hD_h$,
is bounded by $R+$
and is non-degenerate at multiplicity $R$.
Assume that we have $r_i>1$
for every $i=1,\ldots,h$.
If the denominators of $r_i$ are 
invertible in $\Lambda$,
for the pull-back of the characteristic class,
we have an equality
\begin{equation}
C(j_!{\cal F})=
[{\rm Char}({\cal F})]
\label{eqchcc}
\end{equation}
in $H^{2d}([\widetilde X^{(M)}/{\mathbf G}_m^h],
\Lambda(d))$.
\end{cor}

\begin{proof}
By the definition of the characteristic cycle
(\ref{eqchar})
and by (\ref{eqccc}),
it suffices to show
\begin{equation}
(X,{\textstyle \sum_i}r_i\cdot L_i({\rm rk}_{\cal F}))_{T^*X}=
{\rm rank}({\cal F})
\bigl(c(\Omega^1_{X/k})(1-R)^{-1}[R]
\bigr)_{\dim 0}.
\label{eqLrkF}
\end{equation}
By the assumption that
the ramification along $D$
is isoclinic of slope $R$,
the push-forward of
${\rm rk}_{\cal F}\cdot [PG^\vee]$
by the finite \'etale morphism $PG^\vee\to D$
is ${\rm rank}\ {\cal F}\cdot [D]$.
Hence we have
$$
(X,{\textstyle \sum_i}r_i\cdot L_i({\rm rk}_{\cal F}))_{T^*X}=
\dfrac{{\rm rank}\ {\cal F}}{[PG^\vee:D]}
\cdot
(X,{\textstyle \sum_i}r_i\cdot 
{\mathbf L}(1)\times_{{\mathbf P}(T^*X)} PG^\vee
\times_DD_i)_{T^*X}.$$

By applying the excess intersection formula
to the cartesian diagram (\ref{eqLTX}),
we obtain
\begin{align*}
(X,{\mathbf L}(1)\times_{{\mathbf P}(T^*X)} PG^\vee
\times_DD_i)_{T^*X}
&=
\bigl(c(\Omega^1_{X/k})c({\mathbf L}(1))^{-1}[PG^\vee
\times_DD_i]
\bigr)_{\dim 0}\\
&
=[PG^\vee:D]\cdot
\bigl(c(\Omega^1_{X/k})c({\mathbf L}(1))^{-1}[D_i]
\bigr)_{\dim 0}
\end{align*}
for each $i$.
Since the pull-back of the line bundle
${\mathbf L}(1)\times_{{\mathbf P}(T^*X)} PG^\vee$
to $\widetilde G^\vee\sm \widetilde D$
is $L(-\widetilde R)$,
the assertion follows.
\end{proof}

In the case where 
${\cal F}$ is tamely ramified along $D$
namely $R=D$,
the equality (\ref{eqchcc})
is an immediate consequence of
\cite[Corollary 3.2]{mu}.
In fact, in this case,
$[{\rm Char}\ {\cal F}]$
is ${\rm rank}\ {\cal F}$-times
the pull-back of the class of
the $0$-section of the logarithmic
cotangent bundle $T^*X(\log D)$
by the canonical map $T^*X\to T^*X(\log D)$.

\begin{ex}\label{excc}
{\rm
Let $k$ be a perfect field of
characteristic $2$,
set
$X={\mathbf P}^2_k
\supset
U={\mathbf A}^2_k
={\rm Spec}\ k[x,y]$
and let $
D={\mathbf P}^1_k$
be the line at infinity.
Define a finite \'etale morphism
$V\to U$ by
the Artin-Schreier equation $t^2-t=xy$.
Then, the ramification of $V$ over $U$
is bounded by $2D+$
and non-degenerate along $D$
at multiplicity $R=2D$.

By Example {\rm \ref{egAS}},
the characteristic form
${\rm Char}_R(V/U)
\colon {\mathbf Z}/2{\mathbf Z}
\to F^*(T^*X(2D)\times_XD)$
sends $1$ to the section
\begin{equation}
d(xy)+\sqrt{\dfrac yx}dx
\label{eqdxy}
\end{equation}
of $\Omega^1_{X/k}(2D)$
defined on
the radicial double covering $F\colon D^{(2^{-1})}\to D$.
The characteristic cycle
of the locally constant sheaf 
${\cal L}$ on $U$
of rank $1$
corresponding to the injection
${\rm Gal}(V/U)={\mathbf F}_2
\to \{\pm 1\}\subset \Lambda^\times$
is the sum of
the $0$-section $T^*_XX$ of $T^*X$
and the image of a sub line bundle
$L\to F^*(T^*X\times_XD)$
defined by 
(\ref{eqdxy})
defined over 
the radicial double covering $F\colon D^{(2^{-1})}\to D$.
The composition 
$L\to F^*(T^*D\times_DD)$ with the 
the map induced by the canonical surjection
$T^*X\times_XD\to T^*D\times_DD$
is an isomorphism.

By Proposition {\rm \ref{prcc}},
the characteristic class $C(j_!{\cal L})$ is
$(X,X)_{T^*X}+
(K_X+2D)\cdot 2D$
where the canonical divisor
$K_X$ of the projective plane
$X={\mathbf P}^2_k$
is $-3D$.
Hence, the Euler number is
$\chi_c(U_{\bar k},{\cal L})
=\chi(X,\Lambda)-2=1$.
The normalization $Y$ of
$X$ in $V$ is a smooth quadric,
is isomorphic to
${\mathbf P}^1\times
{\mathbf P}^1$
and has the Euler number
$\chi(Y,\Lambda)
=\chi(X,\Lambda)
+
\chi_c(U_{\bar k},{\cal L})
=3+1=4$.
}
\end{ex}

\noindent
Department of Mathematical Sciences,\\
University of Tokyo, Tokyo 153-8914, Japan\\
t-saito@ms.u-tokyo.ac.jp

\end{document}